\documentclass[12pt]{amsart}

\usepackage{appendix}
\usepackage{amssymb}
\usepackage{amsrefs}
\usepackage{amsmath}
\usepackage{amsfonts}
\usepackage{mathtools}
\usepackage{xcolor}
\usepackage{fullpage}
\usepackage{mathrsfs}
\usepackage[shortlabels]{enumitem}
\usepackage{dsfont}
\usepackage[colorlinks]{hyperref}
\usepackage{cleveref}
\usepackage{tikz, pgfplots}
\usetikzlibrary{positioning}
\pgfplotsset{compat=1.18}

\numberwithin{equation}{section}

\newcommand{\lrfl}[1]{\left\lfloor#1\right\rfloor}
\newcommand{\lrceil}[1]{\left\lceil#1\right\rceil}

\newcommand{\lrang}[1]{\left\langle #1 \right\rangle}

\theoremstyle{plain}
\newtheorem{theorem}{Theorem}[section]
\newtheorem{proposition}[theorem]{Proposition}
\newtheorem{lemma}[theorem]{Lemma}

\theoremstyle{definition}
\newtheorem{definition}[theorem]{Definition}

\newtheorem{notation}[theorem]{Notation}

\theoremstyle{remark}
\newtheorem{remark}[theorem]{Remark}

\usepackage{subcaption}

\newcommand{\C}{\mathbb{C}}

\newcommand{\bM}{\mathbb{M}}
\newcommand{\N}{\mathbb{N}}

\newcommand{\Q}{\mathbb{Q}}
\newcommand{\R}{\mathbb{R}}

\newcommand{\bT}{\mathbb{T}}

\newcommand{\Z}{\mathbb{Z}}

\newcommand{\cF}{\mathcal{F}}

\newcommand{\cH}{\mathcal{H}}

\newcommand{\cK}{\mathcal{K}}

\newcommand{\cM}{\mathcal{M}}

\newcommand{\cO}{\mathcal{O}}

\newcommand{\cU}{\mathcal{U}}

\newcommand{\cW}{\mathcal{W}}
\newcommand{\cX}{\mathcal{X}}
\newcommand{\cY}{\mathcal{Y}}
\newcommand{\cZ}{\mathcal{Z}}

\newcommand{\w}{\mathrm{w}}
\newcommand{\z}{\mathrm{z}}

\DeclareMathOperator{\op}{Op}

\DeclareMathOperator{\supp}{ supp }

\DeclareMathOperator{\spn}{Span}

\DeclareMathOperator{\diam}{diam}

\renewcommand{\phi}{\varphi}

\newcommand{\rS}{\mathscr{S}}
\renewcommand{\phi}{\varphi}
\renewcommand{\pmod}[1]{\ensuremath{\,(\mathrm{mod}\, #1)}}
\DeclareMathOperator{\Sp}{Sp}

\DeclareMathOperator{\SL}{SL}

\DeclareMathOperator{\Tr}{Tr}
\DeclareMathOperator{\ortho}{O}
\newcommand{\Hj}{\cH_j}
\newcommand{\1}{\mathds{1}}
\usepackage{xcolor}
\usepackage{soul}

\makeatletter
\let\@wraptoccontribs\wraptoccontribs
\makeatother

\allowdisplaybreaks

\title{Characterizing the support of semiclassical measures for higher-dimensional cat maps}

\author{Elena Kim}
\address[Elena Kim]{Massachusetts Institute of Technology, Department of Mathematics, Cambridge, MA 02142, USA}
\email{elenakim@mit.edu}

\contrib[with an appendix by]{Theresa C. Anderson and Robert J. Lemke Oliver}
\address[Theresa C. Anderson]{Carnegie Mellon University, Department of Mathematical Sciences,
Pittsburgh, PA 15213, USA}
\email{tanders2@andrew.cmu.edu}

\address[Robert J. Lemke Oliver]{Tufts University, Department of Mathematics, Medford, MA 02155, USA  and  University of Wisconsin-Madison, Department of Mathematics, Madison, WI 53706, USA}
\email{robert.lemke\textunderscore oliver@tufts.edu}
\email{lemkeoliver@wisc.edu}

\begin{document}

\begin{abstract}
Quantum cat maps are toy models in quantum chaos associated to hyperbolic symplectic matrices $A\in \Sp(2n,\Z)$.
The macroscopic limits of sequences of eigenfunctions of a quantum cat map are characterized by
semiclassical measures on the torus $\R^{2n}/\Z^{2n}$. We show that if the characteristic polynomial of every
power $A^k$ is irreducible over the rationals, then every semiclassical measure has full support.
The proof uses an earlier strategy of Dyatlov--J\'ez\'equel~\cite{dyatlov2021semiclassical} and
the higher-dimensional fractal uncertainty principle of Cohen~\cite{Cohen}.
Our irreducibility condition is generically true, in fact we show that asymptotically for $100\%$ of matrices $A$,
the Galois group of the characteristic polynomial of $A$ is $S_2 \wr S_n$.

When the irreducibility condition does
not hold, we show that a semiclassical measure cannot be supported on a finite union of parallel non-coisotropic subtori.
On the other hand, we give examples of semiclassical measures supported on the union of two transversal symplectic subtori for $n=2$,
inspired by the work of Faure--Nonnenmacher--De Bi\`evre~\cite{FNdB} in the case $n=1$. This is complementary to the examples
by Kelmer~\cite{Kel} of semiclassical measures supported on a single coisotropic subtorus.
\end{abstract}

\maketitle

\section{Introduction}
One of the central topics in quantum chaos is the study of semiclassical measures,
which capture the high frequency limit of the mass of eigenfunctions.
A typical setting is given by Laplacian eigenfunctions on a Riemannian manifold with an Anosov geodesic flow.
In this paper, we work in a simpler setting of quantum cat maps on the $2n$-dimensional torus $\bT^{2n} \coloneqq \R^{2n}/\Z^{2n}$.
Here the Anosov geodesic flow is replaced by an automorphism of $\bT^{2n}$ given by a matrix $A \in \Sp(2n, \Z)$.
In addition to being technically simpler, quantum cat maps make it possible to construct counterexamples to Quantum Unique Ergodicity
and provide more opportunities for numerical experimentation. They have been extensively studied since the 1980s,
see~\S\ref{subsection:previousresults} below for a historical overview.

A quantum cat map is a family of operators $M_{N, \theta}$, with $N \in \N$, $\theta \in \bT^{2n}$, defined to satisfy an exact Egorov's Theorem. Each $M_{N, \theta}$ acts on the \emph{space of quantum states} $\cH_N(\theta)$, a finite-dimensional
Hilbert space.
We refer the reader to~\S\ref{section:preliminaries} for the proper definitions.

The aim of this paper is to describe the limit as $N \rightarrow \infty$ of the mass distribution of
eigenfunctions of $M_{N, \theta}$. This asymptotic behavior is captured by semiclassical measures, detailed in Definition \ref{def:semiclassical_measure}. We will in particular study the possible \emph{supports of semiclassical measures}.
In the related setting of Riemannian manifolds, results on these supports have found applications to control
estimates~\cite{Ji18}, exponential decay for the damped wave equation~\cite{Jin},
and bounds on restrictions of eigenfunctions~\cite{GZ21}.

\subsection{Main results}
Our first result examines when semiclassical measures must have full support.  
Suppose $A \in \Sp(2n, \Z)$ has a non-unit length eigenvalue. Let $\lambda_1, \ldots, \lambda_l$ be the eigenvalues of $A$ with the largest absolute value. As $A$ is symplectic, $\lambda_1^{-1}, \ldots, \lambda_l^{-1}$ are also eigenvalues of $A$. Let $F_+ \subset \C^{2n}$ be the sum of generalized eigenspaces corresponding to $\lambda_1, \ldots, \lambda_l$ and let $F_- \subset \C^{2n}$ be the sum of generalized eigenspaces corresponding to $\lambda_1^{-1}, \ldots, \lambda_l^{-1}$. 
Now set 
$E_\pm \coloneqq \{v \in F_\pm : \overline{v} = v\}$.
For $v \in E_+ \cup E_-$, let $V_v \subset \Q^{2n}$ be the smallest rational subspace such that $\R v \subset V_v \otimes \R$. Finally, define $\bT_v \subset \bT^{2n}$ to be the subtorus given by the projection of $V_v \otimes \R$ onto $\bT^{2n}$. In general, we do not have $A\bT_v = \bT_v$. Then we have the following characterization of the support of semiclassical measures associated to $A$.

\begin{theorem}\label{thm:full_support}
Let $A \in \Sp(2n, \Z)$ such that $A$ has a non-unit length eigenvalue. Further suppose $A|_{E_+}$ and $A|_{E_-}$ are diagonalizable over $\C$. Let $\mu$ be a semiclassical measure associated to $A$. 
\begin{enumerate}
    \item For some $z \in \bT^{2n}$ and $v \in (E_+ \cup E_-) \setminus \{0\}$, $\overline{\{A^k (z + \bT_v) : k \in \Z\}} \subset \supp \mu$. \label{item:K}
    \item  If for all $k \in \N$, the characteristic polynomial of $A^k$ is irreducible over the rationals, then $\bT^{2n} = \overline{\{A^k (z + \bT_v) : k \in \Z\}}$, i.e., $\mu$ has full support.   \label{item:full_support} 
\end{enumerate} 
\end{theorem}

From Lemma \ref{lem:first-power-suffices}, if $n \geq 2$ and the characteristic polynomial of $A$ is irreducible over $\Q$ with the maximal Galois group $S_2 \wr S_n$, then the characteristic polynomial of $A^k$ for all $k$ is irreducible.  For the definition of the Galois group of a polynomial, see the start of Appendix~\ref{appendix2}.

In Appendix~\ref{appendix1}, we characterize $\overline{\{A^k (z + \bT_v) : k \in \Z\}}$ by the smallest $k_0 \in \N$ such that the characteristic polynomial of $A^{k_0}$ is irreducible over the rationals, including the case of Theorem~\ref{thm:full_support}~\eqref{item:full_support}. If the characteristic polynomial of $A^k$ is reducible, there is an obstruction from rational, $A^k$-invariant tori.

In Appendix~\ref{appendix2}, Anderson and Lemke Oliver show that for 100\% of matrices $A \in \Sp(2n, \Z)$ (when ordered by any norm on the space of $2n \times 2n$ matrices), the characteristic polynomial of $A^k$ is irreducible over the rationals for all $k$ and $A$ has a non-unit length eigenvalue. If all powers of $A$ have irreducible characteristic polynomial irreducible over the rationals, then  $A|_{E_\pm}$ must be diagonalizable over $\C$. We conclude that 100\% of $A \in \Sp(2n, \Z)$ have semiclassical measures with full support. 

We now turn our attention to when the irreducibility condition in Theorem~\ref{thm:full_support}~\eqref{item:full_support} fails. From Theorem~\ref{thm:full_support}~\eqref{item:K}, the support of the semiclassical measure must contain a torus. In the case when the support is exactly a torus, the symplectic structure plays a key role.

We now review some basic definitions from  symplectic linear algebra. 
Let $z=(x, \xi)$, $\omega=(y, \eta) \in \R^{2n}$. Define the \emph{standard symplectic form} $\sigma$ on $\R^{2n}$ by $\sigma(z, \omega) \coloneqq \lrang{\xi, y} - \lrang{\eta, x}$.
\begin{definition}
Suppose $V$ is a subspace of $\R^{2n}$. The \emph{symplectic complement} of $V$ is given by
$$V^{\perp \sigma} \coloneqq \{\omega \in \R^{2n} \mid \sigma(\omega, z) =0 \text{ for all } z \in V\}.$$ 
If $V^{\perp \sigma} = V$, we call $V$ \emph{Lagrangian}.
More generally, if $V^{\perp \sigma} \subset V$, we say that $V$ is \emph{coisotropic} and if $V \subset V^{\perp \sigma}$, we say that $V$ is \emph{isotropic}. Finally, if $V \cap V^{\perp \sigma} = \{0\}$, then $V$ is \emph{symplectic}.
\end{definition}
We know that $\dim V + \dim V^{\perp \sigma} = n$. Therefore, if $V$ is Lagrangian, $\dim V= n$ and  if $V$ is coisotropic, $n \leq \dim V \leq 2n$. If $V=\R^{2n}$ or if $V$ is a codimension-one subset of $\R^{2n}$, $V$ is automatically coisotropic.

Kelmer~\cite{Kel} constructed semiclassical measures supported on proper coisotropic tori. Our next result shows such counterexamples do not exist when $V$ is not coisotropic.

\begin{theorem}\label{thm:coisotropic}
Let $A \in \Sp(2n, \Z)$ be hyperbolic  (i.e., it has no unit-length eigenvalues) and diagonalizable over $\C$. Assume that $V \subset \R^{2n}$ is a rational $A$-invariant subspace which is not coisotropic. Let $\Sigma$ be the union of finitely many subtori of $\bT^{2n}$ with tangent space $V$. Then no semiclassical measure for the corresponding quantum cat map can have support contained in $\Sigma$.
\end{theorem}

First, note that as $\mu$ is $A$-invariant, it  does not impose any meaningful constraints to require $V$ to be $A$-invariant.

As $A$ is integer-valued, the condition of diagonalizability implies that $A$ is semi-simple over $\Q$. In the proof of Theorem~\ref{thm:coisotropic}, we use semi-simplicity to know when we restrict to $V^{\perp \sigma}$, there exists a rational, $A|_{V^{\perp \sigma}}$-invariant complement of $V \cap V^{\perp \sigma}$. If instead of diagonalizability, we assume only that such a complement exists, we suspect the above theorem still holds. This includes the case where $V$ is symplectic: if $V$ is symplectic, then $V^{\perp \sigma}$ is the desired complement. However, in that case, the proof is made more technical by the possibility of generalized eigenvectors.

Theorem~\ref{thm:coisotropic} implies that a semiclassical measure cannot be supported on a single symplectic subtorus. However, in the following result, we give examples of semiclassical measure supported on the union of two symplectic, transversal subtori.

\begin{theorem}\label{thm:counterexample}
There exists hyperbolic $A \in \Sp(4, \Z)$ such that the quantization of $A$ has a sequence of eigenfunctions that that weakly converge to the semiclassical measure 
$$\mu = \frac{1}{2}\left[\delta(x_1, \xi_1) \otimes dx_2 d\xi_2 + dx_1 d\xi_1 \otimes \delta(x_2, \xi_2)\right].$$
\end{theorem}

\subsection{Previous results}\label{subsection:previousresults}
We briefly review the literature on semiclassical measures. For additional context, see the literature reviews in~\cite{dyatlov2021semiclassical}*{\S 1.3} and~\cite{Dy}*{\S 3}.

Developed in the 1970s and 80s, the Quantum Ergodicity theorem of  Shnirelman~\cite{Shn}, Zelditch~\cite{Ze},  and  Colin de
Verdi\`{e}re~\cite{colindeverdiere}  was the first result to study the mass of eigenfunctions in quantum chaos. The theorem implies that on negatively curved compact manifolds, a density-one sequence of Laplacian eigenfunctions weakly converges to the Liouville measure. This is known as \emph{equidistribution}. 

In 1994, Rudnick and Sarnak proposed the Quantum Unique Ergodicity conjecture~\cite{RS}, which poses that on negatively curved compact manifolds, the whole sequence of eigenfunctions weakly converges to the the Liouville measure. The conjecture has been proven in a few specific cases, some of which are discussed below, but remains open in general.

Much of the work on semiclassical measures has focused on positive lower bounds for their Kolmogorov--Sina\u{\i} entropy, known as \emph{entropy bounds}. In the setting of negatively curved Riemannian manifolds, entropy bounds were established by Bourgain and Lindenstrauss~\cite{BL03}, Anantharaman~\cite{anantharaman}, Anantharaman and Nonnenmacher~\cite{AN}, Anantharaman, Koch, and Nonnenmacher~\cite{AKN}, Rivi\`{e}re~\cites{R2, R1}, and Anantharaman and Silberman~\cite{AS}. Using the methods from \cite{anantharaman}, it is likely possible that Theorem~\ref{thm:coisotropic} can be improved to show that a semiclassical measure cannot be supported on $\Sigma \times K$, where $\Sigma$ is as in the statement of Theorem~\ref{thm:coisotropic} and $K$ has small topological entropy.

Recent work has focused on characterizing the support of semiclassical measures on negatively curved surfaces. In 2018, Dyatlov and Jin proved that on hyperbolic surfaces, semiclassical measures have full support~\cite{DJi}. Dyatlov, Jin, and Nonnenmacher generalized this result to negatively curved surfaces in 2022~\cite{DJN}. In 2024, Athreya, Dyatlov, and Miller characterized the support of semiclassical measures on complex hyperbolic quotients~\cite{ADM24}.

Quantum cat maps were introduced in 1980 by  Hannay and Berry~\cite{HB1980}. Much work has been conducted on 
\emph{arithmetic eigenfunctions}, the joint eigenfunctions of a quantum cat map and the Hecke operators. In 2000, Kurlberg and Rudnick showed that the Quantum Unique Ergodicity conjecture holds for arithmetic eigenfunctions in 2 dimensions~\cite{KR}. See also the work of Gurevich and Hadani~\cite{GH}. In 2010, Kelmer considered higher-dimensional arithmetic eigenfunctions,  proving the Quantum Unique Ergodicity conjecture for quantum cat maps that do not have a coisotropic invariant rational subspace~\cite{Kel}. As discussed in the previous section, for quantum cat maps that do have an isotropic invariant rational subspace, Kelmer found sequences of arithmetic eigenfunctions that weakly converge to semiclassical measures supported on proper coisotropic submanifolds of~$\bT^{2n}$. In 2024, Kurlburg, Ostafe, Rudnick, and Shparlinski showed that in higher dimensions, if there are no coisotropic rational A-invariant subspaces,  Quantum Ergodicity holds without the arithmetic assumption ~\cite{KORS24}.

The benefit of studying arithmetic eigenfunctions is that they satisfy additional symmetries. This idea was further developed by Rivi\`{e}re and Wolf in ~\cite{RW25}.  They  characterized the form of semiclassical measures in the case that their eigenfunctions satisfy extra symmetries (but are not necessarily arithmetic eigenfunctions).

We turn our attention back to general eigenfunctions for quantum cat maps. In 2003, Faure, Nonnenmacher, and De Bi\`{e}vre constructed a counterexample to the Quantum Unique Ergodicity conjecture on $2$-dimensional quantum cat maps~\cite{FNdB}. Specifically, they constructed sequences of eigenfunctions that weakly converge to  $\mu = \frac{1}{2} \mu_L + \frac{1}{2} \delta_0$, where $\mu_L$ is the volume measure.

Brooks in~\cite{Bo10} and Faure and Nonnenmacher in~\cite{FN04} proved `entropy-like' bounds for semiclassical measures of 2-dimensional quantum cat maps. 
In~\cite{AKN}, Anantharaman, Koch, and Nonnenmacher conjectured that in all dimensions, the optimal lower bound for semiclassical measures of quantum cat maps is half the  topological entropy. 
The entropy of Kelmer's example~\cite{Kel}, mentioned above, is half of the topological entropy. 
Rivi\`ere~\cite{Ri11} proved a lower bound in higher dimensions that, in some cases, matches the optimal conjecture. Importantly, the Rivi\`ere bound matches the entropy for the 2-dimensional example in~\cite{FNdB} and our 4-dimensional example in Theorem~\ref{thm:counterexample}. Both examples have half the maximal entropy. Under additional assumptions on the eigenvalues of the matrix, it is possible to recover Theorem~\ref{thm:coisotropic} from entropy bounds.

Theorem~\ref{thm:coisotropic} is not implied by entropy bounds. Specifically,  let $A \in \Sp(2, \Z)$ with eigenvalues $\lambda, \lambda^{-1}$, where $|\lambda|>1$. Using the ordering of coordinates $(x_1, \xi_1, x_2, \xi_2)$, $A \oplus A \in \Sp(4, \Z)$ and the entropy of any semiclassical measure associated to the quantization of $A \oplus A$  is bounded between $\log |\lambda|$ and $2 \log |\lambda|$. If $\mu = \delta(x_1, \xi_1) \otimes dx_2 d\xi_2$, then the entropy of $\mu$ is $\log |\lambda|$, which is within the bounds. However, $\mu$ is supported on a proper symplectic subtori, which is ruled out by Theorem~\ref{thm:coisotropic}.

In 2021, Schwartz showed that semiclassical measures for quantum cat maps given by hyperbolic $A \in \Sp(2, \Z)$ have full support~\cite{Sch}. Soon after, Dyatlov and J\'{e}z\'{e}quel studied higher-dimensional quantum cat maps given by symplectic matrices $A$ that satisfy a spectral gap condition~\cite{dyatlov2021semiclassical}.  The authors proved that the support of any semiclassical measure  given by such a quantum cat map must contain a subtorus which is dependent on the eigenspace of either the smallest or largest eigenvalue. Furthermore, if the characteristic polynomial of $A$ is irreducible over the rationals, then all semiclassical measures have full support.
Theorem~\ref{thm:full_support} removes the spectral gap condition from the result of  ~\cite{dyatlov2021semiclassical}. Thus, we characterize the support of semiclassical measures for a larger class of matrices. For example, Theorem~\ref{thm:full_support} characterizes the semiclassical measures for matrices of the form $B \oplus B$, where $B \in \Sp(2, \Z)$, while ~\cite{dyatlov2021semiclassical} does not.
Both the result of Schwartz~\cite{Sch} and of  Dyatlov and J\'{e}z\'{e}quel~\cite{dyatlov2021semiclassical} rely on the fractal uncertainty principle of Bourgain and Dyatlov~\cite{BD}. 
However, as we are working in higher dimensions, but do not assume a spectral gap, we cannot use the fractal uncertainty principle of~\cite{BD}. To prove Theorem~\ref{thm:coisotropic}, we use a more basic uncertainty principle, which easily generalizes to higher dimensions. On the other hand, to prove Theorem~\ref{thm:full_support}~\eqref{item:K}, we generalize on the work of~\cite{dyatlov2021semiclassical} using Cohen's higher-dimensional fractal uncertainty principle~\cite{Cohen}.

\subsection{Proof outline}\label{subsection:proofoutline}
\text
We first outline the proofs of Theorems~\ref{thm:full_support} and ~\ref{thm:coisotropic}~\eqref{item:K}. Our proof strategy follows~\cite{dyatlov2021semiclassical}, which in turn was inspired by~\cite{DJi},~\cite{Jin}, ~\cite{DJN}, and ~\cite{anantharaman}. Our method diverges from that of~\cite{dyatlov2021semiclassical} in several ways, which are noted as they appear in the argument. 

Both Theorems~\ref{thm:full_support} and~\ref{thm:coisotropic} proceed by contradiction. For $N_j \rightarrow \infty$, let $u_j \in \cH_{N_j}(\theta_j)$ be a normalized sequence of eigenfunctions for $M_{N_j, \theta_j}$ that weakly converges to a semiclassical measure $\mu$. Each space of quantum states $\cH_{N_j}(\theta_j)$ is endowed with an inner product $\lrang{\cdot, \cdot}_{\cH_{N_j}(\theta_j)}$, defined in \S~\ref{subsection:semiclassical_quantization}. In the proof of Theorem~\ref{thm:coisotropic}, we suppose towards a contradiction that  $\supp \mu \subset \Sigma$, where $\Sigma$ is the union of finitely many subtori with non-coisotropic tangent space $V$. For Theorem~\ref{thm:full_support}, we assume that for all $v \in E_+ \cup E_- \setminus \{0\}$ and all $z \in \bT^{2n}$, $z + \R v$ intersects $\bT^{2n} \setminus \supp \mu$. These initial assumptions are different from \cite{dyatlov2021semiclassical}. By Lemma~\ref{lem:T_v}, we know that $\overline{\R v \bmod \Z^{2n}} = \bT_v$. Thus, by the $A$-invariance of $\mu$, after a contradiction is proven, we can conclude $\overline{\{A^k(z + \bT_v) : k \in \Z\}} \subset \supp \mu$. 

In both proofs, we employ a partition of unity $b_1+ b_2 =1$ on $\bT^{2n}$, where $\supp b_1 \cap \supp \mu = \emptyset$. We quantize $b_1 \circ A^{k}$ and $b_2 \circ A^k$ to obtain operators $B_1(k)$ and $B_2(k)$. For $\w=w_0 \cdots w_{m-1} \in \{1, 2\}^m$, set $B_\w = B_{w_{m-1}}(m-1) \cdots B_{w_0}(0)$. We split the words  $\w \in \{1, 2\}^m$ into two parts. Loosely, let $\cY$ be the set of $\w$ with a large fraction of 1's and let $\cX$ be the set of  $\w$ with a small fraction of 1's. Set 
$B_\cY$ to be the sum of $B_\w$ with $\w \in \cY$  and  $B_\cX$ to be the sum of $B_\w$ with $\w \in \cX$. A calculation shows $B_\cY + B_\cX =I$. 

To reach a contradiction, it suffices to prove  $\|B_\cY u_j\|_{\cH_{N_j}(\theta_j) \rightarrow \cH_{N_j}(\theta_j)}$
and $\|B_\cX u_j\|_{\cH_{N_j}(\theta_j) \rightarrow \cH_{N_j}(\theta_j)}$ both converge to 0.
The support condition on $b_1$ gives the decay for $B_\cY$. To study $B_\cX$, we employ the triangle inequality: we bound $\# \cX$ and examine the decay of $\|B_\w\|_{\cH_{N_j}(\theta_j) \rightarrow \cH_{N_j}(\theta_j)}$ for $\w \in \cX$. For the latter, we split $\w$ into two equal parts: $\w= \w_+ \w_-$. Then, we use $\w_\pm$ to construct symbols $b_\pm$ such the proof of the decay of $\|B_\w\|_{\cH_{N_j}(\theta_j) \rightarrow \cH_{N_j}(\theta_j)}$ is reduced to  the decay of $\|\op_h(b_+) \op_h(b_-)\|_{L^2 \rightarrow L^2}$. We show that $\|\op_h(b_+) \op_h(b_-)\|_{L^2 \rightarrow L^2} \rightarrow 0$ via different uncertainty principles for Theorems~\ref{thm:coisotropic} and~\ref{thm:full_support}, which, in turn, are different from the uncertainty principle used in \cite{dyatlov2021semiclassical}.

In the case of Theorem~\ref{thm:coisotropic}, we exploit the fact that $A$ is hyperbolic and $V$ is not coisotropic to show that there exists an $A$-invariant symplectic subspace $W=W_+ + W_-$ such that  $\|A^k \mid_{W_-}\|= \cO(\Lambda^{-k})$ and $\|A^{-k} \mid_{W_+}\|= \cO(\Lambda^{-k})$ for some $\Lambda>1$ as $k \rightarrow \infty$. This gives control over $\supp (b_2 \circ A^{k})$, and therefore $\supp (b_\pm)$, in the directions of $W_+$ and $W_-  $. After using Darboux's theorem to  straighten out $\supp (b_2 \circ A^{k})$, the decay of $\|\op_h(b_+) \op_h(b_-)\|_{L^2 \rightarrow L^2}$ follows by the uncertainty principle Lemma~\ref{lem:aandbbounds}.

Meanwhile, for Theorem~\ref{thm:full_support},  we use Darboux's theorem to straighten out $E_\pm$. Under this straightening, we show that the projection of $\supp b_\pm$ onto $E_\mp$ is porous on lines, in the sense of Definition~\ref{def:porous_on_lines}. Then the decay of $\|\op_h(b_+) \op_h(b_-)\|_{L^2 \rightarrow L^2}$ follows from a generalization of Cohen's higher-dimensional fractal uncertainty principle, Proposition~\ref{prop:fractal_uncertainty}.

We now summarize the argument for Theorem~\ref{thm:counterexample}.
The result generalizes the above-mentioned 2-dimensional counterexample of Faure, Nonnenmacher, and De Bi\`{e}vre in~\cite{FNdB}. 
Their proof relies on the existence of short periods: for any hyperbolic $A \in \Sp(2, \Z)$, there is a sequence $N_j \rightarrow \infty$ such that $M_{N_j, \theta}$ has period $P(N_j) \sim \frac{2\log N_j}{\log \lambda}$. For $N= N_j$ and a Gaussian $G_{N}$, they construct a sequence of eigenfunctions by averaging over functions of the form $M_{N, \theta}^t G_{N}$ for $|t| \leq P(N)/2$. 
For $|t| \leq P(N)/4$, $M_{N, \theta}^t G_{N}$ weakly converges to $\delta(x, \xi)$, while for $P(N)/4  \leq |t| \leq P(N)/2$, $M_{N, \theta}^t G_{N}$ weakly converges to $dx d\xi$. Therefore, by averaging, they obtain the semiclassical measure $\tfrac{1}{2}(\delta(x, \xi) +dx_1 d\xi_1)$. 

To obtain our example in 4 dimensions, for $A \in \Sp(2, \Z)$, we examine $A \oplus A$, which quantizes to $M_{N, \theta} \otimes M_{N, \theta}$.  Our eigenfunction is the average over functions of the form $M_{N, \theta}^t G_{N} \otimes M_{N, \theta}^{t+\frac{P(N)}{2}} G_{N}$ for $|t| \leq P(N)/2$. We show for $|t| \leq P(N)/4$, $M_{N, \theta}^t G_{N} \otimes M_{N, \theta}^{t+\frac{P(N)}{2}} G_{N}$ weakly converges to $\delta(x_1, \xi_1) \otimes dx_2 d\xi_2$ and for $P(N)/4  \leq |t| \leq P(N)/2$, $M_{N, \theta}^t G_{N} \otimes M_{N, \theta}^{t+\frac{P(N)}{2}} G_{N}$ weakly converges to $dx_1 d\xi_2 \otimes \delta(x_2, \xi_2)$. Therefore, we have the semiclassical measure $\mu = \frac{1}{2}\left[\delta(x_1, \xi_1) \otimes dx_2 d\xi_2 + dx_1 d\xi_1 \otimes \delta(x_2, \xi_2)\right]$.

\subsection{Structure of the paper} 
\begin{itemize}
    \item In \S\ref{section:preliminaries}, we review the required preliminaries for this paper. First in \S\S\ref{subsection:semiclassical_quantization}--\ref{subsection:symbols_torus}, we survey semiclassical quantizations and metaplectic transformations. In \S\ref{subsection:uncertainty}, we present the two uncertainty principles needed to prove Theorem~\ref{thm:full_support}~\eqref{item:K} and Theorem~\ref{thm:coisotropic}.
    \item In \S\ref{section:groundwork}, we prove  Theorem~\ref{thm:full_support}~\eqref{item:K} and Theorem~\ref{thm:coisotropic} up to Lemma~\ref{lem:Bw_decay}, which claims that $\|B_\w\|_{\cH_{N_j}(\theta_j) \rightarrow \cH_{N_j}(\theta_j)}$ converges to zero for $\w \in \cX$. 
    \item In \S\ref{section:thm1.3}, we prove Lemma~\ref{lem:Bw_decay} under the conditions of Theorem~\ref{thm:coisotropic}.
    \item In \S\ref{section:thm1.4}, we prove Lemma~\ref{lem:Bw_decay} under the conditions of Theorem~\ref{thm:full_support}.
    \item In \S\ref{section:counterexample}, we prove Theorem~\ref{thm:counterexample}.
    \item In Appendix~\ref{appendix1}, we characterize $\overline{\{A^k(z + \bT_v) : k \in \Z\}} \subset \supp \mu$ and prove Theorem~\ref{thm:full_support}~\eqref{item:full_support}.
    \item In Appendix~\ref{appendix2}, Anderson and Lemke Oliver show that, under any ordering by a norm, for 100\% of matrices $A \in \Sp(2n, \Z)$, the characteristic polynomial of $A^k$ for all $k$ is irreducible over the rationals and the Galois group of the characteristic polynomial of $A^k$ is the wreath product $S_2 \wr S_n$.
\end{itemize}

\section{Preliminaries}\label{section:preliminaries}

\subsection{Semiclassical quantization}\label{subsection:semiclassical_quantization}
We begin with a review of the necessary definitions for this paper. First, recall the semiclassical Weyl quantization. For $a \in \rS(\R^{2n})$ and a semiclassical parameter $h \in (0,1]$, 
\begin{equation}\label{eq:weylquant}
\op_h(a)f(x) \coloneqq \frac{1}{(2\pi h)^n} \int_{\R^{2n}} e^{\frac{i}{h} \lrang{x -y, \xi}} a\left( \frac{x + y}{2}, \xi \right) f(y) dy d\xi, \quad  f \in \rS(\R^n).
\end{equation}

Define the symbol class 
\begin{equation}\label{eq:def_S(1)}
S(1) \coloneqq \left\{a \in C^\infty\left(\R^{2n}\right): \sup_{(x, \xi) \in \R^{2n}} \left|\partial^\alpha_{(x, \xi)} a \right| < \infty \text{ for all } \alpha \in \N^{2n}\right\},
\end{equation}
which naturally induces the seminorms $$\|a\|_{C^m} \coloneqq \max_{|\alpha| \leq m } \sup_{\R^{2n}} |\partial_{(x, \xi)}^\alpha a |, \quad m \in \N_0.$$ From~\cite{z12semiclassical}*{Theorem 4.16}, for $a \in S(1)$, $\op_h(a)$ acts on both $\mathscr{S}(\R^n)$ and  $\mathscr{S}'(\R^n)$.

Let $\omega=(y, \eta) \in \R^{2n}$. We call $U_\omega \coloneqq \op_h(a_\omega)$ a \emph{quantum translation}, where $a_\omega(z)\coloneqq \exp(\frac{i}{h} \sigma(\omega, z))$.
Noting that $a_\omega(z) \in S(1)$, $U_\omega$ is well-defined and acts on $\mathscr{S}(\R^n)$. However, the derivatives of $a_\omega$ are not bounded uniformly in $h$.
From~\cite{z12semiclassical}*{Theorem 4.7},
\begin{equation}
  \label{e:U-omega-def}
U_\omega f(x) = e^{\frac{i}{h} \lrang{\eta, x} - \frac{i}{2h}\lrang{y, \eta} } f(x-y).
\end{equation}
Thus, $U_\omega$ is a unitary operator on $L^2(\R^n)$ that satisfies the following exact Egorov's theorem for all $a \in S(1)$:
\begin{equation}\label{eq:egorov}
U^{-1}_\omega \op_h(a) U_\omega =\op_h(\tilde{a}), \quad  \quad \tilde{a}(z)\coloneqq a(z+\omega).
\end{equation}
From the fact that $U_\omega U_{\omega'} = e^{\frac{i}{2h} \sigma(\omega, \omega')} U_{\omega+\omega'}$, we deduce the following commutator formula,
\begin{equation}\label{eq:commutator}
U_\omega U_{\omega'} = e^{\frac{i}{h} \sigma(\omega, \omega')} U_{\omega'}U_\omega.
\end{equation}

Now let $\Sp(2n, \R)$ be the group of \emph{real symplectic $2n \times 2n$ matrices}. By symplectic, we mean that $A$ preserves the standard symplectic form, i.e., $\sigma(Az, A\omega)=\sigma(z,\omega)$. Note that in the 2-dimensional case, $\Sp(2, \R)=\SL(2, \R)$. For each $A \in \Sp(2n, \R)$, denote by $\cM_A$ the set of all unitary transformations $M :L^2(\R^n) \rightarrow L^2(\R^n)$ satisfying the following exact Egorov's theorem,
\begin{equation}\label{eq:MA}
M^{-1} \op_h(a) M= \op_h(a \circ A) \quad \text{for all } a \in S(1).
\end{equation}
From~\cite{z12semiclassical}*{Theorem 11.9}, we have both the existence of these transformations and uniqueness up to a unit factor.

Then $\cM \coloneqq \bigcup_{A \in \Sp (2n, \R)} \cM_A$ is a subgroup of unitary transformations of $L^2(\R^n)$ called the \emph{metaplectic group} and the map $M \mapsto A$ is a group homomorphism from $\cM$ to $\Sp(2n, \R)$. An element of the metaplectic group is a \emph{metaplectic transformation}. As a corollary of~\eqref{eq:MA}, we obtain the following intertwining of the metaplectic transformations and quantum translations:
\begin{equation}\label{eq:intertwining}
M^{-1} U_\omega M=U_{A^{-1} \omega} \quad \text{for all } M \in \cM_A, \quad \omega \in \R^{2n}.
\end{equation}

We turn our attention to quantizations of functions on the torus $\bT^{2n}$. Each $a \in C^\infty(\bT^{2n})$ can be identified with a $\Z^{2n}$-periodic function on $\R^{2n}$. Note that any $a \in C^\infty(\bT^{2n})$ is also an element of  $S(1)$. Therefore, its Weyl quantization $\op_h(a)$ is an operator on $L^2(\R^n)$.

By~\eqref{eq:egorov}, we have the following commutation relations:
\begin{equation}\label{eq:commutation_relations}
\op_h(a) U_\omega = U_\omega \op_h(a) \quad \text{for all } a \in C^\infty(\bT^{2n}), \quad \omega \in \Z^{2n}.
\end{equation}
These commutation relations motivate the definition of the finite-dimensional spaces $\cH_N(\theta)$, where $\theta \in \bT^{2n}$ and $N \in \N$ are such that $\op_h(a)$ descends onto $\cH_N(\theta)$. From~\cite{Bouzouina-deBievre}*{Proposition 2.1}, to ensure that these spaces are nontrivial, for the rest of the paper we assume
$$
h= (2\pi N)^{-1}.
$$

We call $\cH_N(\theta)$ a \emph{space of quantum states}. Specifically, for each $\theta \in \bT^{2n}$, set 
\begin{equation}\label{eq:quantum_state}
\cH_N(\theta)\coloneqq \left\{f \in \rS'(\R^n): U_\omega f=e^{2 \pi i \sigma(\theta, \omega) + N \pi i Q(\omega)} f  \text{ for all } \omega \in \Z^{2n}\right\},
\end{equation}
where the quadratic form $Q$ on $\R^{2n}$ is defined by $Q(\omega)=\lrang{y,\eta}$ for $\omega=(y, \eta) \in \R^{2n}$. Define $\Z^n_N \coloneqq \{0, \ldots, N-1\}^n$.
The following lemma gives an explicit basis for $\cH_N(\theta)$.
\begin{lemma}[\cite{dyatlov2021semiclassical}*{Lemma 2.5}]\label{lem:basis}
The space $\cH_N(\theta)$ is $N^n$-dimensional with a basis $\{e_j^\theta\}$, defined for 
 $j \in \Z^n_N$ and $\theta=(\theta_x, \theta_\xi) \in \R^{2n}$. In particular,
 $$e_j^\theta(x) \coloneqq N^{-\frac{n}{2}} \sum_{k \in \Z^n} e^{-2 \pi i \lrang{\theta_\xi, k}} \delta\left(x- \frac{Nk+j-\theta_x}{N} \right).$$
\end{lemma}

We fix an inner product $\lrang{\cdot, \cdot}_{\cH_N(\theta)}$ on each quantum state $\cH_N(\theta)$ by requiring $\{e_j^\theta\}$ to be an orthonormal basis. It can be shown using translation identities for $e_j^\theta$ (see~\cite{dyatlov2021semiclassical}*{(2.36)}) that although each $\{e_j^\theta\}$ depends on the choice of the representative $\theta_x \in \R^n$, the inner product depends only on $\theta \in \bT^{2n}$. Using the bases $\{e_j^\theta\}$, we can consider the spaces $\cH_N(\theta)$ as fibers of a smooth $N^n$ dimensional vector bundle over $\bT^{2n}$, which we call $\cH_N$. 

Fix $\theta=0$ and consider the operator
\begin{equation}\label{eq:Pi_N_def}
\Pi_{N}(0):\mathscr S(\mathbb R^n)\to \mathcal H_{N}(0), \quad \Pi_{N}(0)f=\sum_{j\in\mathbb Z_{N}^n} \langle f,e_j^0\rangle_{L^2}e_j^0.
\end{equation}

Let $\Pi_{N}(0)^*:\mathcal H_N(0)\to\mathscr S'(\mathbb R^n)$
be the adjoint of~$\Pi_{N}(0)$ with respect to the inner products $\langle\cdot,\cdot\rangle_{L^2}$
and $\langle \cdot,\cdot\rangle_{\mathcal H_N(0)}$, that is
\begin{equation}\label{eq:Pi_adjoint}
\langle \Pi_{ N}(0)f,g\rangle_{\mathcal H_N(0)}=\langle f,\Pi_{N}(0)^*g\rangle_{L^2} \quad\text{for all }
f\in \mathscr S(\mathbb R^n),\ g\in\mathcal H_{N}(0).
\end{equation}
A direct computation shows that $\Pi_{N}(0)^*$ is just the embedding map $\mathcal H_{N}(0)\to \mathscr S'(\mathbb R^n)$,
\begin{equation}\label{eq:adjoint_rule}
\Pi_{N}(0)^* g=g\quad\text{for all }g\in\mathcal H_{N}(0)\subset \mathscr S'(\mathbb R^n).
\end{equation}

Now, define the following symmetrization operator
$$
S_{N}:\mathscr S(\mathbb R^n)\to\mathscr S'(\mathbb R^n),\quad
S_{N}f=\sum_{\omega\in\mathbb Z^{2n}}U_\omega f.
$$
Suppose $N$ is even. Write $\omega=(y,\eta) \in \Z^{2n}$. Then using the Poisson summation formula, we compute
\begin{align*}
    S_{N}f(x)&=\sum_{y,\eta\in\mathbb Z^n}e^{2\pi i N\langle \eta,x\rangle}f(x-y)\\
    &=N^{-n}\sum_{y,\ell\in  \mathbb Z^n}f(x-y)\delta(x-\ell/ N)\\
    &=N^{-n}\sum_{\substack{k,r\in\mathbb Z^n\\j\in\mathbb Z_{N}^n}}f(j/N+r)\delta(x-k-j/ N) \quad \text{where } \ell = Nk+j \text{ and } y= k-r \\
    &= N^{-n/2}\sum_{\substack{r\in\mathbb Z^n \\ j\in\mathbb Z_{N}^n}}
    f(j/N+r)e_j^0(x).
\end{align*}
The reduction to $N$ even is for technical simplicity; in the case where $N$ is odd, we pick up the phase constant $e^{-i \pi N \lrang{y, \eta}}$. 

It follows that
$$
S_{N}f=\sum_{j\in\mathbb Z_{N}^n}\langle f,e_j^0\rangle_{L^2}e_j^0.
$$
Therefore, for $f \in \mathscr{S}(\R^n)$ and even $N$, 
\begin{equation}\label{eq:S_N}
\Pi_N(0)^*\Pi_N(0) f = \sum_{l \in \Z^{2n}} U_l f.
\end{equation}

For fixed $N \in \N$ and $a \in C^\infty(\bT^{2n})$, define the quantization 
$$\op_{N, \theta}(a) \coloneqq \op_h(a)|_{\cH_N(\theta)} : \cH_N(\theta) \rightarrow \cH_N(\theta), \quad \theta \in \bT^{2n},$$ which depends smoothly on $\theta$ since $\cH_N(\theta)$ are fibers of the smooth vector bundle $\cH_N$. 
This restriction holds by the definition of $\cH_N(\theta)$ and the commutation relations given in~\eqref{eq:commutation_relations}.

We set $A \in \Sp(2n, \Z)$ and choose a metaplectic transformation $M \in \cM_A$.  We next want to restrict $M$ to $\cH_N(\theta)$. Recall that for $z=(x, \xi)$, we have $\omega = (y, \eta) \in \Z^{2n}$, $\sigma(z, \omega) = \lrang{\xi, y} - \lrang{\eta, x}$ and $Q(\omega) = \lrang{y, \eta}$.  By~\cite{dyatlov2021semiclassical}*{Lemma 2.9}, there exists a unique $\phi_A \in \Z^{2n}_2$ such that for all $\omega \in\mathbb Z^{2n}$, $Q(A^{-1} \omega) - Q(\omega) = \sigma(\phi_A, \omega) \bmod 2\Z$.  
Using the definition of $\cH_N(\theta)$ and~\eqref{eq:intertwining}, we can verify that $M(\cH_N(\theta)) \subset \cH(A \theta +\frac{ N \phi_A}{2})$ for all $\theta \in \bT^{2n}$.

Denote $M_{N, \theta} \coloneqq M|_{\cH_N(\theta)} : \cH_N(\theta) \rightarrow \cH_N(A \theta +\frac{N \phi_A}{2})$, which depends smoothly on $\theta \in \bT^{2n}$.  We require the domain and range of $M_{N, \theta}$ to be the same, in other words, we must have the following \emph{quantization condition},
\begin{equation}\label{eq:domainrange}
(I-A)\theta =\frac{N \phi_A}{2} \mod \Z^{2}.
\end{equation}

Assuming condition~\eqref{eq:domainrange}, we have the following exact Egorov's theorem for all $a \in C^\infty(\bT^{2n})$,
\begin{equation}\label{eq:exactegorov}
M_{N, \theta}^{-1} \op_{N, \theta} (a) M_{N, \theta} =\op_{N, \theta} (a \circ A).
\end{equation}

In this paper, we describe the behavior of the eigenfunctions of $M_{N, \theta}$ in the semiclassical limit. To do so, we define  semiclassical measures.

\begin{definition}\label{def:semiclassical_measure}
Let $N_j \in \N$, $\theta_j \in \bT^{2n}$ be sequences such that $N_j \rightarrow \infty$ and, for all $j$, the quantization condition 
\eqref{eq:domainrange} holds.  Suppose $u_j \in \cH_{N_j}(\theta_j)$ are eigenfunctions of $M_{N_j, \theta_j}$ of norm 1. We say that the sequence $u_j$ \emph{converges weakly} to a Borel measure $\mu$ on $\bT^{2n}$ if 
$$\lrang{\op_{N_j, \theta_j} (a) u_j, u_j}_{\cH_{N_j}(\theta_j)} \rightarrow \int_{\bT^{2n}} a d\mu \quad \text{ for all } a \in C^\infty(\bT^{2n}). $$
We call such a $\mu$ a \emph{semiclassical measure}.
\end{definition}
By taking $a=1$, we see that every semiclassical measure is a probability measure. From the conjugation condition~\eqref{eq:exactegorov},  for all $k \in \Z$ and Borel sets $\Omega$, $\mu(A^k(\Omega))=\mu(\Omega)$. Otherwise put, $\mu$ is $A$-invariant.

\subsection{Symbol calculus}\label{subsection:symbol_calculus}
For the rest of the paper, we use the following notational conventions. 
\begin{notation}\label{notation}
Suppose $(F, \|\cdot \|_F)$ is a normed vector space and $f_h \in F$ is a family depending on a parameter $h>0$. If $\|f_h\|_F = \cO(h^\alpha)$, we write $f_h = \cO(h^\alpha)_F$. 
\end{notation}

\begin{notation}\label{notation2}
We use $C$ to denote a constant, the value of which may vary in each appearance.  
\end{notation}

We now define the exotic symbol calculus $S_{L, \rho, \rho'}$, which was first introduced in~\cite{dyatlov2021semiclassical}. For a more in-depth presentation, see~\cite{dyatlov2021semiclassical}*{\S 2.1.4}.

\begin{definition}\label{def:symbol_classes}
Let $L \subset \R^{2n}$ be a coisotropic subspace and set $0 \leq \rho' \leq \rho$ such that $\rho + \rho' <1$. We say that an $h$-dependent symbol $a(x, \xi; h) \in C^\infty(\R^{2n})$ lies in $S_{L, \rho, \rho'}(\R^{2n})$ if for any choice of constant vector fields $X_1, \ldots, X_k, Y_1, \ldots, Y_m$ on $\R^{2n}$ with $Y_1, \ldots, Y_m$  tangent to $L$, there exists a constant $C$ such that for all $h \in (0, 1]$, $$\sup_{(x, \xi) \in \R^{2n}} \left|X_1 \cdots X_k Y_1 \cdots Y_m a(x, \xi)\right| \leq Ch^{-\rho k -\rho'm}.$$
\end{definition}

These derivative bounds naturally induce a family of seminorms of $S_{L, \rho, \rho'}$.  

When $\rho = \rho' <1/2$, $S_{L, \rho, \rho'}(\R^{2n})$ becomes the following symbol class, see e.g.~\cite{z12semiclassical}*{\S4.4}:
\begin{equation}\label{eq:S_rho(1)}
S_\rho(1) \coloneqq \left\{a(x, \xi) \in C^\infty(\R^{2n}) : |\partial^\alpha a| \leq C_\alpha h^{-\rho |\alpha|} \text{ for all } \alpha\right\}.
\end{equation}
Using $h= (2 \pi N)^{-1}$, we can replace $h$ in the above definition with $N^{-1}$.

We now quote the following properties of $S_{L, \rho, \rho'}$.
\begin{lemma}[\cite{dyatlov2021semiclassical}*{Lemma 2.3}]\label{lem:symbolproperties}
For $a, b \in S_{L, \rho, \rho'}(\R^{2n})$, the following properties hold uniformly in $N$ and $\theta$.
\begin{enumerate}
\item \label{product} $\op_h(a) \op_h(b) = \op_h(a \# b)$, where $a \# b$ satisfies the following asymptotic expansion as $h \rightarrow 0$ for all $l \in \N$:
$$a \# b(z) = \sum_{k=0}^{l-1} \frac{(-ih)^k}{2^k k!} \left(\sigma(\partial_z, \partial_\omega)^k (a(z)b(\omega))\right) |_{\omega=z} + \cO\left(h^{(1-\rho -\rho')l}\right)_{S_{L, \rho, \rho'}}.$$

\item \label{bounded} $\|\op_h(a)\|_{L^2 \rightarrow L^2}$ is bounded uniformly in $h \in (0,1]$. 

\item \label{sharpgarding} If $a \geq 0$ everywhere,  then there exists $C$ such that $$\lrang{\op_h(a) f, f}_{L^2} \geq - Ch^{1-\rho -\rho'} \|f\|_{L^2}^2 \quad \text{for all } f \in L^2(\R^n), \quad 0 < h \leq 1.$$
\end{enumerate}
\end{lemma}
Note that~\eqref{sharpgarding} is a version of the sharp G\aa rding inequality.

\subsection{Symbols on torus}\label{subsection:symbols_torus}
We can also constrain $S_{L, \rho, \rho'}$ to the torus. 

Let the symbol class $S_{L, \rho, \rho'}(\bT^{2n})$ be comprised of the $\Z^{2n}$-periodic symbols in $S_{L, \rho, \rho'}(\R^{2n})$.
From~\cite{dyatlov2021semiclassical}*{\S\S 2.2.2-3}, Lemma~\ref{lem:symbolproperties} restricts to $a, b \in S_{L, \rho, \rho'}(\bT^{2n})$.  We write the details below.
\begin{lemma}\label{lem:torussymbolproperties} For $a, b \in S_{L, \rho, \rho'}(\bT^{2n})$, the following properties hold uniformly in $N$ and $\theta$.
\begin{enumerate}
\item  $\op_{N, \theta}(a) \op_{N, \theta}(b) =\op_{N, \theta}(a \# b)$ where $a \# b$ satisfies the following asymptotic expansion  for all $l \in \N$:
$$a \# b(z) = \sum_{k=0}^{l-1} \frac{(-i)^k}{2^k k!}(2 \pi N)^{-k} \left(\sigma(\partial_z, \partial_\omega)^k (a(z)b(\omega)) \right)|_{\omega=z} + \cO\left(N^{(\rho +\rho'-1)l}\right)_{S_{L, \rho, \rho'}}.$$
\item \label{norm_boundedness} $\|\op_{N, \theta}(a)\|_{\cH_N(\theta) \rightarrow \cH_N(\theta)}$ is bounded by some $S_{L, \rho, \rho'}$-seminorm of $a$, where the choice of the seminorm depends only on $n, \rho$ and $\rho'$. 
\item \label{sharpgarding2} If $a \geq 0$ everywhere, then $$\lrang{\op_{N, \theta}(a) f, f}_{\cH_N(\theta)} \geq -C_a N^{\rho+\rho'-1} \|f\|_{\cH_N(\theta)}^2 \quad \text{for all} \quad f \in \cH_N(\theta).$$
\item $\op_{N, \theta}(a)^* = \op_{N, \theta}(\overline{a})$.
\end{enumerate}
\end{lemma}  
Importantly, Lemma~\ref{lem:torussymbolproperties} implies the following nonintersecting support property:
\begin{equation}\label{eq:nonintersecting}
\op_{N, \theta}(a) \op_{N, \theta}(b) = \cO(N^{-\infty})_{\cH_N(\theta) \rightarrow \cH_N(\theta)} \quad \text{if} \quad a, b \in S_{L, \rho, \rho'}(\bT^{2n}), \quad \supp a \cap \supp b = \emptyset.
\end{equation}

We also need the following lemma on the product of many quantized observables. 
\begin{lemma}[\cite{dyatlov2021semiclassical}*{Lemma 2.8}]\label{lem:manymultiplication}
Assume that $a_1, \ldots, a_M \in S_{L, \rho, \rho'}(\bT^{2n})$, where $M \leq C_0 \log N$, satisfy $\sup_{\bT^{2n}}|a_j| \leq 1$ and each $S_{L, \rho, \rho'}$-seminorm of $a_j$ is bounded uniformly in $j$. Then
\begin{enumerate}
\item \label{manyproduct} The product $a_1 \cdots a_M$ lies in $S_{L, \rho + \varepsilon, \rho'+\varepsilon}(\bT^{2n})$ for all small $\varepsilon>0$.
\item \label{manysymbol} For all $\varepsilon>0$,
$$\op_{N, \theta}(a_1) \cdots \op_{N, \theta}(a_M) = \op_{N, \theta}(a_1 \cdots a_M) + \cO\left(N^{\rho+\rho' -1 + \varepsilon}\right)_{\cH_N(\theta) \rightarrow \cH_N(\theta)},$$
where the implied constant depends only on $\rho, \rho', \varepsilon, C_0$, and the maximum over $j$ of a certain $S_{L, \rho, \rho'}$-seminorm of $a_j$. 
\end{enumerate}
\end{lemma}

To finish our exposition of the properties of $S_{L, \rho, \rho'}(\bT^{2n})$, we take a special case of~\cite{dyatlov2021semiclassical}*{Lemma 2.7}.
\begin{lemma}\label{lem:b_bounds}
Suppose that $a \in S_{L, \rho, \rho'}(\bT^{2n})$ and $|a| \leq 1$ everywhere. Then for all $u \in \cH_N(\theta)$, 
$$\|\op_{N, \theta}(a) u\|_{\cH_N(\theta)} \leq \|u\|_{\cH_N(\theta)} + CN^{\frac{\rho + \rho' -1}{2}} \|u\|_{\cH_N(\theta)},$$
where the constant $C$ depends only on some $S_{L, \rho, \rho'}$ seminorms of $a$.
\end{lemma}
\begin{proof}
By Lemma~\ref{lem:torussymbolproperties}, 
$$I- \op_{N, \theta}(a)^* \op_{N, \theta}(a) = \op_{N, \theta}(1-|a|^2) + \cO(N^{\rho + \rho' -1})_{\cH_N(\theta) \rightarrow \cH_N(\theta)}.$$
As $1-|a|^2 \geq 0$, we again apply Lemma~\ref{lem:torussymbolproperties} to see
$$\lrang{\op_{N, \theta}(1-|a|^2) u, u }_{\cH_N(\theta)} \geq -CN^{\rho +\rho'-1}\|u\|^2_{\cH_N(\theta)}.$$
Thus,
$$\|\op_{N, \theta}(a) u\|_{\cH_N(\theta)} \leq \|u\|_{\cH_N(\theta)} + CN^{\frac{\rho + \rho' -1}{2}} \|u\|_{\cH_N(\theta)},$$
concluding the proof.
\end{proof}

\subsection{Uncertainty principles}\label{subsection:uncertainty}
As mentioned in \S\ref{subsection:previousresults}, previous work on semiclassical measures on quantum cat maps in~\cite{Sch} and~\cite{dyatlov2021semiclassical} applied the fractal uncertainty principle of Bourgain-Dyatlov~\cite{BD}. Our work also requires uncertainty principles, but instead employs a more basic uncertainty principle for Theorem~\ref{thm:coisotropic} and the higher-dimensional fractal uncertainty principle of Cohen~\cite{Cohen} for Theorem~\ref{thm:full_support}~\eqref{item:K}.

\subsubsection{Basic uncertainty principle}
Let $1 \leq d \leq n$.
Our uncertainty principle depends on the following estimate. For $x \in \R^n$, we use the notation $x=(x', x'')$, where $x' \in \R^d$ and $x'' \in \R^{n-d}$.
\begin{lemma}\label{lem:bumpdecay}
Fix $\chi \in C_c^\infty(\R^d)$, $\delta \in [\tfrac{1}{2},1]$, and $x'_0, \xi'_0 \in \R^d$. Then
$$\left\|\chi\left(\frac{x'-x'_0}{h^\delta}\right) \chi\left(\frac{h D_{x'} -\xi'_0}{h^\delta}\right)\right\|_{L^2(\R^n) \rightarrow L^2(\R^n)} = \cO\left(h^{\frac{d}{2}(2\delta-1)}\right).$$  
\end{lemma}

\begin{proof}
First note that
\begin{equation}\label{eq:chi_bound}
\left\|\chi\left(\frac{x'-x'_0}{h^\delta}\right)\right\|_{L^\infty(\R^d) \rightarrow L^2(\R^d)} \leq  \left\|\chi\left(\frac{x'-x'_0}{h^\delta}\right)\right\|_{L^2(\R^d)} = Ch^{\frac{d}{2}\delta}.
\end{equation}
We also have that 
\begin{equation}\label{eq:chi_FT}
\chi\left(\frac{hD_{x'}-\xi'_0}{h^\delta}\right) = \cF_{x'}^{-1} \chi\left(\frac{x'}{h^{\delta-1}}-\frac{\xi'_0}{h^\delta}\right) \cF_{x'},
\end{equation}
where $\cF_{x'}$ denotes the  standard Fourier transform taken over only the first $d$ variables. 

Let $u(x) \in L^2(\R^n)$. For almost every $x'' \in \R^{n-d}$, define the function $u_{x''} \in L^2(\R^d)$ by $u_{x''}(x') = u(x', x'')$. Then set $v_{x''} = \chi\left(\frac{x'-x'_0}{h^\delta}\right) \chi\left(\frac{h D_{x'} -\xi'_0}{h^\delta}\right) u_{x''}$. 
From~\eqref{eq:chi_bound} and~\eqref{eq:chi_FT},
\begin{align*}
\|v_{x''}\|_{L^2} &\leq \left\|\chi \left(\frac{x'-x'_0}{h}\right) \cF^{-1}_{x'}\right\|_{L^1\rightarrow L^2} \left\|\chi\left(\frac{x'}{h^{\delta-1}}-\frac{\xi'_0}{h^\delta}\right) \cF_{x'} u_{x''}\right\|_{L^1}\\
&\leq \left\|\chi\left(\frac{x'-x'_0}{h}\right)\right\|_{L^\infty \rightarrow L^2} \|\cF^{-1}_{x'}\|_{L^1 \rightarrow L^\infty} \left\|\chi\left(\frac{x'}{h^{\delta-1}}-\frac{\xi'_0}{h^\delta}\right)\right\|_{L^2} \|\cF_{x'} u_{x''} \|_{L^2}\\
&\leq Ch^{\frac{d}{2}(2 \delta-1)}\|u_{x''}\|_{L^2},
\end{align*}
where each $L^p$ norm is taken over $\R^d$.
Squaring the above inequality and integrating over $x'' \in \R^{n-d}$ finishes the proof. 
\end{proof}

We now prove our basic uncertainty principle. 
\begin{lemma}\label{lem:aandbbounds}
Let $L_a$ and $L_b$ be coisotropic subspaces such that $L_a\subset \{x'=0\}$ and $L_b\subset \{\xi'=0\}$. Fix $x'_0, \xi'_0 \in \R^d$ and suppose $a \in S_{L_a, \rho + \varepsilon, \varepsilon}(\R^{2n})$ and $b \in S_{L_b, \rho + \varepsilon, \varepsilon}(\R^{2n})$ with $\supp a \subset \{|x'-x'_0| \leq h^{\delta}\}$ and  $\supp b \subset \{|\xi'-\xi'_0| \leq h^{\delta}\}$, where 
$\rho \in (\tfrac{1}{2}, 1)$ with $\rho + 2\varepsilon <1$ and $\delta \in (\tfrac{1}{2}, \rho + \varepsilon]$.
Let $\delta'= \min\{1- \rho- 2\varepsilon, d(2 \delta-1)/2\}$. Then $$\|\op_h(a) \op_h(b) \|_{L^2(\R^n) \rightarrow L^2(\R^n)} = \cO(h^{\delta'}).$$
\end{lemma}
\begin{proof}
 We use $\|\cdot\|$ to denote $\|\cdot\|_{L^2(\R^n) \rightarrow L^2(\R^n)}$. Let $\tilde{\chi} \in C_c^\infty(\R^d)$ such that $\tilde{\chi} =1$ on $\{|x'|\leq 1\}$ and $\supp \tilde{\chi} \subset \{|x'| \leq 2\}$.  Define $\chi_x, \chi_\xi \in C^\infty(\R^{2n})$  by $\chi_x(x, \xi) = \tilde{\chi}(\frac{x'-x_0'}{h^\delta})$ and $\chi_\xi(x, \xi) = \tilde{\chi}(\frac{\xi'-\xi'_0}{h^\delta})$.  Since $L_a \subset \{x' = 0\}$, $L_b \subset \{\xi' =0\}$, and $\delta \leq \rho +\varepsilon$, we know $\chi_x \in S_{L_a, \rho + \varepsilon, \varepsilon}$ and $\chi_\xi \in S_{L_b, \rho + \varepsilon, \varepsilon}$.
 Then using Lemma~\ref{lem:symbolproperties}~\eqref{product} and the fact that $\supp a\cap\supp(1-\chi_x)=\supp b \cap\supp(1-\chi_\xi)=\emptyset$,  we know that 
 $$\op_h (a)=\op_h(a) \op_h\left(\chi_x \right) + \cO(h^{1-\rho-2\varepsilon})_{L^2 \rightarrow L^2},$$ 
$$\op_h (b)=\op_h\left(\chi_\xi \right)\op_h(b) + \cO(h^{1-\rho-2\varepsilon})_{L^2 \rightarrow L^2}.$$  
Thus, it suffices to prove that $$\left\|\op_h(a) \op_h\left(\chi_x \right) \op_h\left(\chi_\xi \right) \op_h(b)\right\| =\cO(h^{\frac{d }{2}(2 \delta-1)}).$$

From Lemma~\ref{lem:symbolproperties}~\eqref{bounded}, both $\|\op_h(a)\|$ and $\|\op_h(b)\|$ are uniformly bounded in $h \in (0,1]$. Thus, using Lemma~\ref{lem:bumpdecay}, we know
\begin{align*}
\left\|\op_h(a) \op_h\left(\chi_x\right) \op_h\left(\chi_\xi \right) \op_h(b)\right\| &\leq C \left\|\op_h\left(\tilde{\chi} \left(\frac{x' -x'_0}{h^{\delta}}\right)\right) \op_h\left(\tilde{\chi} \left( \frac{\xi'-\xi'_0}{h^{\delta}}\right)\right)\right\|\\
&\leq C h^{\frac{d }{2}(2 \delta-1)},
\end{align*}
which concludes our proof.
\end{proof}

\subsubsection{Higher-dimensional fractal uncertainty principle}
We begin with some definitions. Let $B_r(x)$ denote the ball centered at $x$ of radius $r$.
\begin{definition}\label{def:porous_on_balls}
A set $X \subset \R^n$ is \emph{$\nu$-porous on balls} from scales $\alpha_0$ to $\alpha_1$ if for every ball $B$ of diameter $\alpha_0 < R< \alpha_1$, there is some $x \in B$ such that $B_{\nu R}(x) \cap X = \emptyset$.
\end{definition}

\begin{definition}\label{def:porous_on_lines}
A set $X \subset \R^n$ is \emph{$\nu$-porous on lines} from scales $\alpha_0$ to $\alpha_1$ if for all line segments $\tau$ with length $\alpha_0 < R< \alpha_1$, there is some $x \in \tau$ such that $B_{\nu R}(x) \cap X = \emptyset$.
\end{definition}
Clearly, line porosity is a stronger condition than ball porosity. 

\begin{definition}\label{def:semiclassical_FT}
For $h>0$, we define the \emph{unitary semiclassical Fourier transform} $\cF_h :L^2(\R^n) \to L^2(\R^n)$ given by
$$\cF_h f(\xi) =(2 \pi h)^{-\frac{n}{2}}\int_{\R^n} e^{-\frac{i}{h} \lrang{x, \xi}} f(x) dx.$$
\end{definition}

We now state Cohen's higher-dimensional fractal uncertainty principle. 
\begin{theorem}[\cite{Cohen}*{Theorem 1.1}]\label{thm:fractal_uncertainty}
Set $0 < \nu \leq \frac{1}{3}$. Let
\begin{itemize}
\item $X_- \subset [-1,1]^n$ be $\nu$-porous on balls from scales $h$ to 1;
\item $X_+ \subset [-1,1]^n$ be $\nu$-porous on lines from scales $h$ to 1.
\end{itemize}
Then there exist $\beta, C >0$, depending only on $\nu$ and $n$, such that
\begin{equation*}
\|\1_{X_-} \cF_h \1_{X_+}\|_{L^2 (\R^n) \rightarrow L^2(\R^n)} \leq Ch^\beta.
\end{equation*}
\end{theorem}

\begin{notation}
For any set $S \subset \R^n$, define 
\begin{equation}\label{eq:neighborhood_def}
S(\delta) \coloneqq S + B_\delta (0).
\end{equation}
\end{notation}
If $S$ is porous, for our application of the fractal uncertainty principle, we also want to say that $S(\delta)$ is porous. The following two lemmas, which are generalizations of~\cite{DJN}*{Lemma 2.11}, outline when this is possible.

\begin{lemma}\label{lem:gen_porous_balls}
Let $\nu \in (0,1 )$, $0 < \alpha_0 \leq \alpha_1$, and $0 < \alpha_2 \leq \frac{\nu}{2} \alpha_1$. Assume that $X \subset \R^n$ is $\nu$-porous on balls from scales $\alpha_0$ to $\alpha_1$. Then the neighborhood $X(\alpha_2)$ is $\frac{\nu}{2}$-porous on balls from scales $\max(\alpha_0, \frac{2}{\nu}\alpha_2)$ to $\alpha_1$.
\end{lemma}

\begin{proof} We take a ball $B \subset \R^n$ of diameter $\max(\alpha_0, \frac{2}{\nu} \alpha_2) \leq R \leq \alpha_1$. As $X$ is $\nu$-porous on balls on scales $\alpha_0$ to $\alpha_1$, there exists some $x \in B$ such that $B_{\nu R} (x) \cap X = \emptyset$.

As $\nu R \geq 2 \alpha_2$, we know $B_{\frac{\nu R}{2} + \alpha_2}(x) \subset B_{\nu R}(x)$. Therefore $B_{\frac{\nu R}{2}} \cap X(\alpha_2) = \emptyset$. 
\end{proof}

\begin{lemma}\label{lem:gen_porous_lines}
Let $\nu \in (0,1 )$, $0 < \alpha_0 \leq \alpha_1$, and $0 < \alpha_2 \leq \frac{\nu}{2} \alpha_1$. Assume that $X \subset \R^n$ is $\nu$-porous on lines from scales $\alpha_0$ to $\alpha_1$. Then the neighborhood $X(\alpha_2)$ is $\frac{\nu}{2}$-porous on lines from scales $\max(\alpha_0, \frac{2}{\nu}\alpha_2)$ to $\alpha_1$.
\end{lemma}

\begin{proof} We take a line segment $\tau \subset \R^n$ of length $\max(\alpha_0, \frac{2}{\nu} \alpha_2) \leq R \leq \alpha_1$. As $X$ is $\nu$-porous on lines on scales $\alpha_0$ to $\alpha_1$, there exists some $x \in \tau$ such that $B_{\nu R} (x) \cap X = \emptyset$.

As $\nu R \geq 2 \alpha_2$, we know $B_{\frac{\nu R}{2} + \alpha_2}(x) \subset B_{\nu R}(x)$. Therefore $B_{\frac{\nu R}{2}} \cap X(\alpha_2) = \emptyset$. 
\end{proof}

Adapting the argument from~\cite{DJN}*{Proposition 2.9} to higher dimensions we can generalize Theorem~\ref{thm:fractal_uncertainty} to unbounded sets. The proof relies on almost orthogonality, Lemma~\ref{lem:gen_porous_balls}, and Lemma~\ref{lem:gen_porous_lines}.
\begin{proposition}\label{prop:FUP_unbounded}
Set $0 < \nu \leq \frac{1}{3}$. Let
\begin{itemize}
\item $X_-\subset\mathbb R^n$ be $\nu$-porous on balls from scales $h$ to 1;
\item $X_+\subset\mathbb R^n$ be $\nu$-porous on lines from scales $h$ to 1.
\end{itemize}
Then there exist $\beta, C >0$, depending only on $\nu$ and $n$, such that
\begin{equation*}
\|\1_{X_-} \cF_h \1_{X_+}\|_{L^2 (\R^n) \rightarrow L^2(\R^n)} \leq Ch^\beta.
\end{equation*}
\end{proposition}

\begin{proof}
1. We first show that instead of $\1_{X_\pm}$, it suffices instead to consider certain smooth functions $\chi_\pm \in C^\infty(\R^n; [0,1])$. We set $\chi_\pm$ to be the convolution of the indicator function of $X_\pm (\nu h/8)$ with a smooth cutoff function supported in $B_{\nu h/8}(0)$. For a more detailed construction, see~\cite{DZ16}*{Lemma 3.3}. Then,
$$
\supp \chi_\pm \subset X_\pm(\nu h/4), \quad \supp(1-\chi_\pm)\cap X_\pm = \emptyset,
$$
\begin{equation}\label{eq:derivative_bounds_chi}
\sup |\partial^\alpha \chi_\pm| \leq C_\alpha h^{-|\alpha|},
\end{equation}
where $C_\alpha$ depends on $\alpha$, $\nu$, and $n$.
Therefore, 
$$\|\1_{X_-} \cF_h \1_{X_+}\|_{L^2 (\R^n) \rightarrow L^2(\R^n)} = \|\1_{X_-}\chi_- \cF_h \chi_+ \1_{X_+} \|_{L^2 (\R^n) \rightarrow L^2(\R^n)} \leq  \|\chi_- \cF_h  \chi_+\|_{L^2 (\R^n) \rightarrow L^2(\R^n)}.$$

We next split up the support of $\chi_\pm$ via a partition of unity. Fix $\chi \in C_c^\infty(\R^n; [0,1])$ such that $\supp \chi \subset [-1, 1]^n$ and 
$$1 = \sum_{j \in \Z^n} \chi_j \quad \text{where} \quad \chi_j(x) \coloneqq \chi(x-j).$$
Then set
\begin{equation}\label{eq:split_up_chi}
\chi^\pm_j \coloneqq \chi_j \chi_\pm, \quad \supp \chi_j^\pm \subset X_\pm(\nu h/4) \cap \left([-1,1]^n +j\right).
\end{equation}

Note that $\chi_j^\pm$ still satisfy the derivative bounds~\eqref{eq:derivative_bounds_chi}. With respect to the strong operator topology, 
$$\chi_- \cF_h \chi_+ = \sum_{j,k \in \Z^n} A_{jk} \quad \text{where} \quad A_{jk} \coloneqq \chi_j^- \cF_h \chi_k^+.$$

Therefore, it suffices to show 
\begin{equation}\label{eq:Ajk_sum_bound}
\left\|\sum_{j,k \in \Z^n} A_{jk}\right\|_{L^2(\R^n) \rightarrow L^2(\R^n)} \leq Ch^{\beta}.
\end{equation}

2. We will show
\begin{align}
\label{eq:sum_AA*_estimate} \sup_{j,k} \sum_{j',k'} \|A_{jk}A^*_{j'k'}\|^{1/2}_{L^2(\R^n) \rightarrow L^2(\R^n)} &\leq Ch^\beta;\\ 
\label{eq:sum_A*A_estimate}  \sup_{j,k} \sum_{j',k'} \|A^*_{j'k'}A_{jk}\|^{1/2}_{L^2(\R^n) \rightarrow L^2(\R^n)} &\leq Ch^\beta.
\end{align}
Then,~\eqref{eq:sum_AA*_estimate} and~\eqref{eq:sum_A*A_estimate} imply~\eqref{eq:Ajk_sum_bound} via  the Cotlar--Stein Theorem~\cite{z12semiclassical}*{Theorem C.5}.

Note that~\eqref{eq:sum_AA*_estimate} and~\eqref{eq:sum_A*A_estimate} follow from the following bounds for all $j, k, j', k' \in \Z^n$ and $N \in \N:$ 
\begin{align}
\label{eq:Ajk_estimate} \|A_{jk}\|_{L^2(\R^n) \rightarrow L^2(\R^n)} &\leq Ch^{2 \beta},\\ 
\label{eq:AA*_estimate} \|A_{jk}A^*_{j'k'}\|_{L^2(\R^n) \rightarrow L^2(\R^n)} &\leq C_Nh^{-n} (1+ |j-j'| + |k-k'|)^{-N},\\ 
\label{eq:A*A_estimate}\|A^*_{j'k'}A_{jk}\|_{L^2(\R^n) \rightarrow L^2(\R^n)} &\leq C_Nh^{-n} (1+ |j-j'| + |k-k'|)^{-N}, 
\end{align}
where $C, \beta>0$ depend only $\nu$ and $n$ and $C_N >0$ depends only on $N$ and $n$. 
Indeed, we obtain~\eqref{eq:sum_AA*_estimate} using~\eqref{eq:Ajk_estimate} for $|j-j'| + |k -k'| \leq h^{-\beta/2n}$ and using $\eqref{eq:AA*_estimate}$ with $N = \lceil 4+ 4n + 2n/\beta \rceil$ for  $|j-j'| + |k -k'| > h^{-\beta/2n}$. Similarly,~\eqref{eq:sum_A*A_estimate} follows from~\eqref{eq:Ajk_estimate} and~\eqref{eq:A*A_estimate}.

3. We first show~\eqref{eq:Ajk_estimate} using Theorem~\ref{thm:fractal_uncertainty}. For $j \in \Z^n$, set $I_j =[-1,1]^n +j$. Then by~\eqref{eq:split_up_chi},
\begin{align*}
\|A_{jk}\|_{L^2(\R^n) \rightarrow L^2(\R^n)} &\leq \|\1_{X_-(\nu h/4) \cap I_j} \cF_h \1_{X_+(\nu h/4) \cap I_k}\|_{L^2(\R^n) \rightarrow L^2(\R^n)}\\
&=\|\1_{[X_-(\nu h/4) \cap I_j]-j} \cF_h \1_{[X_+(\nu h/4) \cap I_k]-k}\|_{L^2(\R^n) \rightarrow L^2(\R^n)}.
\end{align*}

By Lemma~\ref{lem:gen_porous_balls} and Lemma~\ref{lem:gen_porous_lines}, we know that $X_-(\nu h/4)$ is $\nu/2$-porous on balls from scales $h$ to 1 and $X_+(\nu h/4)$ is $\nu/2$-porous on lines from scales $h$ to 1. We then apply Theorem~\ref{thm:fractal_uncertainty} to conclude~\eqref{eq:Ajk_estimate}. 

4. We now prove~\eqref{eq:AA*_estimate};~\eqref{eq:A*A_estimate} is shown similarly.

If $|k -k'| > 2 \sqrt{n}$, then $\supp \chi_k^+ \cap \supp \chi_{k'}^+ = \emptyset$, which implies $A_{jk}A^*_{j'k'} =0$. Thus, we assume $|k-k'| \leq 2 \sqrt{n}$. We denote the integral kernel of $A_{jk}A^*_{j'k'}$ by $\cK$. Specifically,
$$\cK(x,y) = (2\pi h)^{-n} \chi_j^-(x) \chi_{j'}^-(y) \int_{\R^n} e^{i \lrang{y-x, \xi}/h} \chi_k^+(\xi) \chi_{k'}^+(\xi) d\xi.$$
It suffices to assume that $|j-j'| >2 \sqrt{n}$. Then $\frac{1}{10 \sqrt{n}}|j-j'| \leq |x-y|$ on the support of $\cK$. As $\chi_k^+ \chi_{k'}^+$ is supported inside a cube with side length $2$ and satisfies the derivative bounds~\eqref{eq:derivative_bounds_chi}, we can integrate by parts $N$ times in $\xi$ to get 
$$\sup_{x, y} |\cK(x,y)| \leq C_Nh^{-n}|j-j'|^{-N}.$$
As $\cK(x,y)$ is supported inside a cube of side length 2, from Schur's inequality (see~\cite{z12semiclassical}*{Theorem 4.21}) we know
$$\|A_{jk}A^*_{j'k'}\|_{L^2(\R^n) \rightarrow L^2(\R^n)} \leq C_Nh^{-n}(1 +|j-j'| + |k-k'|)^{-N},$$
which completes the proof.
\end{proof}

We can extend Proposition~\ref{prop:FUP_unbounded} by using a simple generalization of~\cite{DJN}*{Proposition 2.10}. Using the notation of~\cite{DJN}, we set $\gamma_0^\pm = \varrho$ and $\gamma_1^\pm =0$ to obtain the following.

\begin{proposition}
\label{prop:fractal_uncertainty}
Set $0 < \nu \leq \frac{1}{3}$ and $1/2 < \varrho \leq 1$. Let 
\begin{itemize}
\item $X_- \subset \R^n$ be $\nu$-porous on balls from scales $h^\varrho$ to 1;
\item $X_+ \subset \R^n$ be $\nu$-porous on lines from scales $h^\varrho$ to 1.
\end{itemize}
Then there exists $\beta, C >0$, depending only on $\nu$, $n$, and $\varrho$ such that
\begin{equation*}
\|\1_{X_-} \cF_h \1_{X_+}\|_{L^2 (\R^n) \rightarrow L^2(\R^n)} \leq Ch^\beta.
\end{equation*}
\end{proposition}

\section{Proof of Theorems~\ref{thm:coisotropic} and~\ref{thm:full_support}~\eqref{item:K} up to a key lemma}\label{section:groundwork}

Here we begin the proofs of Theorem~\ref{thm:coisotropic} and  Theorem~\ref{thm:full_support}~\eqref{item:K}.  Theorem~\ref{thm:full_support}~\eqref{item:full_support} is shown in 
Lemma~\ref{prop:K}

\subsection{Definitions}\label{subsection:definition}
Choose $N_j \rightarrow \infty$ and $\theta_j \in \bT^{2n}$ such that $N_j, \theta_j$ satisfy the quantization condition~\eqref{eq:domainrange}. Let $u_j \in \cH_{N_j}(\theta_j)$ be a normalized sequence of eigenfunctions for $M_{N_j, \theta_j}$ that weakly converges to a semiclassical measure $\mu$. For Theorem~\ref{thm:coisotropic}, we assume towards a contradiction that $\supp \mu \subset \Sigma$, where $\Sigma$ is the union of finitely many subtori $V$ with non-coisotropic tangent space $V$. For Theorem~\ref{thm:full_support}, we instead assume $\bT^{2n} \setminus \supp \mu$ intersects $z + \R v \mod \Z^{2n}$ for all $z \in \bT^{2n}$ and $v \in E_+ \cup E_-$.

As outlined in \S\ref{subsection:proofoutline}, let $b_1, b_2 \in C^\infty(\bT^{2n})$ such that $b_1+b_2=1$ form a partition of unity on $\bT^{2n}$ and $\supp b_1 \cap \supp \mu = \emptyset$. We note that $b_1$ and $b_2$ are further specified in \S\ref{subsubsection:constructionofsymbols}  for Theorem~\ref{thm:coisotropic} and in \S\ref{subsection:Bwdecay1.4} for Theorem~\ref{thm:full_support}.

In the following exposition, we fix $j$ and omit it from our writing. Define $$B_1 \coloneqq \op_{N, \theta}(b_1) \quad \text{and} \quad B_2 \coloneqq \op_{N, \theta}(b_2).$$  Note that these operators form a quantum partition of unity: $B_1 + B_2=I$.  

For an operator $L: \cH_N(\theta) \rightarrow \cH_N(\theta)$, set 
\begin{equation*}
L(T) \coloneqq M_{N, \theta}^{-T} LM_{N, \theta}^{T}: \cH_N(\theta) \rightarrow \cH_N(\theta).
\end{equation*}

Now for $m \in \N$, define the set of words $\cW(m)=\{1,2\}^m$. We write an element of $\cW(m)$ as $\w=w_0 \cdots w_{m-1}$. Setting $\w \in \cW(M)$, let
\begin{equation*}
B_\w \coloneqq  B_{w_{m-1}}(m-1)\cdots B_{w_{1}}(1) B_{w_{0}}(0),
\end{equation*}
with corresponding principal symbol
\begin{equation*}
b_\w \coloneqq \prod_{k=0}^{m-1} b_{w_k} \circ A^k.
\end{equation*}

For a function $c: \cW(m) \rightarrow \R$, define the operator 
\begin{equation*}
B_c \coloneqq \sum_{\w \in \cW(m)} c(\w) B_\w,
\end{equation*}
with principal symbol 
\begin{equation*}
b_c \coloneqq \sum_{\w \in \cW(m)} c(\w) b_\w.
\end{equation*}  If $c=\1_E$ for $E \subset \cW(m)$, we use the notation $B_E$ to denote $B_{\1_E}$.

To specify the values of $m$ used in the rest of the paper, we first define the following constants. Set $\Lambda >1$ to be the absolute value of a particular eigenvalue of $A$, precisely chosen in~\eqref{eq:Lambda1} for Theorem~\ref{thm:coisotropic} and~\eqref{eq:Lambda2} for Theorem~\ref{thm:full_support}. Set 
\begin{equation}\label{eq:original_Lambda_def}
\Lambda_+ \coloneqq \max \{|\lambda| : \lambda \text{ is an eigenvalue of } A\}.
\end{equation}
The earlier works \cite{Sch} and \cite{dyatlov2021semiclassical}, used $\Lambda = \Lambda_+$.
However, to prove Theorem~\ref{thm:coisotropic}, we require more control over the precise stable/unstable behavior of $A$ on a particular subspace determined by $V$.

Pick 
$$\rho, \rho' \in (0, 1) \quad \text{such that} \quad \rho' \leq \rho \quad \text{and} \quad \rho + \rho' <1.$$ The particular choices of $\rho, \rho'$ are given in~\eqref{eq:rho_def} and~\eqref{eq:rho} for Theorem~\ref{thm:coisotropic} and~\ref{thm:full_support}, respectively. 
Then let $J$ be an integer such that 
\begin{equation}\label{eq:J_bounds}
J> \max\left\{\frac{5 \rho \log 2}{2 \log \Lambda}, \frac{25 \rho \log \Lambda_+}{\log \Lambda}\right\}.
\end{equation}
Finally, define 
\begin{equation}\label{eq:original_T_Def}
T_0 \coloneqq \lrfl{\frac{\rho \log N }{J \log \Lambda}} \quad \text{and} \quad  T_1 \coloneqq J T_0.
\end{equation}
In this paper, we will construct words of length  $T_0$ and $T_1$.

Define the following function $F:\cW(T_0) \rightarrow [0,1]$, which gives the proportion of the digit 1 in a word:
\begin{equation*}
F(\w) \coloneqq \frac{|\{k \in \{0,\ldots, T_0 -1\}: w_k=1\}|}{T_0},
\end{equation*}
for $\w = w_0 \cdots w_{T_0-1}$.
For $\alpha \in (0, \frac{1}{2})$, which we later select to be sufficiently small in~\eqref{eq:alpha_1.4}, with additional condition~\eqref{eq:alpha_1.3} for Theorem~\ref{thm:coisotropic}, let 
\begin{equation*}
\cZ = \{ \w \in \cW(T_0): F(\w) >\alpha\}.
\end{equation*}
We use $\cZ$ to split $\cW(2T_1)$ into the following two disjoint sets, where a word in $\cW(2T_1)$  is now written as a concatenation of $2J$ elements of $\cW(T_0)$. 
Specifically, set
$$\cY=\{\w_{(1)} \cdots\w_{(2J)} : \w_{(k)} \in \cZ \text{ for some } 1 \leq k \leq 2J \}$$
and
\begin{equation}\label{eq:def_X}
\cX=\cW(2T_1) \setminus \cY=\{\w_{(1)} \cdots \w_{(2J)} :  \w_{(k)} \in \cW(2T_1) \setminus \cZ \text{ for all } 1 \leq k \leq 2J\}.
\end{equation}
We call elements of $\cX$ \emph{uncontrolled long logarithmic words} and elements of $\cY$ \emph{controlled long logarithmic words}. To simplify notation, we set
$$\Hj \coloneqq \cH_{N_j}(\theta_j).$$

Noting that $B_\cY + B_\cX=1$, it suffices to show $\|B_\cY u_j\|_{\cH_j}, \|B_\cX u_j\|_{\cH_j} \rightarrow 0$ to find a contradiction. 

\subsection{Decay of $B_\cY$}\label{subsection:Y}
We begin with the decay of $\|B_\cY u_j\|_{\cH_j}$, the proof of which relies on the fact that $\supp b_1 \cap \supp \mu = \emptyset$.

\begin{lemma}\label{lem:Ydecay}
As $j \rightarrow \infty$, we have $$\|B_\cY u_j\|_{\Hj} \rightarrow 0.$$
\end{lemma}

To prove Lemma~\ref{lem:Ydecay}, we follow the proof outline from~\cite{dyatlov2021semiclassical}*{Lemma 3.1}, beginning with the following estimate for $B_1$.
\begin{lemma}\label{lem:B1bound}
Suppose that for $b_1 \in C^\infty(\bT^{2n})$, $\supp b_1 \cap \supp \mu = \emptyset$. Then for $B_1 = \op_{N_j, \theta_j}(b_1)$,  $$\|B_1 u_j\|_{\Hj} \rightarrow 0.$$
\end{lemma}

\begin{proof}
Recalling the definition of $S(1)$ from~\eqref{eq:def_S(1)}, we see that $b_1 \in S(1)$. Therefore by Lemma~\ref{lem:torussymbolproperties}, 
\begin{align*}
\left\|\op_{N_j, \theta_j}(b_1) u_j \right\|_{\Hj} &= \lrang{\op_{N_j, \theta_j}(b_1)^* \op_{N_j, \theta_j}(b_1)u_j, u_j}_{\Hj}\\
&=\lrang{\op_{N_j, \theta_j}(|b_1|^2) u_j, u_j}_{\Hj} + \cO\left(N_j^{-1}\right).
\end{align*}
As $u_j$ weakly converge to $\mu$, we know $$\lrang{\op_{N_j, \theta_j}\left(|b_1|^2\right) u_j, u_j}_{\Hj} \rightarrow \int_{\bT^{2n}} |b_1|^2 d\mu =0.$$ Therefore,  $\|\op_{N_j, \theta_j}(b_1) u_j \|_{\Hj} =o(1)$. 
\end{proof}

Recall the symbol class $S_\rho(1)$, defined in~\eqref{eq:S_rho(1)}.
In order to examine $B_\cY$, we use the following lemma, which shows that $b_\w$ lies in $S_{1/10}(1)$. We adapt the proof from~\cite{dyatlov2021semiclassical}*{Lemma 3.7}.

\begin{lemma}\label{lem:symbolclassofb_w}
For every $\w \in \cW(T_0)$, the symbol $b_\w$ belongs to the symbol class $S_{\frac{1}{10}}(1)$ and 
\begin{equation}\label{eq:symbol_formula}
B_\w = \op_{N, \theta}(b_\w) + \cO\left(N_j^{-\frac{7}{10}}\right)_{\cH_j \rightarrow \cH_j}.
\end{equation}
\end{lemma}

\begin{proof}
By Lemma~\ref{lem:manymultiplication}~\eqref{manyproduct}, to show $b_\w \in S_{1/10}(1)$, it suffices to show for $i=1,2$ and $0 \leq k \leq T_0$, each  $S_{\frac{1}{20}}(1)$-seminorm of $b_i \circ A^k$ is bounded uniformly in $k$ and $N$.  Once we prove that $b_\w \in S_{1/10}(1)$,~\eqref{eq:symbol_formula} follows from Lemma~\ref{lem:manymultiplication}~\eqref{manysymbol}.

Let $j_1, \ldots, j_m \in \{1, \ldots, 2n\}$. Using the fact that $A$ is a linear map, we calculate for $z \in \bT^{2n}$
$$\partial_{j_1} \cdots \partial_{j_m} \left(b_i \circ A^k\right)(z) = D^{m}b_i\left(A^k z\right) \cdot \left(A^k \partial_{j_1}, \ldots, A^k \partial_{j_m}\right),$$ where  $D^{m}b_i$ denotes the $m$-th derivative of $b_i$, an $m$-linear form, uniformly bounded in $N$. By Gelfand's formula, for all $\varepsilon >0$,  $\|A^k\| \leq \cO(\Lambda_+^{k(1+\varepsilon)})$ as $k \rightarrow \infty$. From~\eqref{eq:J_bounds}, we have
$$
T_0 \leq \frac{\log N}{25 \log \Lambda_+}.
$$
Choosing $\varepsilon$ sufficiently small,
$$
\sup_{\R^{2n}} \left|\partial_{j_1} \cdots \partial_{j_m} \left(b_i \circ A^k\right)\right| \leq C \left|A^k \partial_{j_1}\right| \cdots \left|A^k \partial_{j_m}\right| \leq C \Lambda_+^{km(1+\varepsilon)} \leq C \Lambda_+^{T_0 m(1+\varepsilon)} \leq CN^{\frac{m}{20}},
$$
which completes the proof. 
\end{proof}

We want to use Lemma~\ref{lem:B1bound} to bound $B_\cZ$. To do so, we require the following general lemma, which is adapted from~\cite{dyatlov2021semiclassical}*{Lemma 3.8}.

\begin{lemma}\label{lem:cd_bounds}
Let $c, d :\cW(T_0) \rightarrow \R$ with $0 \leq c(\w) \leq d(\w) \leq 1$ for all $\w \in \cW(T_0)$. Then for all $u \in \cH_j$, there exists some $\delta>0$ such that 
$$\|B_c u\|_{\Hj} \leq \|B_d u\|_{\Hj} +CN_j^{-\delta} \|u\|_{\cH_j}.$$
\end{lemma}

\begin{proof}
To simplify notation, we fix $j$ and omit it from our writing. 
First, note that $|\cW(T_0)|=2^{T_0} \leq N^{\frac{\rho\log 2}{J\log \Lambda}}$. From Lemma~\ref{lem:symbolclassofb_w}, for each $\w \in \cW(T_0)$, $b_\w \in S_{\frac{1}{10}}(1)$. Therefore, $N^{-\frac{\rho\log 2}{J\log \Lambda}} b_c$, $N^{-\frac{\rho\log 2}{J\log \Lambda}} b_d \in S_{\frac{1}{10}}(1)$. 

Recall that for any coisotropic $L$, $S_{\frac{1}{10}}(1) = S_{L, \frac{1}{10}, \frac{1}{10}}$. Therefore, we can apply Lemma~\ref{lem:torussymbolproperties}:
\begin{align*}
-&CN^{-\frac{4}{5}}\|u\|_\cH^2\\
&\leq \lrang{\op_{N, \theta} \left(N^{-\frac{2\rho\log 2}{J\log \Lambda}} \left(|b_d|^2-|b_c|^2 \right) \right)u,u}_\cH\\
&\leq N^{-\frac{2\rho\log 2}{J\log \Lambda}}\left(\lrang{\op_{N, \theta} \left(b_d\right)^*\op_{N, \theta} \left(b_d \right) u, u}_\cH -\lrang{\op_{N, \theta} \left(b_c\right)^*\op_{N, \theta} \left(b_c \right) u,u}_\cH \right)+CN^{-\frac{4}{5}}\|u\|^2_{\cH}\\
&=N^{-\frac{2\rho\log 2}{J\log \Lambda}}\left(\left\|\op_{N, \theta} \left(b_d\right) u\right\|^2_\cH -\left\|\op_{N, \theta} \left(b_c\right)u\right\|^2_\cH\right) +CN^{-\frac{4}{5}}\|u\|^2_{\cH}\\
&\leq N^{-\frac{2\rho\log 2}{J\log \Lambda}}\left(\left\|B_d u\right\|^2_\cH -\left\|B_c u\right\|^2_\cH\right) +CN^{-\frac{4}{5}}\|u\|^2_{\cH}.
\end{align*}

Therefore, $$\|B_c u\|_\cH^2 \leq \|B_d u\|_\cH^2 +CN^{-\frac{4}{5}}N^{\frac{2\rho\log 2}{J\log \Lambda}}\|u\|_\cH^2.$$
Recalling from~\eqref{eq:J_bounds} that $J> \frac{5 \rho \log 2}{2 \log \Lambda}$, we conclude $CN^{-\frac{4}{5}}N^{\frac{2\rho\log 2}{J\log \Lambda}}\leq CN^{-2\delta}$  for some $\delta>0$.
\end{proof}

To prove  Lemma~\ref{lem:Ydecay}, we will control $B_\cY$ by the behavior of $B_\cZ$ and the behavior of $B_{\cW(T_0)\setminus \cZ}$. Thus, we first bound $B_\cZ$ and $B_{\cW(T_0)\setminus \cZ}$ via Lemma~\ref{lem:B1bound}  and Lemma~\ref{lem:cd_bounds}. 

\begin{lemma}\label{lem:BZbound}
For some $\delta>0$, we have $$\|B_\cZ u_j \|_{\Hj} \rightarrow 0$$ and
$$\|B_{\cW(T_0)\setminus \cZ}\|_{\Hj \rightarrow \Hj} \leq 1 + CN_j^{-\delta}.$$
\end{lemma}

\begin{proof}
Set $B_F = \sum_{\w} F(\w) B_\w$, where we recall that
$$F(\w)=\frac{|\{k \in \{0,\ldots, T_0 -1\}: w_k=1\}|}{T_0}.$$
Then note that
$$B_F = \frac{1}{T_0} \sum_{\w \in \cW(T_0)} \left( \sum_{k=0}^{T_0-1} \1_{\{w_k=1\}} \right) B_\w= \frac{1}{T_0} \sum_{k=0}^{T_0-1} \sum_{\substack{\w \in \cW(T_0) \\ w_k=1}}B_\w= \frac{1}{T_0} \sum_{k=0}^{T_0-1} B_1(k).$$ Therefore,
$$\|B_F u_j\|_{\Hj} \leq \max_k \|B_1(k) u_j \|_{\Hj} = \|B_1 u_j\|_{\Hj}= o(1),$$
where the final equality follows from  Lemma~\ref{lem:B1bound}.

We now use our bounds on $B_F$ to bound $B_\cZ$. First, note that for all $\w \in \cW(T_0)$, we have $0 \leq \1_{\cZ}(\w) \leq \frac{F(\w)}{\alpha}$.
Therefore, by Lemma~\ref{lem:cd_bounds}, $\|B_\cZ u_j \|_{\Hj} \leq \frac{1}{\alpha} \|B_Fu_j\|_{\Hj} +CN_j^{-\delta} = o(1)$. 

We now turn our attention to $B_{\cW(T_0)\setminus \cZ}$.
As $\1_{\cW(T_0)\setminus \cZ} \leq \1_{\cW(T_0)}$, using Lemma~\ref{lem:cd_bounds} and the fact that $B_{\cW(T_0)}=1$, we have
$$\|B_{\cW(T_0)\setminus \cZ}\|_{\Hj \rightarrow \Hj} \leq 1 +CN_j^{-\delta},$$
thus completing the proof.
\end{proof}

We finally prove Lemma~\ref{lem:Ydecay}.
\begin{proof}[Proof of Lemma~\ref{lem:Ydecay}]
Using the fact that $B_{\cW(T_0)} =I$, we have that
$$B_\cY = \sum_{k=1}^{2J} M_{N_j, \theta_j}^{-(2J-1)T_0} \left(B_{\cW(T_0)\setminus \cZ} M^{T_0}_{N_j, \theta_j}\right)^{2J-k}B_\cZ M_{N_j, \theta_j}^{(k-1)T_0}.$$

Therefore, applying Lemma~\ref{lem:BZbound} for sufficiently large $N_j$,
\begin{align*}
\|B_\cY u_j\|_{\Hj} &\leq \sum_{k=1}^{2J} \left\| \left(B_{\cW(T_0)\setminus \cZ} M^{T_0}_{N_j, \theta_j}\right)^{2J-k}B_\cZ M_{N_j, \theta_j}^{(k-1)T_0} u_j\right\|_{\Hj}\\
&\leq \sum_{k=1}^{2J} \left\|B_{\cW(T_0)\setminus \cZ} \right\|_{\Hj \rightarrow \Hj}^{2J-k} \left\|B_\cZ M_{N_j, \theta_j}^{(k-1)T_0} u_j\right\|_{\Hj}\\
&\leq 4J\left\|B_\cZ  u_j\right\|_{\Hj}, 
\end{align*}
which decays to 0. 
\end{proof}

\subsection{Decay of $B_\cX$}\label{subsection:X}
Following the approach of~\cite{dyatlov2021semiclassical}, we first bound the size of $\cX$, then  show the decay of $\|B_\w\|_{\cH_j \rightarrow \cH_j}$ for $\w \in \cX$.

We recall the following lemma from~\cite{dyatlov2021semiclassical} and for the reader's convenience, recreate the authors' proof here.  
\begin{lemma}[\cite{dyatlov2021semiclassical}*{Lemma 3.13}]\label{lem:sizeofX}
There is a constant $c>0$ with no $\alpha$-dependence and a constant $C>0$ that may have $\alpha$-dependence such that for $N$ sufficiently large,
$$\# \cX \leq C(\log N)^{2J} N^{\frac{2c \rho \sqrt{\alpha}}{\log \Lambda}}.$$
\end{lemma}

\begin{proof}
As $\# \cX =\#(\cW(T_0) \setminus \cZ)^{2J}$, it suffices to bound $\#(\cW(T_0) \setminus \cZ)$. Recall that $\w=w_0 \cdots w_{T_0-1}$ is in $\cW(T_0) \setminus \cZ$ if and only if $\#\{j \in \{0, \ldots, T_0-1\}: w_j =1\} < \alpha T_0$. Thus, using the fact that $\alpha<1/2$ and $T_0 = \lrfl{\frac{\rho \log N}{J \log \Lambda}}$,
\begin{align}
\# (\cW(T_0) \setminus \cZ) &\leq \sum_{0 \leq l \leq \alpha T_0} \binom{T_0}{l} \leq (\alpha T_0 + 1) \binom{T_0}{\lrceil{\alpha T_0}} \nonumber \\
&\leq C \log N \exp(-(\alpha \log \alpha +(1-\alpha)\log(1-\alpha))T_0) \label{line1}\\
&\leq CN^{\frac{c \rho \sqrt{\alpha}}{J \log \Lambda}}\log N \label{line2},
\end{align}
where Stirling's formula gives~\eqref{line1} and $c>0$ in~\eqref{line2} is chosen so that 
$$-(\alpha \log \alpha +(1-\alpha)\log(1-\alpha)) \leq c \sqrt{\alpha},$$
for all $\alpha \in (0,1/2)$.
\end{proof}

The following is the key lemma for Theorem~\ref{thm:coisotropic} and Theorem~\ref{thm:full_support}, requiring different proofs under the respective conditions for each theorem. In both cases, the proof requires an uncertainty principle. We defer the proof under the conditions of Theorem~\ref{thm:coisotropic} to \S\ref{section:thm1.3} and the proof under the conditions of Theorem~\ref{thm:full_support} to \S\ref{section:thm1.4}.

\begin{lemma}\label{lem:Bw_decay}
Under the conditions of Theorem~\ref{thm:coisotropic} or those of Theorem~\ref{thm:full_support}, there exists $C, \beta >0$ such that for all $\w \in \cX$,
\begin{equation}\label{eq:Bw_bound}
\|B_\w\|_{\cH_j \rightarrow \cH_j} \leq Ch^\beta.
\end{equation}
\end{lemma}

Assuming Lemma~\ref{lem:Bw_decay}, we can finish the proofs of Theorem~\ref{thm:coisotropic} and Theorem~\ref{thm:full_support}.
Recall that $h=(2\pi N)^{-1}$. As $c$ from Lemma~\ref{lem:sizeofX} does not depend on $\alpha$, we can pick $\alpha$ sufficiently small such that
\begin{equation}\label{eq:alpha_1.4}
(\log N)^{2J} N^{\frac{2c \rho \sqrt{\alpha}}{\log \Lambda}}h^{\beta} \leq Ch^{\frac{\beta}{2}}. 
\end{equation} 
For Theorem~\ref{thm:coisotropic}, we also require that $\alpha$ satisfy~\eqref{eq:alpha_1.3}. This is due to the fact that we will split up $B_\w$ into an $\alpha$-dependent number of parts to gain more control over $\supp b_\w$. We want to bound this number of parts, which gives the constraint~\eqref{eq:alpha_1.3}.

Then, from Lemma~\ref{lem:sizeofX} and Lemma~\ref{lem:Bw_decay}, $$1 \leq \|B_\cY u_j\|_{\Hj} + \|B_\cX u_j\|_{\Hj}  \leq  \|B_\cY u_j\|_{\Hj} + \#\cX \cdot Ch^\beta \leq  \|B_\cY u_j\|_{\Hj} + Ch^{\frac\beta 2}.$$
From Lemma~\ref{lem:Ydecay}, the right-hand side decays to 0 as $h \rightarrow 0$, a contradiction. This immediately concludes the proof of Theorem~\ref{thm:coisotropic}.

Under the assumptions of Theorem~\ref{thm:full_support}, we now know for some $z \in \bT^{2n}$ and some $v \in E_+ \cup E_- \setminus \{0\}$, $z + \R v \subset \supp \mu$. Then by Lemma~\ref{lem:T_v}, $z + \bT_v \subset \supp \mu$. Finally, since $\mu$ is $A$-invariant, $\supp \mu$ must also contain the set $\overline{\{A^l (z + \bT_v) : l \in \Z\}}$, completing the proof of Theorem~\ref{thm:full_support}~\eqref{item:K}.

\section{Proof of Lemma~\ref{lem:Bw_decay} for Theorem~\ref{thm:coisotropic}}\label{section:thm1.3}
Recall that Theorem~\ref{thm:coisotropic} states if $A \in \Sp(2n,\Z)$ is hyperbolic and diagonalizable in $\C$ with a semiclassical measure supported in union of tori with rational tangent space $V$.  From \S\ref{section:groundwork}, it remains to prove Lemma~\ref{lem:Bw_decay}.

\subsection{Proof groundwork}\label{subsection:proof_groundwork}
We fix a hyperbolic $A \in \Sp(2n, \Z)$ that is diagonalizable in $\C$. We use this diagonalizability condition in Lemma~\ref{lem:W}, \eqref{eq:v_u_gelfand}, and \eqref{eq:spectralbounds}.  Let $M \in \cM_A$ be a metaplectic transformation quantizing $A$.
As $A$ is symplectic, its transpose is conjugate to its inverse. Thus, if $\lambda$ is an eigenvalue of $A$, $\lambda^{-1}$ is also an eigenvalue.

Assume that $V \subset \R^{2n}$ is a rational subspace which is invariant under $A$ and not coisotropic. Let $\Sigma$ be the union of finitely many subtori of $\bT^{2n}$ with tangent space $V$. 

Take sequences $N_j \in \N$  and $\theta_j \in \bT^{2n}$ such that $N_j \rightarrow \infty$ and $N_j, \theta_j$ satisfy the quantization condition~\eqref{eq:domainrange}. Set $h=(2\pi N)^{-1}$. Suppose towards a contradiction that $u_j \in \cH_{N_j}(\theta_j)$ are normalized eigenfunctions of $M_{N_j, \theta_j}$ that converge weakly to a semiclassical measure $\mu$ with $\supp \mu \subset \Sigma$.

\subsubsection{Construction of symplectic spaces}
We have assumed that $V$ is not coisotropic, which includes the case where $V$ is a
proper symplectic subspace of~$\mathbb R^{2n}$. As we later apply Darboux's theorem,
the easiest case to examine is when $V$ is symplectic. However, under our more general assumption, we can still reduce to the study of symplectic spaces. We construct symplectic $W$ in the following lemma and eventually use $W^{\perp \sigma}$ in lieu of $V$.

\begin{lemma}\label{lem:W}
There exists a nontrivial, symplectic, rational, $A$-invariant subspace $W\subset\mathbb R^{2n}$ such that $W \cap V = \{0\}$ and $W \subset V^{\perp \sigma}$. 
\end{lemma}

\begin{proof}
As $A$ is diagonalizable over $\C$, its minimal polynomial is square-free. Additionally, as $A$ has integer entries, its minimal polynomial has integer coefficients. Therefore, $A$ is semi-simple over $\Q$.
Note that both $V^{\perp\sigma}$ and $V\cap V^{\perp\sigma}$ are rational, $A$-invariant subspaces.
Therefore, there exists a rational, $A$-invariant subspace $W \subset V^{\perp \sigma}$ such that $W\cap V=\{0\}$ and $V^{\perp\sigma} = W \oplus (V\cap V^{\perp\sigma})$.
As $V$ is not coisotropic, $W$ is nontrivial. 

It remains to show that $W$ is symplectic, i.e. $W\cap W^{\perp\sigma}=\{0\}$.
Assume that $w\in W\cap W^{\perp\sigma}$. Then $\sigma(w,v)=0$
for every $v\in W$. Moreover, as $W\subset V^{\perp\sigma}$, we have
$\sigma(w,v)=0$ for every $v\in V$, in particular for every $v\in V\cap V^{\perp\sigma}$.
Since $V^{\perp\sigma}=W\oplus (V\cap V^{\perp\sigma})$, we see that
$\sigma(w,v)=0$ for every $v\in V^{\perp\sigma}$, and thus $w\in V$.
Together with $w\in W$, this implies that $w=0$. 
\end{proof}

At the start of this section, we assumed towards a contradiction that $\supp \mu \subset \Sigma$. We can write $\Sigma = \bigcup_{i=1}^S (x_i + V) \mod \Z^{2n}$ for some set of $x_i \in \R^{2n}$ and $S< \infty$. As $V \subset W^{\perp \sigma}$, we set 
\begin{equation}\label{eq:sigma'def}
\Sigma' \coloneqq  \bigcup_{i=1}^S (x_i + W^{\perp \sigma}) \bmod \Z^{2n}
\end{equation}
and note that
$$\supp \mu \subset \Sigma'.$$

Now set 
\begin{equation}\label{eq:Lambda1}
\Lambda \coloneqq \max \left\{|\lambda| \colon \lambda \text{ is an eigenvalue of } A|_W \right\}.
\end{equation}
In other words, $\Lambda$ is the spectral radius of $A|_{W}$. As $W$ is symplectic and $A$ is hyperbolic, $\Lambda>1$.

Let $W_\pm$ be the real part of the sum of eigenspaces corresponding to eigenvalues $\lambda$ of $A|_{W}$ with $|\lambda| = \Lambda^{\pm 1}$. Similarly, let $W_0$ be the real parts of the sum of eigenspaces corresponding to eigenvalues $\lambda$ of $A|_{W}$ with $\Lambda^{-1} < |\lambda| <\Lambda$.

We see that 
$$\R^{2n} = W^{\perp \sigma} \oplus W_+ \oplus W_- \oplus W_0.$$
Define 
$$d \coloneqq \dim W_+ =\dim W_-.$$ 
Note that if $v, u$ are eigenvectors of $A$ with respective eigenvalues $\lambda, \gamma$ and $\sigma(v, u) \neq 0$, then $\gamma = \lambda^{-1}$.  This follows from the fact that
$$\sigma(v, u) = \sigma(Av, Au) = \lambda \gamma \sigma(v, u).$$
Thus, we can find bases $\{u^\pm_1, \ldots, u^\pm_d\}$ for $W_\pm$ such that 
\begin{equation*}
\sigma(u_i^+, u_j^+)=0, \quad \sigma(u_i^-, u_j^-)=0, \quad \sigma(u_i^-, u_j^+) =\delta_{ij}. 
\end{equation*}

Note that $W_\pm$ are $A$-invariant. 
Since $A$ is diagonalizable, as $k \rightarrow \infty$, 
\begin{equation}\label{eq:v_u_gelfand}
\|A^k \mid_{W_-}\| = \cO\left(\Lambda^{-k}\right), \quad \|A^{-k} \mid_{W_+}\|= \cO\left(\Lambda^{-k}\right).
\end{equation}

Let $F_+$ be the sum of the eigenspaces corresponding to all eigenvalues $\lambda$ of $A$ with $|\lambda|>1$. Similarly, let $F_-$ be the sum of the eigenspaces corresponding to eigenvalues $\lambda$ with $|\lambda|<1$.

Set $L_\pm$ to be the real part of $F_\pm$, i.e., 
$$L_\pm \coloneqq \{v \in F_\pm : \overline{v} = v\}.$$ Note that $L_\pm$ are real, $A$-invariant vector spaces such that $L_+ \oplus L_-=\R^{2n}$. 
We call $L_+$ and $L_-$, respectively, the \emph{unstable} and \emph{stable subspaces}. 

As $A$ is diagonalizable, 
\begin{equation}\label{eq:spectralbounds}
\begin{split}
\|A^{k}\mid_{W}\| &= \cO\left(\Lambda^{k}\right) \quad \text{as } k \rightarrow \infty,\\
\|A^{k}\mid_{L_-}\| &= o(1) \quad \text{as } k \rightarrow \infty.
\end{split}
\end{equation}
Now suppose $u, v \in L_-$. For all $k$, $\sigma(u, v) = \sigma(A^k u, A^k v)$. As $\|A^k u\|, \|A^k v\| \rightarrow 0$, we see $\sigma(u, v) =0$. Therefore, $L_- \subset L_-^{\perp \sigma}$.

As $\dim L_- =n$, we deduce that $L_-$ is Lagrangian. A similar argument shows that $L_+$ is also Lagrangian. Note that $W_\pm\subset L_\pm$.

\subsubsection{Construction of symbols}\label{subsubsection:constructionofsymbols}
Recall that $W \subset \R^{2n}$ is a rational, symplectic, $A$-invariant subspace. $\Sigma'$ is the union of finitely many subtori of $\bT^{2n}$ with tangent space $W^{\perp \sigma}$. Let $\ell_0$ be the smallest distance between two different tori that make up $\Sigma'$. If $\Sigma'$ is only made up of one torus, set $\ell_0 =1$. 

Fix 
\begin{equation}\label{eq:rho_def}
\frac{1}{2} < \rho < 1, \quad \rho' = 0.
\end{equation}

Letting $\bT_{W^{\perp \sigma}}$ denote the subtorus $\{{W^{\perp \sigma}} \bmod (\Z^{2n} \cap {W^{\perp \sigma}})\}$, we examine the torus $\bT^{2n} / \bT_{W^{\perp \sigma}}$. Let $\pi_{W^{\perp \sigma}}: \bT^{2n} \rightarrow \bT^{2n} / \bT_{W^{\perp \sigma}}$ be the natural projection. From~\eqref{eq:sigma'def}, we note that 
$$\Sigma' = \bigcup_{i=1}^S \pi_{W^{\perp \sigma}}^{-1}(\pi_{W^{\perp \sigma}}(x_i)).$$ 
Define $\pi_\pm: \R^{2n} \rightarrow W_\pm$ to be the projection onto $W_\pm$ with kernel $W^{\perp \sigma} \oplus W_\mp \oplus W_0$.   Let 
\begin{equation*}
P \coloneqq \max\{\|\pi_+\|, \|\pi_-\|\}.
\end{equation*}
Recall that $\ell_0$ is the smallest distance between any of the tori that make up $\Sigma'$. Choose $r>0$ such that
\begin{equation}\label{eq:first_r_condition}
\begin{alignedat}{2}
&r < \frac{2 \ell_0}{\sqrt{n}(1+ \|A\|)} \quad \quad &\text{if} \quad (W^{\perp \sigma} + \Z^{2n}) \cap (W_+ \cup W_-) = \{0\},\\
&r< \min \left \{\frac{2 \ell_0}{\sqrt{n}(1+ \|A\|)}, \frac{D}{P \sqrt{n}(1 + \|A\|)}\right\}  \quad \quad  &\text{if} \quad (W^{\perp \sigma} + \Z^{2n}) \cap (W_+ \cup W_-) \neq \{0\},
\end{alignedat}
\end{equation}
where 
\begin{equation}\label{eq:D_def}
D \coloneqq \min \left\{|\Vec{j}|: \Vec{j} \in (W^{\perp \sigma} + \Z^{2n}) \cap (W_+ \cup W_-), \text{ } \vec{j} \neq 0\right\}.
\end{equation}

Let $\phi_1 + \phi_2 =1$ be a partition of unity on $\bT^{2n}/\bT_{W^{\perp \sigma}}$ such that $\phi_1=0$ in a neighborhood of $\pi_{W^{\perp \sigma}}(\Sigma')$ and $\phi_2$ is supported in $\pi_{W^{\perp \sigma}}(\Sigma') + B_{r\sqrt{n}/2} (0)$. We further split this partition of unity into 
$$\phi_1 = \tilde{\phi}_1 + \cdots + \tilde{\phi}_R \quad  \text{and} \quad \phi_2 = \tilde{\phi}_{R+1} + \cdots + \tilde{\phi}_{R+S},$$ such that
$$
\diam(\supp \tilde{\phi}_i) \leq r \sqrt{n} \text{ for all } i \quad \text{and} \quad \supp \tilde{\phi}_{R+i} \subset B_{r\sqrt{n}/2} (\pi_{W^{\perp \sigma}}(x_i)) \text{ for } 1 \leq i \leq S.
$$
Then define the functions on $\bT^{2n}$
$$b_j \coloneqq \pi_{W^{\perp \sigma}}^*\phi_j \quad \text{and} \quad \tilde{b}_j \coloneqq \pi_{W^{\perp \sigma}}^*\tilde{\phi}_j.$$
Intuition for these functions is given by  Figure~\ref{fig:support}.

Note that $\supp b_1 \cap \supp \mu = \emptyset$. We use $b_1, b_2$ to construct $B_\w$ from Lemma~\ref{lem:Bw_decay}. Importantly, $\w \in \cX$, where we recall the definition of $\cX$ from~\eqref{eq:def_X}.

\begin{figure}
\centering
\begin{subfigure}{\textwidth}
\begin{tikzpicture}[scale=.45]  
\begin{scope}[xshift=-1000]
\filldraw[color=blue!25] (0,0) rectangle (8,8);
\filldraw[color=white] (3.5,0) -- (8, 2.25) -- (8, 2.75) -- (2.5, 0) --(3.5,0);
\filldraw[color=white] (.5,0) -- (8, 3.75) -- (8, 4.25) -- (0, .25) -- (0,0) --(.5,0);
\filldraw[color=white] (0, 2.25) -- (8, 6.25) -- (8, 6.75) -- (0,2.75) -- (0, 2.25);
\filldraw[color=white] (0, 3.75) -- (8, 7.75) -- (8,8) -- (7.5, 8) -- (0,4.25) -- (0, 3.75);
\filldraw[color=white] (0, 6.25) -- (3.5,8) --(2.5,8) -- (0,6.75) --(0,6.25);
\draw[very thick] (3,0) -- (8, 2.5);
\draw[very thick] (0,0) -- (8,4);
\draw[very thick] (0,2.5) -- (8, 6.5);
\draw[very thick] (0, 4) -- (8,8);
\draw[very thick] (0, 6.5) -- (3,8);
\draw (0,0) rectangle (8,8);
\draw node[scale=1.2] at (4, -.75) {$\text{supp } b_1$};
\end{scope}
\begin{scope}[xshift=-750]
\filldraw[color=red!25] (3.5,0) -- (8, 2.25) -- (8, 2.75) -- (2.5, 0) --(3.5,0);
\filldraw[color=red!25] (.5,0) -- (8, 3.75) -- (8, 4.25) -- (0, .25) -- (0,0) --(.5,0);
\filldraw[color=red!25] (0, 2.25) -- (8, 6.25) -- (8, 6.75) -- (0,2.75) -- (0, 2.25);
\filldraw[color=red!25] (0, 3.75) -- (8, 7.75) -- (8,8) -- (7.5, 8) -- (0,4.25) -- (0, 3.75);
\filldraw[color=red!25] (0, 6.25) -- (3.5,8) --(2.5,8) -- (0,6.75) --(0,6.25);
\draw[very thick] (3,0) -- (8, 2.5);
\draw[very thick] (0,0) -- (8,4);
\draw[very thick] (0,2.5) -- (8, 6.5);
\draw[very thick] (0, 4) -- (8,8);
\draw[very thick] (0, 6.5) -- (3,8);
\draw (0,0) rectangle (8,8);
\draw node[scale=1.2] at (4, -.75) {$\text{supp } b_2$};
\end{scope}
\begin{scope}[xshift=-500]
\filldraw[color=blue!60] (0,5) -- (6,8) -- (5,8) -- (0, 5.5) --(0,5) ;
\filldraw[color=blue!60] (0,1) -- (8,5) -- (8,5.5) -- (0, 1.5) --(0,1) ;
\filldraw[color=blue!60] (6,0) -- (8, 1) -- (8,1.5) -- (5,0) --(6,0) ;
\draw[very thick] (3,0) -- (8, 2.5);
\draw[very thick] (0,0) -- (8,4);
\draw[very thick] (0,2.5) -- (8, 6.5);
\draw[very thick] (0, 4) -- (8,8);
\draw[very thick] (0, 6.5) -- (3,8);
\draw[thick, <->] (4, 2.95) -- (3.75, 3.45); 
\draw node[scale=.7] at (4, 2.7){$r\sqrt{n}$};
\draw (0,0) rectangle (8,8);
\draw node[scale=1.2] at (4, -.75) {$\text{supp } \tilde{b}_1$};
\end{scope}
\begin{scope}[xshift=-250]
\filldraw[color=red!75] (3.5,0) -- (8, 2.25) -- (8, 2.75) -- (2.5, 0) --(3.5,0);
\filldraw[color=red!75] (0, 2.25) -- (8, 6.25) -- (8, 6.75) -- (0,2.75) -- (0, 2.25);
\filldraw[color=red!75] (0, 6.25) -- (3.5,8) --(2.5,8) -- (0,6.75) --(0,6.25);
\draw[very thick] (3,0) -- (8, 2.5);
\draw[very thick] (0,0) -- (8,4);
\draw[very thick] (0,2.5) -- (8, 6.5);
\draw[very thick] (0, 4) -- (8,8);
\draw[very thick] (0, 6.5) -- (3,8);
\draw[thick, <->] (3.2,4.4) --(3.5, 3.93) ; \draw node[scale=.7] at (3.5, 3.7)  {$r\sqrt{n}$};
\draw (0,0) rectangle (8,8);
\draw node[scale=1.2] at (4, -.75) {$\text{supp } \tilde{b}_{R+1}$};
\end{scope}
\end{tikzpicture}
\caption{}
\end{subfigure}
\begin{subfigure}{\textwidth}
\centering
\begin{tikzpicture}[scale=.45]
\begin{scope}[xshift=-250]
\filldraw[color=red!75] (4,0) -- (8, 2) -- (8, 3) -- (2, 0) --(4,0);
\filldraw[color=red!75] (0, 2) -- (8, 6) -- (8, 7) -- (0,3) -- (0, 2);
\filldraw[color=red!75] (0, 6) -- (4,8) --(2,8) -- (0,7) --(0,6);
\draw[very thick] (3,0) -- (8, 2.5);
\draw[very thick] (0,0) -- (8,4);
\draw[very thick] (0,2.5) -- (8, 6.5);
\draw[very thick] (0, 4) -- (8,8);
\draw[very thick] (0, 6.5) -- (3,8);
\draw (0,0) rectangle (8,8);
\end{scope}
\begin{scope}
\filldraw[color=red!75] (1,0) -- (8, 3.5) -- (8, 4.5) -- (0, .5) -- (0,0) --(1,0);
\filldraw[color=red!75] (0, 3.5) -- (8, 7.5) -- (8,8) -- (7, 8) -- (0,4.5) -- (0, 3.5);
\draw[very thick] (3,0) -- (8, 2.5);
\draw[very thick] (0,0) -- (8,4);
\draw[very thick] (0,2.5) -- (8, 6.5);
\draw[very thick] (0, 4) -- (8,8);
\draw[very thick] (0, 6.5) -- (3,8);
\draw (0,0) rectangle (8,8);
\end{scope}
\end{tikzpicture}
\caption{}
\end{subfigure}
\caption{The colored regions of (A) depict of the support of $b_i$ and $\tilde{b}_i$. Here $\Sigma'$ is given by the black lines with $S=2$. The width of $\supp \tilde{b}_i$ is $r \sqrt{n}$. The illustration (B) demonstrates the action of $A$ on $\supp \tilde{b}_{R+1}$. Specifically, $A (\supp \tilde{b}_{R+1})$ must look like one of the two images.}
\label{fig:support}
\end{figure}

\subsubsection{Preliminary bound on $B_\w$} \label{subsubsection:preliminary_bound}
We write $\w$ as the following concatenation of words:
\begin{equation}\label{eq:w_decomp}
\w = 2_{s_1} 1_{r_1} 2_{s_2} 1_{r_2} \cdots 2_{s_K} 1_{r_K} 2_{s_{K+1}},
\end{equation}
where $1_{r_i} \in \{1\}^{r_i}$ and $2_{s_i} \in \{2\}^{s_i}$ with $r_i \geq 1$, $s_i \geq 0$. As $\w \in \cX$, we know $r_1 + \cdots + r_K \leq 2\alpha T_1$. Therefore,
\begin{equation}\label{eq:K_bound}
K \leq 2\alpha T_1.
\end{equation}

Using~\eqref{eq:w_decomp}, from the decompositions $b_1 =\tilde{b}_1 + \cdots + \tilde{b}_R$ and $b_2 = \tilde{b}_{R+1} + \cdots + \tilde{b}_{R+S}$, we can write the symbol of $B_\w$ as 
$$b_\w = \sum_{\z_m \in \bM} \tilde{b}_{\z_m},$$
for some set of words $\bM \subset \{1, \ldots, R+S\}^{2 T_1}$ and
$$
\z_m = \z_{(1)} \z_{(2)} \cdots \z_{(2K+1)} \quad \text{with} \quad  \z_{(2i-1)} \in \{R+1, \ldots, R+S\}^{s_i} \quad \text{and} \quad \z_{(2i)} \in \{1, \ldots, R\}^{r_i}.
$$

We use the notation $\z_{(2i-1)} \coloneqq \mathfrak{z}_0^{(2i-1)}\cdots \mathfrak{z}_{s_i-1}^{(2i-1)}$ and $\z_{(2i)} = \mathfrak{z}_0^{(2i)}\cdots \mathfrak{z}_{r_i-1}^{(2i)}$ and set
$$
\tilde{b}_{\z_{(2i-1)}} =\prod_{k=0}^{s_i-1} \tilde{b}_{\mathfrak{z}_k^{(2i-1)}} \circ A^k.$$

Split $\bM$ into two disjoint sets: $\bM = \bM_1 \cup \bM_2$, where
$$\bM_1 \coloneqq \{ \z_m \in \bM: \tilde{b}_{\z_{(2i-1)}}=0 \text{ for some } 1 \leq i \leq K+1\},$$
$$\bM_2 \coloneqq \{ \z_m \in \bM: \tilde{b}_{\z_{(2i-1)}} \neq 0 \text{ for all } 1 \leq i \leq K+1\}.$$

Using 
$$\tilde{b}_{\z_m} = \prod_{k=0}^{2T_1-1} \tilde{b}_{\mathfrak{z}_k} \circ A^k,$$
we set 
$$\tilde{B}_{\z_m} \coloneqq  \tilde{B}_{\mathfrak{z}_{2T_1-1}}(2T_1-1)\cdots \tilde{B}_{\mathfrak{z}_{1}}(1) \tilde{B}_{\mathfrak{z}_{0}}(0),$$
where $\tilde{B}_i = \op_{N_j, \theta_j}(\tilde{b}_i)$. Thus, $B_\w = \sum_{\z_m \in \bM_1} \tilde{B}_{\z_m}  + \sum_{\z_m \in \bM_2} \tilde{B}_{\z_m}$. To establish a preliminary bound for $\|B_\w\|$, we estimate each sum separately. We first focus on $\sum_{\z_m \in \bM_2} \tilde{B}_{\z_m}$.

\begin{lemma}\label{lem:M_2_bound}
We have the following bound:
\begin{equation*}
\left\|\sum_{\z_m \in \bM_2} \tilde{B}_{\z_m}\right\| \leq S N^{\frac{2 \alpha \rho \log SR}{\log \Lambda}} \max_{\z_m \in \bM_2} \|B_{\z_m}\|.
\end{equation*}
\end{lemma}

\begin{proof}
It suffices to show $\# \bM_2 \leq S N^{\frac{2 \alpha \rho \log SR}{\log \Lambda}}$. We fix $1 \leq i \leq K+1$ and examine $\tilde{b}_{\z_{(2i-1)}}$.
Note that for each $1 \leq k \leq S$, there exists a unique $1 \leq j_k \leq S$ such that $A(\pi_{W^{\perp \sigma}} (x_k))=\pi_{W^{\perp \sigma}}(x_{j_k})$. In other words, $A(\supp \tilde{\phi}_{R+k}) \cap \supp \tilde{\phi}_{R+j_k} \neq \emptyset$. See Figure~\ref{fig:support}.
We know
$$A(\supp \tilde{\phi}_{R+k'}) \subset A\left(B_{\frac{r\sqrt{n}}{2}}(\pi_{W^{\perp \sigma}}(x_{k'}))\right) \subset B_{\frac{\|A\| r\sqrt{n}}{2}}(\pi_{W^{\perp \sigma}}(x_{j_{k'}})).$$
By our choice of $r$, if $k \neq k'$, $A(\supp \tilde{\phi}_{R+k'}) \cap \supp \tilde{\phi}_{R+j_k} = \emptyset$.
Thus, for each $1 \leq k \leq S$, there exists exactly one $1 \leq j_k \leq S$ such that $\supp \tilde{b}_{R+ j_k}$ intersects $\supp \tilde{\phi}_{R+j_k}(\supp \tilde{b}_{R+k})$. Therefore, $\tilde{b}_{\z_{(2i-1)}}$ is nonzero only if 
$$\mathfrak{z}_1^{(2i-1)} = j_{\mathfrak{z}_0^{(2i-1)}}, \quad \mathfrak{z}_2^{(2i-1)} = j_{\mathfrak{z}_1^{(2i-1)}}, \quad \ldots, \quad \mathfrak{z}_{s_i-1}^{(2i-1)} = j_{\mathfrak{z}_{s_i-2}^{(2i-1)}}.$$ 
Consequently, for each choice of $\mathfrak{z}_0^{(2i-1)} \in \{R+1, \ldots, R+S\}$, there is only one way to pick $\mathfrak{z}_1^{(2i-1)}, \ldots, \mathfrak{z}_{s_i-1}^{(2i-1)}$ such that $\tilde{b}_{\z_{(2i-1)}}$ is nonzero. Then, there are $S$ nonzero choices of each $\tilde{b}_{\z_{(2i-1)}}$. Therefore, $\# \bM_2 \leq S^{K+1} R^{2 \alpha T_1}$. By~\eqref{eq:original_T_Def} and~\eqref{eq:K_bound},
\begin{equation*}
S^{K+1} R^{2 \alpha T_1} \leq S(SR)^{2\alpha T_1} \leq S N^{\frac{2 \alpha \rho \log SR}{\log \Lambda}},
\end{equation*}
which completes the proof.
\end{proof}

We now turn our attention to $\sum_{\z_m \in \bM_1} \tilde{B}_{\z_m}$.
\begin{lemma} \label{lem:zero_symbol}
We have the following bound:
$$\left\|\sum_{\z_m \in \bM_1} \tilde{B}_{\z_m}\right\|= \cO(N_j^{-\infty}).$$
\end{lemma}
\begin{proof}
Choose $\z_m \in \bM_1$. From the proof of Lemma~\ref{lem:M_2_bound}, we see that as $\z_m \in \bM_1$, there exists some $0 \leq k < 2T_1 -1$ such that
$$\supp (\tilde{b}_{\mathfrak{z}_k} \circ A^k) \cap \supp (\tilde{b}_{\mathfrak{z}_{k+1}} \circ A^{k+1})  = \emptyset.$$
Therefore, by the nonintersecting support property~\eqref{eq:nonintersecting}, 
$$\|\tilde{B}_{\mathfrak{z}_{k+1}}(k+1) \tilde{B}_{\mathfrak{z}_{k}}(k)\|_{\Hj \rightarrow \Hj}=\|\tilde{B}_{\mathfrak{z}_{k+1}}(1) \tilde{B}_{\mathfrak{z}_{k}}(0)\|_{\Hj \rightarrow \Hj} = \cO(N_j^{-\infty}).$$

From~\cite{dyatlov2021semiclassical}*{(2.45)}, viewing $a \in C^\infty(\bT^{2n})$ as a $\Z^{2n}$-periodic function in $C^\infty(\R^{2n})$, we have  
\begin{equation}\label{eq:change_norms}
\max_{\theta \in \bT^{2n}} \|\op_{N, \theta}(a) \|_{\cH_N(\theta) \rightarrow \cH_N(\theta)} = \|\op_h(a)\|_{L^2(\R^n) \rightarrow L^2(\R^n)},
\end{equation}
where $h=(2 \pi N)^{-1}$ and $\op_h$ is the semiclassical Weyl quantization, defined in~\eqref{eq:weylquant}.
Thus, using the fact that $M$ is unitary,
$$
\|\tilde{B}_{\mathfrak{z}_{i}}(i)\|_{\Hj \rightarrow \Hj} = \|\tilde{B}_{\mathfrak{z}_{i}}\|_{\Hj \rightarrow \Hj}  \leq \|\op_{h_j}(\tilde{b}_{\mathfrak{z}_{i}})\|_{L^2 \rightarrow L^2}.
$$

Therefore,
\begin{align*}
\|\tilde{B}_{\z_m}\|_{\Hj \rightarrow \Hj} &\leq \|\tilde{B}_{\mathfrak{z}_{2T_1-1}}(2T_1-1)\|_{\Hj \rightarrow \Hj} \cdots \|\tilde{B}_{\mathfrak{z}_{k+1}}(k+1) \tilde{B}_{\mathfrak{z}_{k}}(k)\|_{\Hj \rightarrow \Hj} \cdots\|\tilde{B}_{\mathfrak{z}_{0}}(0)\|_{\Hj \rightarrow \Hj}\\
&\leq \left(\max_{1 \leq i \leq R+S} \|\op_{h_j}(\tilde{b}_i)\|_{L^2 \rightarrow L^2} \right)^{2T_1-2} \cdot \cO(N_j^{-\infty}).
\end{align*}
By Lemma~\ref{lem:symbolproperties}~\eqref{bounded}, we know that $\max_i \|\op_{h_j}(\tilde{b}_i)\|_{L^2 \rightarrow L^2}$ is bounded by a constant. Thus, 
$$\left\|\sum_{\z_m \in \bM_1} \tilde{B}_{\z_m}\right\| \leq \#\bM_1  \cdot C^{2T_1-2}  \cdot \cO(N_j^{-\infty}) \leq  (C (R+S))^{2T_1} \cdot \cO(N_j^{-\infty}).$$
As $\frac{\rho \log N_j}{\log \Lambda} + J \geq T_1$, there exists some $m \in \N$ such that $(C (R+S))^{2T_1} = o(N_j^m)$. 
We conclude $\|\sum_{\z_m \in \bM_1} \tilde{B}_{\z_m}\|= \cO(N_j^{-\infty})$. 
\end{proof}

By Lemma~\ref{lem:M_2_bound} and Lemma~\ref{lem:zero_symbol}, we have 
\begin{equation}\label{eq:preliminary_Bw_bound}
\|B_\w\|_{\Hj \rightarrow \Hj} \leq  S N^{\frac{2 \alpha \rho \log SR}{\log \Lambda}} \max_{\z_m \in \bM_2} \|\tilde{B}_{\z_m}\|_{\Hj \rightarrow \Hj} + \cO(N_j^{-\infty}).
\end{equation}

\subsection{Proof of Lemma~\ref{lem:Bw_decay}}\label{subsection:bwdecaycoisotropic} 
As mentioned above, to prove Theorem~\ref{thm:coisotropic}, it remains to show Lemma~\ref{lem:Bw_decay}, that there exists $C$, $\beta>0$ such that for all $\w \in \cX \subset \cW(2T_1)$,
\begin{equation}\label{eq:end_bound}
\|B_\w \|_{\Hj \rightarrow \cH_j} \leq Ch^\beta.
\end{equation}

1. Suppose that $\w \in \cX$. From~\eqref{eq:preliminary_Bw_bound},
to show~\eqref{eq:end_bound}, it suffices to find  $C,\beta>0$ such that for each $\z_m$, $\|\tilde{B}_{\z_m}\|_{\Hj \rightarrow \Hj} \leq Ch^{2\beta}$. Indeed, as $h = (2 \pi N)^{-1}$, we can pick $\alpha$ sufficiently small so that in addition to~\eqref{eq:alpha_1.4}, $\alpha$ also satisfies
\begin{equation}\label{eq:alpha_1.3}
N^{\frac{2 \alpha \rho \log SR}{\log \Lambda}} \leq h^{-\beta}.
\end{equation}
Then~\eqref{eq:end_bound} follows from~\eqref{eq:preliminary_Bw_bound}.

2. We now rewrite $\tilde{B}_{\z_m}$ as the product of the quantizations of two symbols.

Write $\z_m = \z_+ \z_-$, where $\z_\pm$ are words of length $T_1$. 
Relabel $\z_+ = \mathfrak{z}_{T_1}^+ \cdots \mathfrak{z}_1^+$, $\z_- = \mathfrak{z}_{0}^- \cdots \mathfrak{z}_{T_1-1}^-$. 
We see that
\begin{equation}\label{eq:expansion}
\begin{aligned}
\tilde{B}_{\z_m} &= \tilde{B}_{\mathfrak{z}^-_{T_1 -1}}(2T_1-1) \cdots \tilde{B}_{\mathfrak{z}^-_0}(T_1)\tilde{B}_{\mathfrak{z}_1^+}(T_1 -1) \cdots \tilde{B}_{\mathfrak{z}_{T_1-1}^+}(0)\\& = M_{N_j, \theta}^{-T_1} \tilde{B}_{\z_-} \tilde{B}_{\z_+}(-T_1)M_{N_j, \theta}^{T_1}.
\end{aligned}
\end{equation}

Now set $$\tilde{b}_+ =\prod_{k=1}^{T_1} \tilde{b}_{\mathfrak{z}_k^+} \circ A^{-k} \quad \text{and} \quad \tilde{b}_- =\prod_{k=0}^{T_1-1} \tilde{b}_{\mathfrak{z}_k^-} \circ A^{k}.$$
To work  with $\tilde{b}_\pm$, we first show they lie in an appropriate symbol class. We adapt~\cite{dyatlov2021semiclassical}*{Lemma~3.7}.

\begin{lemma}\label{lem:bpm_symbol_class}
For all $\varepsilon>0$, $\tilde{b}_\pm \in S_{L_{\pm}, \rho + \varepsilon, \varepsilon}$ with bounds on the semi-norms that do not depend on $N$ or $\z_m$. Moreover, 
$$\tilde{B}_{\z_+}(-T_1)= \op_{N_j, \theta_j}(\tilde{b}_+) + \cO\left(N_j^{\rho +\varepsilon-1}\right)_{\Hj \rightarrow \Hj},$$
$$\tilde{B}_{\z_-}= \op_{N_j, \theta_j}(\tilde{b}_-) + \cO\left(N_j^{\rho +\varepsilon-1}\right)_{\Hj \rightarrow \Hj},$$
where the constants in $\cO(\cdot)$ are uniform in $N$ and $\z_m$.
\end{lemma}

\begin{proof}
We give the proof for $\tilde{b}_-$; the proof for $\tilde{b}_+$ follows similarly. By Lemma~\ref{lem:manymultiplication}, it suffices to show for $i=1, \ldots, R+S$ and $0 \leq k \leq T_1-1$, each  $S_{L_-, \rho, 0}$-seminorm of $\tilde{b}_i \circ A^k$ is bounded uniformly in $k$ and $N$.  

Each $\tilde{b}_i$ was constructed to be invariant along ${W^{\perp \sigma}}$. Therefore, if $X$ is a constant vector field on $\R^n$ that is tangent to ${W^{\perp \sigma}}$, then $X \tilde{b}_i =0$. As ${W^{\perp \sigma}}$ is symplectic, it is enough to bound $X_1 \cdots X_l Y_1 \cdots Y_m (\tilde{b}_i \circ A^k)$, where $X_j$, $Y_j$ are constant vector fields on $\R^{2n}$ tangent to $W$ with  $Y_j$ also tangent to $L_-$.

Using the fact that $A$ is a linear map, we calculate 
$$X_1\cdots X_l Y_1 \cdots Y_m \left(\tilde{b}_i \circ A^k\right)(z) = D^{l+m}\tilde{b}_i\left(A^k z\right) \cdot \left(A^k X_1, \ldots, A^k X_l, A^k Y_1, \ldots, A^k Y_m\right),$$ where  $D^{l+m}\tilde{b}_i$ denotes the $(l + m)$-th derivative of $\tilde{b}_i$, an $(l + m)$-linear form uniformly bounded in $N$. Therefore, using the bounds on powers of $A$ given by~\eqref{eq:spectralbounds},
\begin{align*}
\sup_{\R^{2n}} \left|X_1\cdots X_l Y_1 \cdots Y_m \left(\tilde{b}_i \circ A^k\right)\right| &\leq C \left|A^kX_1\right| \cdots \left|A^k X_l\right|\left|A^k Y_1\right|\cdots \left|A^k Y_m\right|\\
&\leq C \Lambda^{kl}\leq C \Lambda^{T_1 l} \leq CN^{l\rho }.
\end{align*}
The final inequality follows from~\eqref{eq:original_T_Def}, the definition of $T_1$.
\end{proof}

From~\eqref{eq:expansion} and Lemma~\ref{lem:bpm_symbol_class}, 
\begin{equation*}
\left\|\tilde{B}_{\z_m} \right\|_{\Hj \rightarrow \Hj} \leq \left\|\op_{N_j, \theta_j}(\tilde{b}_-)\op_{N_j, \theta_j}(\tilde{b}_+)\right\|_{\Hj \rightarrow \Hj} + CN_j^{\rho +\varepsilon-1}.
\end{equation*}

3. We now reduce the estimate of $\tilde{B}_{\z_m}$ to an estimate on operators on $\R^n$ using the Cotlar--Stein Theorem~\cite{z12semiclassical}*{Theorem C.5}. Similarly to~\eqref{eq:change_norms}, viewing 
$\tilde{b}_\pm$ as $\Z^{2n}$-periodic functions in $C^\infty(\R^{2n})$, we know
\begin{equation}\label{eq:B_w_split}
\left\|\tilde{B}_{\z_m} \right\|_{\Hj \rightarrow \Hj} \leq \left\|\op_h(\tilde{b}_-)\op_h(\tilde{b}_+)\right\|_{L^2(\R^n) \rightarrow L^2(\R^n)} + Ch^{1-\rho -\varepsilon}.
\end{equation}

We want show $\|\op_h(\tilde{b}_-)\op_h(\tilde{b}_+)\|_{L^2 \rightarrow L^2} \leq Ch^{\gamma}$ for some $C, \gamma >0$.
We begin by constructing a partition of unity. Fix $\psi(z) \in C^\infty(\R^{2n})$ to be compactly supported on $|z| \leq \sqrt{n}$ such that $\sum_{k \in \Z^{2n}} \psi(z-k)^2=1$ for all $z \in \R^{2n}$. Recall the definition of $r$ from~\eqref{eq:first_r_condition} and set 
$$\psi_k(z) \coloneqq \psi\left(\frac{2z}{r}-k\right).$$ 
Note that $\psi_k \in  S_{L, 0, 0}(\R^{2n})$ uniformly in $k$ for any coisotropic $L \subset \R^{2n}$. Now we use $\psi_k$ to split up $\op_h(\tilde{b}_-) \op_h(\tilde{b}_+)$, setting 
$$P_k \coloneqq \op_h(\tilde{b}_-) \op_h(\psi_k^2) \op_h (\tilde{b}_+).$$
We see that $\op_h(\tilde{b}_-) \op_h(\tilde{b}_+) = \sum_{k \in \Z^{2n}} P_k$, where the series converges in the strong operator topology as an operator $L^2(\R^n) \rightarrow L^2(\R^n)$.  

Via the Cotlar--Stein Theorem, to show $\|\op_h(\tilde{b}_-)\op_h(\tilde{b}_+)\|_{L^2 \rightarrow L^2} \leq Ch^{\gamma}$,  it suffices to show that 
\begin{equation}\label{eq:Cotlar_Stein}
\sup_{k \in \Z^{2n}} \sum_{l \in \Z^{2n}} \|P^*_k P_l\|^{\frac{1}{2}}_{L^2 \rightarrow L^2} \leq Ch^{\gamma} \quad \text{and} \quad \sup_{k \in \Z^{2n}} \sum_{l \in \Z^{2n}} \|P_k P_l^*\|^{\frac{1}{2}}_{L^2 \rightarrow L^2} \leq Ch^{\gamma}.
\end{equation}

The proof of~\eqref{eq:Cotlar_Stein} follows from the next two lemmas. 

\begin{lemma}[\cite{dyatlov2021semiclassical}*{Lemma 4.5}]\label{lem:close_together}
For every $m>0$, there exists a constant $C_m$ such that for all $k,l \in \Z^{2n}$ with $|k-l| \geq 10\sqrt{n}$, we have
$$\|P^*_k P_l\|^{\frac{1}{2}}_{L^2 \rightarrow L^2} \leq C_m h^m |k-l|^{-m} \quad \text{and} \quad \|P_k P_l^*\|^{\frac{1}{2}}_{L^2 \rightarrow L^2} \leq C_m h^m |k-l|^{-m}.$$
\end{lemma}

\begin{lemma}\label{lem:P_k}
There exists constants $C, \delta' >0$  such that for every $k \in \Z^{2n}$,  \begin{equation}\label{eq:Pk_decay}
\|P_k\|_{L^2 \rightarrow L^2} \leq Ch^{\delta'}.
\end{equation}
\end{lemma}

The proof of Lemma~\ref{lem:close_together} exploits the fact that the supports of $\psi_k$ and $\psi_l$ are sufficiently disjoint and requires no modifications from its original argument in~\cite{dyatlov2021semiclassical}. 

However, in Lemma~\ref{lem:P_k}, although the statement is the same as~\cite{dyatlov2021semiclassical}*{Lemma 4.6}, the proof is different. In~\cite{dyatlov2021semiclassical}, the authors assume a specific condition on the complements of the supports of $b_1$ and $b_2$. The authors use this condition to establish the porosity required to apply the 1-dimensional fractal uncertainty principle. However, in our argument, we picked $b_2$ to be supported in a neighborhood of $\Sigma'$. Thus, we cannot assume the same condition on $b_1$ and $b_2$ as in~\cite{dyatlov2021semiclassical}. As a result, we use a different uncertainty argument to show decay. 

4. We show Lemma~\ref{lem:P_k}.

\begin{proof}[Proof of Lemma~\ref{lem:P_k}]
1. Recall that each $\psi_k$ belongs to $S_{L_\pm, 0, 0}$ uniformly in $k$, while $\tilde{b}_\pm$ are in the larger symbol class $S_{L_\pm, \rho + \varepsilon, \varepsilon}$. Then $\tilde{b}_\pm \psi_k \in S_{L_\pm, \rho + \varepsilon, \varepsilon}(\R^{2n})$, so by Lemma~\ref{lem:symbolproperties}, 
$$P_k = \op_h(\tilde{b}_+ \psi_k) \op_h(\tilde{b}_- \psi_k) + \cO(h^{1-\rho  - 2\varepsilon})_{L^2 \rightarrow L^2}.$$  
We will show~\eqref{eq:Pk_decay}  via Lemma~\ref{lem:aandbbounds}, so we first prove the supports of $\tilde{b}_\pm \psi_k$ are sufficiently small.
We first focus on $\tilde{b}_- \psi_k$. Since $\supp \psi_k \subset B_{\frac{r\sqrt{n}}{2}}(\frac{rk}{2})$, we have 
\begin{equation}\label{eq:b_tilde_support}
\supp \tilde{b}_- \psi_k \subset B_{\frac{r\sqrt{n}}{2}}\left(\frac{rk}{2}\right)
\cap \left(\bigcap_{l=0}^{T_1-1} A^{-l} \supp \tilde{b}_{\mathfrak{z}_l^-} \right).
\end{equation}

Recall that $\tilde{b}_j \coloneqq \pi_{W^{\perp \sigma}}^* \tilde{\phi}_j$, where $\pi_{W^{\perp \sigma}}: \bT^{2n} \rightarrow \bT^{2n}/ \bT_{W^{\perp \sigma}}$ is the natural projection and $\bT_{W^{\perp \sigma}} = {W^{\perp \sigma}} / (\Z^{2n} \cap {W^{\perp \sigma}})$.
We have $\bT^{2n} / \bT_{W^{\perp \sigma}} = \R^{2n} /({W^{\perp \sigma}}+ \Z^{2n}) = (\R^{2n} / {W^{\perp \sigma}})/ \Gamma$, where $\Gamma \coloneqq ({W^{\perp \sigma}}+ \Z^{2n})/{W^{\perp \sigma}}$. As $W$ is transverse to ${W^{\perp \sigma}}$, $W$ intersects each element of $\Gamma$ exactly once.  
We can thus view each $\tilde{\phi}_i \in C^{\infty}(W)$ as a $(({W^{\perp \sigma}}+ \Z^{2n}) \cap W)$-periodic function given by the projection of $\tilde{b}_j$ onto $W$ with kernel ${W^{\perp \sigma}}$. Note that   
\begin{equation}\label{eq:diam_supp}
\diam \left(\supp \tilde{\phi}_i \bmod ((W^{\perp \sigma}+ \Z^{2n})\cap W)\right) \leq r\sqrt{n}.
\end{equation}

Recall that $\pi_\pm: \R^{2n} \rightarrow W_\pm$ is the projection onto $W_\pm$ with kernel $W^{\perp \sigma} \oplus W_\mp \oplus W_0$ and $P = \max\{\|\pi_+\|, \|\pi_-\|\}$. Set $\tilde{\pi}_\pm \coloneqq \pi_\pm |_{W}$. Then from~\eqref{eq:b_tilde_support},
$$\pi_+ \left(\supp \left(\tilde{b}_- \psi_k\right)\right) \subset B_{\frac{Pr\sqrt{n}}{2}}\left(\pi_+\left(\frac{rk}{2}\right)\right)
\cap \left(\bigcap_{l=0}^{T_1-1} A^{-l} \tilde{\pi}_+ \left(\supp \tilde{\phi}_{\mathfrak{z}_l^-}\right)\right).$$

2. Let $\omega_1, \omega_2 \in  \pi_+(\supp (\tilde{b}_- \psi_k))$. We show by induction that $|A^l \omega_1 - A^l \omega_2 | \leq Pr\sqrt{n}$ for each $0 \leq l \leq T_1 -1$. Clearly this holds for $l=0$. Now suppose $|A^l \omega_1 - A^l \omega_2 | \leq Pr\sqrt{n}$ for some $0 \leq l < T_1 -1$. We know that $A^{l+1} \omega_1, A^{l+1} \omega_2 \in \tilde{\pi}_+(\supp \tilde{\phi}_{\mathfrak{z}_{l+1}^-})$. From~\eqref{eq:diam_supp}, there exists a $\Vec{j} \in (W^{\perp \sigma}+ \Z^{2n})\cap W_+$ such that 
$$|A^{l+1} \omega_1 - A^{l+1} \omega_2 - \Vec{j} | \leq Pr \sqrt{n}.$$
If $(W^{\perp \sigma}+ \Z^{2n})\cap W_+= \{0\}$, then clearly $|A^{l+1} \omega_1 - A^{l+1} \omega_2| \leq Pr \sqrt{n}$.
Else, by the choice of $r$ in~\eqref{eq:first_r_condition}, we have 
$$|\Vec{j}| \leq Pr \sqrt{n} + |A^{l+1} \omega_1 - A^{l+1} \omega_2| \leq Pr \sqrt{n}(1+ \|A\|)<D.$$ 
From~\eqref{eq:D_def}, $\Vec{j}=0$, which gives $|A^l \omega_1 - A^l \omega_2| \leq Pr \sqrt{n}$. 

From~\eqref{eq:v_u_gelfand}, there exists some $C_0>0$ such that for all $0 \leq l \leq T-1$, we have $|\omega_1 -\omega_2| \leq C_0 \Lambda^{-l} Pr\sqrt{n}$.
Importantly, 
\begin{equation}\label{eq:z_bounds}
|\omega_1 - \omega_2| \leq C_0 \Lambda^{-T_1 + 1} Pr \sqrt{n},
\end{equation}
for any $\omega_1, \omega_2 \in \pi_+(\supp (\tilde{b}_- \psi_k))$.

By a similar argument, we can conclude 
for $\omega_1, \omega_2 \in  \pi_- (\supp (\tilde{b}_+ \psi_k))$,
$$|\omega_1 - \omega_2| \leq C_0 \Lambda^{-T_1 + 1} Pr \sqrt{n}.$$

3. Recall that  $\{u^\pm_1, \ldots, u^\pm_d\}$ are bases for $W_\pm \subset L_\pm$ such that $\sigma(u^\pm_i, u^\pm_j) =0$ and $\sigma(u^-_i, u^+_j) = \delta_{ij}$. 
Thus, using the linear version of Darboux's theorem, there exists a symplectic matrix $Q \in \Sp(2n, \R)$ such that
\begin{itemize}
    \item $Q \partial_{x_i} = u_i^+$ for $1 \leq i \leq d$;
    \item $Q \partial_{\xi_i} = u_i^-$ for $1 \leq i \leq d$; 
    \item $Q \spn ( \partial_{x_{d+1}}, \ldots, \partial_{x_n}, \partial_{\xi_{d+1}}, \ldots, \partial_{\xi_n}) = W^{\perp \sigma} \oplus W_0$.
\end{itemize}
Let $M_Q$ denote a metaplectic transformation such that $M^{-1}_Q \op_h(a) M_Q = \op_h(a \circ Q)$ for all $a \in S(1)$.

For $x \in \R^n$, we use the notation $x = (x', x'')$, where $x' \in \R^d$ and $x'' \in \R^{n-d}$. Then $L'_\pm \coloneqq Q L_\pm$ are Lagrangian subspaces with $L'_+ \subset \{\xi' =0\}$ and $L'_- \subset \{x' =0\}$ such that $\tilde{b}_\pm \psi_k \circ Q \in S_{L'_\pm, \rho + \varepsilon, \varepsilon}$.

Thus by~\eqref{eq:z_bounds}, letting $\pi_{x'}:\R^{2n} \rightarrow \R^d$ denote the orthogonal projection onto $x'$, for $\omega_1, \omega_2 \in \supp \tilde{b}_- \psi_k$, 
$$|\pi_{x'}(Q \omega_1) - \pi_{x'}(Q \omega_2) | < C_0 \Lambda^{-T_1+1}  Pr \sqrt{n}.$$

Since $T_1 \geq \frac{\rho \log N}{\log \Lambda}-1$, we have  $\Lambda^{-T_1 +1} \leq \Lambda^2(2\pi h)^{\rho}.$ Therefore, $$|\pi_{x'}(Q \omega_1) - \pi_{x'}(Q \omega_2) | < \left(C_0 \Lambda^2 (2 \pi)^\rho 
 P r \sqrt{n}\right) h^\rho.$$ We conclude that there exists some $x'_0 \in \R^d$ such that for $h$ sufficiently small
 $$\supp \tilde{b}_- \psi_k \circ Q \subset \left\{|x' - x'_0| \leq h^{\rho -\varepsilon}\right\}.$$ A similar argument shows that there exists some $\xi'_0 \in \R^d$ such that $$\supp \tilde{b}_+ \psi_k \circ Q   \subset \left\{|\xi' - \xi'_0| \leq h^{\rho -\varepsilon}\right\}.$$

4. We use the fact that $M_{Q}$ is unitary and apply Lemma~\ref{lem:aandbbounds} and the exact Egorov's Theorem~\eqref{eq:MA} to know that
\begin{align*}
\|\op_h(\tilde{b}_+\psi_k)\op_h(\tilde{b}_- \psi_k)\|_{L^2 \rightarrow L^2} &= \|M_Q^{-1}\op_h(\tilde{b}_+\psi_k)M_Q M_Q^{-1}\op_h(\tilde{b}_- \psi_k)M_Q\|_{L^2 \rightarrow L^2}\\
&=\|\op_h(\tilde{b}_+ \psi_k \circ Q)\op_h(\tilde{b}_- \psi_k \circ Q)\|_{L^2 \rightarrow L^2}\\
&\leq \cO(h^{\delta'}),
\end{align*}
for $\delta' = \min\left\{1- \rho-2\varepsilon, d\left( \rho-1/2 -\varepsilon\right)\right\}$. By our choice of $1/2 < \rho <1$, for $\varepsilon$ sufficiently small, $\delta' >0$.

We conclude that $\|P_k\|_{L^2 \rightarrow L^2} \leq \cO(h^{\delta'})$.
\end{proof}

5. Using Lemma~\ref{lem:close_together} and Lemma~\ref{lem:P_k} to apply the Cotlar--Stein Theorem, we know \linebreak
$\|\op_h(\tilde{b}_+) \op_h(\tilde{b}_-)\|_{L^2 \rightarrow L^2} \leq Ch^{\gamma}$ for some $\gamma>0$.
Then by~\eqref{eq:B_w_split},  $\|\tilde{B}_{\z_m}\|_{\Hj \rightarrow \Hj} \leq Ch^{2\beta}$ for $\beta = \frac{1}{2} \min \{\gamma, 1- \rho - \varepsilon\}$.

Thus, from our choice of $\alpha$ in~\eqref{eq:alpha_1.3} and from the inequality~\eqref{eq:preliminary_Bw_bound}, we have 
$$\left\|B_\w\right\|_{\Hj \rightarrow \Hj} \leq \sum_{\z_m \in \bM}\left\|\tilde{B}_{\z_m} \right\|_{\Hj \rightarrow \Hj} \leq  C  N^{\frac{2 \alpha \rho \log SR}{\log \Lambda}} h^{2\beta} + \cO(h^\infty) \leq C h^\beta,$$
which finishes our proof of Lemma~\ref{lem:Bw_decay} and therefore Theorem~\ref{thm:coisotropic}.

\section{Proof of Lemma~\ref{lem:Bw_decay} for Theorem~\ref{thm:full_support}~\eqref{item:K}}\label{section:thm1.4} 
From \S\ref{section:groundwork}, to prove Theorem~\ref{thm:full_support}~\eqref{item:K}, it remains to prove Lemma~\ref{lem:Bw_decay}.
As previously noted, our proof strategy generalizes that of~\cite{dyatlov2021semiclassical}.  

The differences between the two  argument stem from the fact that we do not assume a spectral gap. In ~\cite{dyatlov2021semiclassical}, the authors use the eigenspaces of the unique largest and smallest eigenvalues, each of which is 1-dimensional. The analogue in our proof is the sum of the generalized eigenspaces of the largest and smallest eigenvalues. As these spaces can be higher-dimensional, we prove multidimensional versions of the results from~\cite{dyatlov2021semiclassical}. Our argument picks up some additional small technicalities as we must consider generalized eigenvectors and complex eigenvectors. 
 
\subsection{Theorem~\ref{thm:full_support} setting}

Fix $A \in \Sp(2n, \Z)$ such that $A$ has a non-unit length eigenvalue. We label the eigenvalues of $A$ as
$\lambda_1, \ldots, \lambda_n, \lambda_1^{-1}, \ldots,  \lambda_n^{-1}$ with
$|\lambda_1| \geq \cdots \geq |\lambda_n| \geq 1$. Suppose we have
$$|\lambda_1| = |\lambda_2| = \cdots = |\lambda_l| > |\lambda_{l+1}|.$$ 
Let $F_+$ be the sum of generalized eigenspaces of~$A$ corresponding to $\lambda_1, \ldots, \lambda_l$,  $F_-$  the sum of eigenspaces corresponding to $\lambda_1^{-1}, \ldots, \lambda_l^{-1}$, and  $F_0$ the sum of generalized eigenspaces of~$A$ corresponding to $\lambda_{l+1}, \ldots, \lambda_n, \lambda_{l+1}^{-1}, \ldots, \lambda_n^{-1}$.
Now set $E_+$, $E_-$, and $E_0$ to be the real part of these spaces, i.e., 
$$E_\pm \coloneqq \{v \in F_\pm : \overline{v} = v\}, \quad E_0 \coloneqq \{v \in F_0 : \overline{v} = v\}.$$
We further assume that $A|_{E_\pm}$ is diagonalizable over $\C$. This assumption will later control the growth of elements of $E_\pm$ under the action of $A$. It also implies that $F_\pm$ are the sums of standard eigenspaces.

Let 
\begin{equation}\label{eq:L_pm_def}
L_\pm \coloneqq E_\pm \oplus E_0
\end{equation}
and note that $\R^{2n} = E_+ \oplus L_- = E_- \oplus L_+$.
Define 
\begin{equation}\label{eq:Lambda2}
\Lambda \coloneqq |\lambda_1| \quad \text{and} \quad \gamma  \coloneqq |\lambda_{l+1}|.
\end{equation}
If $l = n$, instead set $\gamma \coloneqq \Lambda^{-1}$.

Take sequences $N_j \in \N$  and $\theta_j \in \bT^{2n}$ such that $N_j \rightarrow \infty$ and $N_j, \theta_j$ satisfy the quantization condition~\eqref{eq:domainrange}. Set $h=(2\pi N)^{-1}$. Let $u_j \in \cH_{N_j}(\theta_j)$ be a sequence of normalized eigenfunctions of $M_{N_j, \theta_j}$ that weakly converges to the semiclassical measure $\mu$.

Then fix two numbers $\rho, \rho' \in (0,1)$ such that 
\begin{equation}\label{eq:rho}
\rho + \rho' <1, \quad \rho \frac{\log \gamma}{\log \Lambda} < \rho' < \frac{1}{2} < \rho.
\end{equation}
We set $\Lambda_+ \coloneqq \Lambda$ and use $\Lambda$, $\Lambda_+$,  
and $\rho$ to define $J$, $T_0$, and $T_1$ given by~\eqref{eq:J_bounds} and~\eqref{eq:original_T_Def}.

Using Gelfand's formula, for  $\varepsilon >0$ and $k$ sufficiently large depending on $\varepsilon$,
\begin{equation}\label{eq:Gelfand}
\|A^k\| \leq  \Lambda^{k(1+\varepsilon)}, \quad \|A^k|_{L_-}\| \leq \gamma^{k(1+\varepsilon)}. 
\end{equation}
As $A^{-1}$ has the same eigenvalues as $A$, we also know for $k$ sufficiently large 
\begin{equation*}
\|A^{-k}\|  \leq \Lambda^{k(1+\varepsilon)}, \quad  \|A^{-k}|_{L_+}\| \leq \gamma^{k(1+\varepsilon)}. 
\end{equation*}

We need the following lemma to prove symplectic properties of $E_\pm$ and $L_\pm$.

\begin{lemma}\label{lem:inverse_eigenvalues}
Suppose $v$ is an eigenvector with eigenvalue $\lambda_v$ and $u$ is a generalized eigenvector with eigenvalue $\lambda_u$. If $\sigma(v, u) \neq 0$, then $\lambda_u = \lambda_v^{-1}$.   
\end{lemma}
\begin{proof}
First, suppose $v$ and $u$ are standard eigenvectors. As $A$ is symplectic,
$$\sigma(v, u) = \sigma(Av, Au) = \lambda_v \lambda_u \sigma(v, u).$$
Therefore, $\lambda_u = \lambda_v^{-1}$. 
More generally, suppose $u$ is a generalized eigenvector of rank $\beta +1$. Then there exists generalized eigenvectors $u_1, \ldots, u_\beta$ such that
$$(A - \lambda_u I) u = u_1, \quad (A - \lambda_u I) u_1 = u_2, \quad \ldots, \quad (A - \lambda_u I) u_{\beta-1} = u_\beta, \quad (A- \lambda_v I) u_\beta =0.$$
We see
$$\sigma(v, u) = \sigma(Av, Au)= \lambda_v \sigma(v, u_1) + \lambda_v \lambda_u \sigma(v, u).$$
If $\lambda_v \sigma(v, u_1) =0$, we have $\lambda_u = \lambda_v^{-1}$. If not, we repeat the above argument with $\sigma(v, u)$ replaced by $\sigma(v, u_1)$. We continue to run this argument, stopping and concluding $\lambda_u = \lambda_v^{-1}$ if possible. If not, then we eventually conclude $\sigma(v, u_\beta) \neq 0$. As $u_\beta$ is a standard eigenvector, we have $\lambda_u = \lambda_v^{-1}$.
\end{proof}
Following from the definitions of $E_\pm$ and $E_0$ and from Lemma~\ref{lem:inverse_eigenvalues}, we see that $E_\pm \subset E_0^{\perp \sigma}$ and $E_0 \subset E_\pm^{\perp \sigma}$. Thus, $L_\pm^{\perp \sigma} \subset L_\pm$; in other words, $L_\pm$ are coisotropic. Additionally, we see $E_+ \oplus E_-$ is symplectic and $E_\pm$ are Lagrangian subspaces of $E_+ \oplus E_-$ .

As we assumed that $A|_{E_\pm}$ is diagonalizable, we can choose  eigenbases for $E_\pm \otimes \C$. Select respective inner products that make these eigenbases orthonormal. We denote the real part of these inner products by $\lrang{\cdot, \cdot}_\pm$. Let $\{v_1^{\pm}, \ldots, v_l^{\pm}\}$ be orthonormal bases for $E_\pm$ with respect to $\lrang{\cdot, \cdot}_\pm$. There exists matrices $B_\pm$ such that $B_\pm$ is an orthogonal matrix with respect to $\lrang{\cdot, \cdot}_\pm$ and $A|_{E_\pm} = \Lambda^{\pm 1} B_\pm$. Clearly, $\lrang{A v_i^\pm, A v_i^\pm}_\pm = \Lambda^{\pm 2}$. 

Let $\pi_\pm : E_+ \oplus E_- \rightarrow E_\pm$ be the projection onto $E_\pm$ with kernel $E_\mp$. We define a norm $\|\cdot \|_E$ on $E_+ \oplus E_-$ by 
 $\|v \|_E^2 = \lrang{\pi_+ v, \pi_+ v}_+ + \lrang{\pi_- v, \pi_- v}_-$. Then $\|A v_i^\pm\|_E = \Lambda^{\pm 1}$.

 We use the coordinates given by $v_1^+, \ldots, v_l^+, v_1^-, \ldots, v_l^-$ to define the symplectic form $\sigma_E$ on $E_+ \oplus E_-$. Specifically, 
 \begin{equation}\label{eq:darboux_conditions}
\sigma_E(v_i^+, v_j^+) = 0, \quad \sigma_E(v_i^-, v_j^-)=0, \quad \sigma_E(v_i^-, v_j^+) =\delta_{ij}. 
\end{equation}

We now see how elements of $E_\pm$ scale under $A^j$ for $j \in \Z$. Let $t \in \R^l$. As $E_\pm$  is $A$-invariant, we know that $A^j \sum_{i=1}^l t_i v_i^\pm \in E_\pm$. Therefore, for some $s^\pm \in \R^l$, $A^j \sum_{i=1}^l t_i v_i^\pm = \sum_{i=1}^l s^\pm_i v_i^\pm$.

As $v_i^\pm$ are pairwise orthogonal,  
$\|\sum_{i=1}^l s^\pm_i v_i^\pm \|_E^2 = \sum_{i=1}^l |s_i^\pm|^2 \|v_i^\pm\|_E^2 = |s^\pm|^2$.
We also see
$\|A^j \sum_{i=1}^l t_i v_i^\pm \|_E^2 = \Lambda^{\pm 2 j} |t|^2$.

Then for some $s^\pm \in \R^l$, 
\begin{equation}\label{eq:Aj_rescaling}
A^j \sum_{i =1}^l t_i v_i^\pm = \sum_{i=1}^l s^\pm_i v^\pm_i \quad \text{with} \quad \Lambda^{\pm j} |t| = |s^\pm|.
\end{equation}

\subsubsection{Safe sets}
We generalize the definition of \emph{safe sets} from~\cite{dyatlov2021semiclassical}*{Definition 3.3}.

\begin{definition}\label{def:safe} A set $\cU \subset \bT^{2n}$ is \emph{safe} if, for each $z \in \bT^{2n}$ and $v \in E_+ \cup E_-$ with $|v| = 1$, there exists some $c \in \R$ such that $z + cv \bmod \Z^{2n} \in \cU$. 
\end{definition}

Recall that in \S\ref{subsection:definition}, we assumed that $\bT^{2n} \setminus \supp \mu$ intersects $z + \R v \bmod \Z^{2n}$ for all $z \in \bT^{2n}$ and $v \in E_+ \cup E_-$. In other words, $\bT^{2n} \setminus \supp \mu$ is safe.

In the next two lemmas, we respectively generalize~\cite{dyatlov2021semiclassical}*{Lemma 3.5} and~\cite{dyatlov2021semiclassical}*{Lemma 3.6} to construct a particular partition of unity. 

\begin{lemma}\label{lem:safe_compact}
If $\cU \subset \bT^{2n}$ is open and safe, then there exists a compact and safe $K \subset \cU$. 
\end{lemma}

\begin{proof}
We first take a compact exhaustion of $\cU$. Specifically, let $\cU = \bigcup_{j \in \N} K_j$, where $K_j$ are compact and $K_j \subset K^\circ_{j+1}$. Suppose none of the $K_j$'s are safe. In other words, for each $j$ there exists $z_j \in \bT^{2n}$ and $v_j \in E_+ \cup E_-$ with $|v_j| =1$ such that  $z_j + \R v_j \bmod \Z^{2n} \cap K_l =\emptyset$ for all $l \leq j$. Pass to convergent subsequences $z_j \rightarrow z_\infty$, $v_j \rightarrow v_\infty$, where $v_j, v_\infty \in E_\varsigma$ for $\varsigma \in \{+, -\}$. Then for all $c \in \R$ and $l \in \N$, $z_\infty + cv_\infty \bmod \Z^{2n} \notin K^\circ_l$, a contradiction.
\end{proof}

\begin{lemma}\label{lem:construct_b1_b2}
Let $\cU$ be an open, safe subset of $\bT^{2n}$. Then there exist $b_1, b_2 \in C^\infty(\bT^{2n})$ such that
$$b_1, b_2 \geq 0, \quad b_1 + b_2 = 1, \quad \supp b_1 \subset \cU$$
and the complements $\bT^{2n} \setminus \supp b_1$, $\bT^{2n} \setminus \supp b_2$ are both safe.
\end{lemma} 

\begin{proof}
1. We first show that there exist two compact sets 
$K_1, K_2 \subset \bT^{2n}$ such that $K_1 \cap K_2 = \emptyset$, $K_1 \subset \cU$, and $K_1$, $K_2$ are both safe.

Let $R> \sqrt{2n}$ and let $S_R$ be the $(2n-1)$-sphere of radius $R$ centered at the origin.  Denote by $\pi : \R^{2n} \to \bT^{2n}$ the projection map onto the torus.
 Define 
$$D_R \coloneqq \pi(S_R) \subset \bT^{2n}.$$

Let $z \in [0,1]^{2n}$ and $v \in E_+ \cup E_-$ with $|v| =1$. Since $|z| \leq \sqrt{2n}$, there exists $c \in \R$ such that $z + c v \in S_R$. Therefore, $D_R$ is safe. 

However, $D_R$ has at most countably many intersections with any line on the torus. Thus, for all $v \in E_+ \cup E_-$ with $|v|=1$, the intersection $D_R \cap \pi(z+ \R v)$ has empty interior in $\pi(z + \R v)$. This implies the open set $\cU \setminus D_R$ is safe. Then, by Lemma~\ref{lem:safe_compact}, there exists a safe compact set $K_1 \subset \cU \setminus D_R$. The complement $\bT^{2n} \setminus K_1$ contains $D_R$ and thus is an open safe set. Again by Lemma~\ref{lem:safe_compact}, we can find $K_2$, a compact safe subset of $\bT^{2n} \setminus K_1$. We see that $K_1, K_2 \subset \bT^{2n}$ are both safe and satisfy  $K_1 \cap K_2 = \emptyset$ and $K_1 \subset \cU$.

2. Using a partition of unity subordinate to the cover of $\bT^{2n}$ by the sets $\cU \setminus K_2$, $\bT^{2n} \setminus K_1$, we choose $b_1, b_2 \in C^\infty(\bT^{2n})$ such that 
$$b_1, b_2 \geq 0, \quad b_1 + b_2 =1, \quad \supp b_1 \subset \cU \setminus K_2, \quad \supp b_2 \subset \bT^{2n} \setminus K_1.$$
The complements of $\supp b_1, \supp b_2$ respectively contain the sets $K_2, K_1$ and thus are both safe. 
\end{proof}

\subsection{Proof of Lemma~\ref{lem:Bw_decay}} \label{subsection:Bwdecay1.4}
Setting $\cU = \bT^{2n} \setminus \supp \mu$, we choose $b_1$, $b_2 \in C^\infty (\bT^{2n})$ from Lemma~\ref{lem:construct_b1_b2}.  Importantly, $\bT^{2n} \setminus \supp b_1$ and $\bT^{2n} \setminus \supp b_2$ are safe. We use $b_1, b_2$ to construct $B_\w$ from Lemma~\ref{lem:Bw_decay}.

First, we find a reformulation of~\eqref{eq:Bw_bound}. 
Decompose the word $\w$ into two words of length $T_1$:
$\w =\w_+ \w_-$. Then we relabel $\w_+$ and $\w_-$ as
$$\w_+ = w^+_{T_1} \cdots w_1^+, \quad \w_- = w_0^- \cdots w_{T_1-1}^-.$$
Now set 
$$b_+ =\prod_{k=1}^{T_1} b_{w_k^+} \circ A^{-k} \quad \text{and} \quad b_- =\prod_{k=0}^{T_1-1} b_{w_k^-} \circ A^{k}.$$
To work with $b_\pm$, we need to show they lie in an appropriate symbol class. We adapt~\cite{dyatlov2021semiclassical}*{Lemma 3.7}.
\begin{lemma}
For all $\varepsilon>0$, $b_\pm \in S_{L_\mp, \rho + \varepsilon, \rho' + \varepsilon} (\bT^{2n})$ with bounds on the semi-norms that do not depend on $N$ or $\w$. 
Moreover,
\begin{align}\label{eq:b+-}
\begin{split}
B_{\w_+}(-T_1) &= \op_{N, \theta}(b_+) + \cO\left(N_j^{\rho+\rho' +2\varepsilon-1}\right)_{\cH_j \rightarrow \cH_j},\\
B_{\w_-} &= \op_{N, \theta}(b_-) + \cO\left(N_j^{\rho+\rho' +2\varepsilon-1}\right)_{\cH_j \rightarrow \cH_j},
\end{split}
\end{align}
where the constants in $\cO(\cdot)$ are uniform in $N$ and $\w$.
\end{lemma}
\begin{proof}
We give only the proof for $b_-$ as the proof for $b_+$ follows similarly. By Lemma~\ref{lem:manymultiplication}, it suffices to show for $i=1,2$ and $0 \leq k \leq T_1-1$, each  $S_{L_-, \rho + \frac{\varepsilon}{2}, \rho' +\frac{\varepsilon}{2}}$-seminorm of $b_i \circ A^k$ is bounded uniformly in $k$ and $N$.  

Let $X_1, \ldots, X_l, Y_1, \ldots, Y_m$ be constant vector fields on $\R^{2n}$ with $Y_j$ tangent to $L_-$.
Using the fact that $A$ is a linear map, we calculate 
$$X_1\cdots X_l Y_1 \cdots Y_m \left(b_i \circ A^k\right)(z) = D^{l+m}b_i\left(A^k z\right) \cdot \left(A^k X_1, \ldots, A^k X_l, A^k Y_1, \ldots, A^k Y_m\right),$$ where  $D^{l+m}b_i$ denotes the $(k + m)$-th derivative of $b_i$, a $(k + m)$-linear form, uniformly bounded in $N$. Therefore, using the bounds on powers of $A$ given by~\eqref{eq:Gelfand},
\begin{align*}
\sup_{\R^{2n}} \left|X_1\cdots X_l Y_1 \cdots Y_m \left(b_i \circ A^k\right)\right| &\leq C \left|A^kX_1\right| \cdots \left|A^k X_l\right|\left|A^k Y_1\right|\cdots \left|A^k Y_m\right|\\
&\leq C \Lambda^{k(1+\varepsilon')l} \gamma^{k(1+\varepsilon') m}\\
&\leq C \Lambda^{T_1(1+\varepsilon') l} \gamma^{T_1(1+\varepsilon')m}\\
&\leq CN^{\rho (1+\varepsilon') l + \rho'(1+\varepsilon') m}.
\end{align*} The final inequality follows from~\eqref{eq:original_T_Def} and~\eqref{eq:rho}. Choosing $\varepsilon' < \frac{\varepsilon}{2 \rho}$, the right-hand-side is bounded by $CN^{(\rho+\frac{\varepsilon}{2}) l + (\rho'+\frac{\varepsilon}{2})m}$, which finishes the proof. 
\end{proof}

We note that $B_\w = M^{-T_1}_{N, \theta} B_{\w_-} B_{\w_+} M^{T_1}_{N, \theta}$. Therefore, by~\eqref{eq:b+-}, we know there exists $C>0$ such that $\|B_\w\|_{\cH_j \rightarrow \cH_j} \leq \|\op_{N, \theta}(b_-) \op_{N, \theta}(b_+)\|_{\cH_j \rightarrow \cH_j} + CN_j^{\rho + \rho' + \varepsilon-1}$. Thus, to show~\eqref{eq:Bw_bound}, it suffices to instead show 
\begin{equation}\label{eq:B+-_N_bound}
\|\op_{N, \theta}(b_-) \op_{N, \theta}(b_+)\|_{\cH_j \rightarrow \cH_j} \leq Ch^\beta.
\end{equation}

By~\cite{dyatlov2021semiclassical}*{(2.45)}, viewing $b_\pm \in C^\infty(\bT^{2n})$ as a $\Z^{2n}$-periodic function in $C^\infty(\R^{2n})$, 
$$\max_{\theta \in \bT^{2n}} \|\op_{N, \theta}(b_-) \op_{N, \theta}(b_+) \|_{\cH_N(\theta) \rightarrow \cH_N(\theta)} = \|\op_h(b_-) \op_h(b_+)\|_{L^2(\R^n) \rightarrow L^2(\R^n)}.$$

Therefore,~\eqref{eq:B+-_N_bound} reduces to 
\begin{equation}\label{eq:B+-_L2_bound}
\|\op_h(b_-) \op_h(b_+)\|_{L^2(\R^n) \rightarrow L^2(\R^n)} \leq Ch^\beta.
\end{equation}

\subsubsection{Porosity}
For $z \in \bT^{2n}$ and $t=(t_1, \ldots, t_l) \in \R^l$, define
$$\phi_\pm^t(z) \coloneqq z + \sum_{i=1}^l t_i v^\pm_i \mod \Z^{2n}.$$

To show~\eqref{eq:B+-_L2_bound}, we will employ the fractal uncertainty principle of Proposition~\ref{prop:fractal_uncertainty}. This requires a discussion on porosity. 

Recall the definitions of $L_\mp$ and $\rho'$ from~\eqref{eq:L_pm_def} and~\eqref{eq:rho}, respectively. 
Then define 
\begin{equation}\label{eq:Omega_def}
\Omega_\pm(z) \coloneqq \{t \in \R^l : \exists v \in L_\mp \text{ such that } |v| \leq h^{\rho'} \text{ and } \phi_\pm^t(z) +v \in \supp b_\mp \}, \quad z \in \bT^{2n}.
\end{equation}
Intuitively, to construct $\Omega_\pm(z)$, first lift $z$ to a point in $\R^{2n}$ and $\supp b_\mp$ to a subset of $\R^{2n}$. Then intersect $\supp b_\mp$ with the set 
\begin{equation}\label{eq:Omega_set}
\{\phi^t_\pm(z) +v : t \in \R^l, v \in L_\mp, |v| \leq h^{\rho '}\},
\end{equation}
which is illustrated in Figure~\ref{fig:Omega}. The set $\Omega_\pm(z)$ is given by the set of all $t \in \R^l$ in this intersection. We restrict to $|v| \leq h^{\rho '}$  since we will later apply a partition of unity to split up the support of $b_\mp$.

\begin{figure}
    \centering
\begin{tikzpicture}
\draw[draw=white, fill=black!10] (-4,1) rectangle (4,3); 
\draw[thick] (-4,0) -- (4,0) node[anchor=west]{$E_\pm$};
\draw[thick] (0,-2) -- (0,4) (0,4.2) node{$L_\mp$};
\draw (-4, 1) -- (4, 1);
\draw (-4, 3) -- (4, 3);
\filldraw (1,2) circle (0.05cm) node[anchor=west]{$z$};
\draw (1,1) -- (1,3);
\draw[<->] (1.2, 2.15) -- (1.2, 2.95);
\path (1.5, 2.6) node{$h^{\rho'}$};
\end{tikzpicture}
    \caption{The shaded region is an illustration of the set~\eqref{eq:Omega_set}.}
    \label{fig:Omega}
\end{figure}

The following lemma generalizes~\cite{dyatlov2021semiclassical}*{Lemma 4.4} to higher dimensions to show that $\Omega_\pm$ is porous on lines. Note that porosity on lines implies porosity on balls.

\begin{lemma}\label{lem:porous_lines}
Let $\varrho \in (0, \rho)$. Then there exists $\nu, h_0 \in (0,1)$, independent of $N$ and $\w$, such that if $0 < h \leq h_0$, for every $z \in \bT^{2n}$, $\Omega_\pm(z)$ is $\nu$-porous on lines from scales $h^\varrho$ to $1$.
\end{lemma}

\begin{proof}
We just examine $\Omega_+(z)$; the proof of $\Omega_-(z)$ can be handled similarly by reversing the direction of time. The $\nu$ in the lemma statement is the minimum of the $\nu$ for $\Omega_+(z)$ and $\Omega_-(z)$

1. As the complements $\bT^{2n} \setminus \supp b_1$ and $\bT^{2n} \setminus \supp b_2$ are safe, by Lemma~\ref{lem:safe_compact}, there exists compact subsets $K_1, K_2 \subset \bT^{2n}$ such that the interiors $K_1^\circ$, $K_2^\circ$ are safe and $$K_1 \cap \supp b_1=K_2 \cap \supp b_2 = \emptyset.$$

We claim that there exist constants $R>0$, $\nu_0 >0$ such that for $j=1, 2$, if $I \subset \R^l$ is a line segment of length $R$, then for each $z \in \bT^{2n}$, there exists $\tau \in I$ such that $\{\phi^t_+ (\phi^\tau_+ (z)) : |t| \leq \nu_0R\} \subset K_j$.

As we shift by $z$, it suffices to only consider $I$ centered at the origin. Suppose towards a contradiction that $R$ and $\nu_0$ do not exist. Then there exists a sequence of points $z_m \in \bT^{2n}$ and lines $I_m$ of length $m$ centered at the origin such that for all $\tau \in I_m$, $K_j$ does not contain $\{\phi^t_+ (\phi^\tau_+ (z_m)) : |t| \leq 1/m \}$.
We pass to convergent subsequences and assume $z_m \rightarrow z_\infty$ and  $I_m$ converges to $\R \tau$ for some $\tau \in \R^l$, i.e, the direction of $I_m$ converges to the direction of $\R \tau$. Rescale  $\tau$ such that for $v_\infty = \sum_{i=1}^l \tau_i v_i^+$, $|v_\infty|=1$. As $K_j^\circ$ is safe, there exists $c \in \R$ such that $z_\infty + cv_\infty \mod \Z^{2n} \in K^\circ_j$. In other words, $\phi_+^{c \tau}(z_\infty) \in K_j^\circ$. Thus, for sufficiently large $m$, there exists $\tau_m \in I_m$ such that $\{\phi^t_+(\phi^{\tau_m}_+(z_m)) : |t| \leq 1/m\} \subset K_j^\circ$, a contradiction.

2. 
Let $z \in \bT^{2n}$. We set $\nu = \Lambda^{-1} \nu_0$. We want to show that $\Omega_+(z)$ is $\nu$-porous on lines from scales $h^\varrho$ to 1. Let $I \subset \R^l$ be a line segment of length $h^\varrho \leq |I| \leq 1$. Let $j$ denote the smallest integer such that $\Lambda^{j} |I| \geq R$.  Using the definition of $T_1$ in~\eqref{eq:original_T_Def} and $N = (2 \pi h)^{-1}$, we have 
$$
\Lambda^{T_1 -1} |I| \geq \Lambda^{-1} (2 \pi)^{-\rho} h^{\varrho - \rho}.
$$
As the right-hand side converges to $+\infty$ as $h \rightarrow 0$, by taking sufficiently small $h$, we make the assumption that $0 < j < T_1$. 

We know that $A^j (\{\phi^t_+(z) : t \in I \}) = \{\phi^t_+(A^jz) : t \in \tilde{I} \}$ for some line segment $\tilde{I} \subset \R^l$. By~\eqref{eq:Aj_rescaling}, $\tilde{I}$ is a line segment of length  $\Lambda^{j} |I| \geq R$. 
We conclude that there exists $\tau \in \tilde{I}$ such that $\{\phi^t_+(\phi^\tau_+(A^j z)) : |t| \leq \nu_0 R \} \subset K_{w^-_j}$.

By~\eqref{eq:Aj_rescaling}, we know that 
$$ \left\{A^j\sum_{i=1}^l t_i v_i^+ : |t| \leq  \Lambda^{-j} \nu_0 R\right\} \subset  \left\{\sum_{i=1}^l t_i v_i^+ : |t| \leq \nu_0 R\right\}.$$

Let $\tau' \in \R^l$ be given by $A^{-j} \sum_{i=1}^l \tau_i v_i^+ = \sum_{i=1}^l \tau_i' v_i^+$. Note that $\tau' \in I$. Set $B$ to be the ball of radius $\Lambda^{-j}\nu_0 R$ centered at $\tau'$. Then, 
\begin{equation}\label{eq:propagate_forward_line}
A^j \phi^t_+(z) \in K_{w_j^-} \quad \text{ for all } t \in B.
\end{equation} 
By our choice of $j$, we have $\Lambda^{-j}\nu_0 R \geq  \Lambda^{-1} |I| \nu_0 = |I| \nu$.

It remains to show that $B \cap \Omega_+(z) = \emptyset$, in other words, we need to show that for each $t \in B$, 
$$\phi^t_+(z) + v \notin \supp b_-  \text{ for all } v \in L_-  \text{ such that }  |v| \leq h^{\rho'}.$$

By~\eqref{eq:rho} and~\eqref{eq:Gelfand}, there exists some $C>0$ depending on $\varepsilon$ such that for $v \in L_-$ with $|v| \leq h^{\rho'}$ and $j \leq T_1$,
$$|A^j v| \leq C \gamma^{j(1+\varepsilon)} h^{\rho'} \leq C \gamma^{T_1(1+\varepsilon)} h^{\rho'} \leq C(2 \pi)^{-\rho \frac{\log \gamma}{\log \Lambda}(1+\varepsilon)}h^{\rho'-\rho \frac{\log \gamma}{\log \Lambda}(1+\varepsilon)}.$$
For sufficiently small $\varepsilon$, the right-hand-side converges to $0$.
Using~\eqref{eq:propagate_forward_line} and the fact that $K_{w_j^-} \cap \supp b_{w_j^-} = \emptyset$, we know $A^j(\phi^t_+(z) +v) \notin \supp b_{w_j^-}$ for sufficiently small $h$. From the definition of $b_-$, it follows that $\phi^t_+(z) +v \notin \supp b_-$.
\end{proof}

\subsubsection{Application of fractal uncertainty principle}
Recall that we have reduced the estimate $\|B_\w\|_{\Hj \rightarrow \Hj} \leq Ch^\beta$ to~\eqref{eq:B+-_L2_bound}. We prove~\eqref{eq:B+-_L2_bound} via the Cotlar--Stein Theorem~\cite{z12semiclassical}*{Theorem~C.5} and the higher-dimensional fractal uncertainty principle.  We begin by constructing a partition of unity.

Fix a function
$$\tilde{\psi} \in C^\infty_c\left(B_{\sqrt{n}}(0); \R\right), \quad \sum_{k \in \Z^{2n}} \tilde{\psi}(z-k)^2 =1 \quad \text{for all} \quad z \in \R^{2n}.$$

Now, we consider the partition of unity $\sum_{k \in \Z^{2n}} \psi_k^2 = 1$, where the $h$-dependent symbol $\psi_k \in C^\infty_c(\R^{2n})$ is given by 
$$\psi_k(z) := \tilde{\psi}\left(\frac{\sqrt{n}}{h^{\rho'}}z -k\right), \quad k \in \Z^{2n}.$$
Note that for any coisotropic subspace $L$, $\psi_k$ lies in $S_{L, \rho, \rho'}(\R^{2n})$ uniformly in $h$ and $k$.

Then set $$P_k := \op_h(b_-) \op_h(\psi_k^2) \op_h(b_+).$$
We see that $$\op_h(b_-)\op_h(b_+) = \sum_{k \in \Z^{2n}} P_k,$$ where the series converges in the strong operator topology as an operator $L^2(\R^n) \rightarrow L^2(\R^n)$.  
Via the Cotlar--Stein Theorem, to show~\eqref{eq:B+-_L2_bound}, it suffices to show that 
\begin{equation}\label{eq:Cotlar_Stein2}
\sup_{k \in \Z^{2n}} \sum_{l \in \Z^{2n}} \|P^*_k P_l\|^{\frac{1}{2}}_{L^2 \rightarrow L^2} \leq Ch^\beta \quad \text{and} \quad \sup_{k \in \Z^{2n}} \sum_{l \in \Z^{2n}} \|P_k P_l^*\|^{\frac{1}{2}}_{L^2 \rightarrow L^2} \leq Ch^\beta.
\end{equation}

We can deduce~\eqref{eq:Cotlar_Stein2} from~\cite{dyatlov2021semiclassical}*{Lemma 4.5}, stated in this paper Lemma~\ref{lem:close_together}, and the following lemma.

\begin{lemma}\label{lem:P_k_decay}
There exists constants $C, \beta>0$ such that for every $k \in \Z^{2n}$, we have 
\begin{equation}\label{eq:P_k_decay}
\|P_k\|_{L^2 \rightarrow L^2} \leq Ch^\beta
\end{equation} where $\beta$ only depends on $\nu$ from Lemma~\ref{lem:porous_lines} and on $\rho$. 
\end{lemma}

Although $P_k$ is defined differently from \S\ref{section:thm1.3} in this section, the original proof of~\cite{dyatlov2021semiclassical}*{Lemma~4.5} holds for both definitions.
On the other hand, although the statement of Lemma~\ref{lem:P_k_decay} is the same as~\cite{dyatlov2021semiclassical}*{Lemma 4.6}, the key ingredient at the end of the proof is different. In~\cite{dyatlov2021semiclassical}, the authors assume a spectral gap condition for $A$. As a result, their lemma follows from an application of the 1-dimensional fractal uncertainty principle. However, as we made no such assumption, our lemma necessitates the higher-dimensional fractal uncertainty principle.

\begin{proof}[Proof of Lemma~\ref{lem:P_k_decay}]
1. We know that the $\psi_k$'s belong uniformly to the symbol classes $S_{L_\pm, \rho, \rho'}$ and that $b_\pm$ belong to the symbol classes $S_{L_\pm, \rho + \varepsilon, \rho' + \varepsilon}$ for all $\varepsilon>0$. 
Since $S_{L_\pm, \rho,\rho'} \subset S_{L_\pm, \rho + \varepsilon, \rho' + \varepsilon}$, 
$$P_k =\op_h(b_- \psi_k) \op_h(b_+ \psi_k) + \cO(h^{1- \rho -\rho' -2\varepsilon})_{L^2 \rightarrow L^2}.$$

Therefore,~\eqref{eq:P_k_decay} reduces to 
\begin{equation}\label{eq:P_k_decay1}
\|\op_h(b_- \psi_k) \op_h(b_+ \psi_k) \|_{L^2 \rightarrow L^2} \leq C h^\beta.
\end{equation}

2. We next study the supports of the symbols $b_\pm \psi_k$. We have that
$$
\supp \psi_k \subset  B_{ h^{\rho'}}\left(\frac{h^{\rho'}}{\sqrt{n}}k \right).
$$
As $\R^{2n} = E_\pm \oplus L_\mp$, any $z \in \supp \psi_k$ can be written as $z = h^{\rho'}n^{-\frac{1}{2}}k + \sum_{i =1}^l t_i v_i^\pm +u_\mp$, where $u_\mp \in L_{\mp}$ with $|u_\mp| \leq C_0 h^{\rho'}$ for some constant $C_0$ depending only on the matrix $A$. 

Choose $s^{(k_\pm)} \in \R^l$ such that $h^{\rho'}n^{-\frac{1}{2}}k \in \sum_{i=1}^l s^{(k_\pm)}_i v_i^\pm + L_\mp$. We set $z^{(k)} :=  h^{\rho'} n^{-\frac{1}{2}}k$ mod $\Z^{2n} \in \bT^{2n}$. Using~\eqref{eq:Omega_def}, the definition of $\Omega_\pm(z)$, we have
\begin{equation}\label{eq:supp_b_psi}
\supp(b_\mp \psi_k) \subset \bigcup_{t \in \tilde{\Omega}_\pm} \left(\sum_{i=1}^l t_i v^\pm_i  + L_\mp \right), \quad \text{where} \quad  \tilde{\Omega}_\pm :=s^{(k_\pm)} + \Omega_\pm (z^{(k)}) \subset \R^l.
\end{equation}

3. We now want to apply Darboux's theorem to ``straighten out" $E_\pm$ and $L_\mp$. Recall  from~\eqref{eq:darboux_conditions} that $\sigma_E(v_i^\pm, v_j^\pm) =0$ and $\sigma_E(v_i^-, v_j^+)= \delta_{ij}$. Therefore, using the linear version of Darboux's theorem, we construct a symplectic matrix $Q \in \Sp(2n, \R)$ such that  
\begin{itemize}
\item $Q \partial_{x_i} = v^-_i$ for $1 \leq i \leq l$; 
\item $Q \partial_{\xi_i} = v^+_i$ for $1 \leq i \leq l$;
\item $Q \spn(\partial_{x_1}, \ldots, \partial_{x_n}, \partial_{\xi_{l+1}}, \ldots, \partial_{\xi_n}) = L_-$;
\item $Q \spn(\partial_{x_{l+1}}, \ldots, \partial_{x_n}, \partial_{\xi_{1}}, \ldots, \partial_{\xi_n}) = L_+$.
\end{itemize}

Now, let $M_Q \in \cM_Q$ be a metaplectic transformation satisfying $M^{-1}_Q \op_h(a) M_Q = \op_h(a \circ Q)$ for all $a \in S(1)$. Therefore,
\begin{align*}
\|\op_h(b_- \psi_k) \op_h(b_+ \psi_k)\|_{L^2 \to L^2} &=\left\|M_Q^{-1}\op_h(b_- \psi_k)  M_Q M_Q^{-1} \op_h(b_+ \psi_k)M_Q\right\|_{L^2 \to L^2} \\
&=\left\|\op_h\left((b_- \psi_k) \circ Q\right) \op_h\left((b_+ \psi_k) \circ Q\right) \right\|_{L^2 \rightarrow L^2}.
\end{align*}
Thus, it suffices to prove
\begin{equation}\label{eq:P_k_decay2}
\left\|\op_h\left((b_- \psi_k) \circ Q\right) \op_h\left((b_+ \psi_k) \circ Q\right) \right\|_{L^2(\R^n) \rightarrow L^2(\R^n)} \leq Ch^\beta
\end{equation}
in lieu of~\eqref{eq:P_k_decay1}. 

As $\supp \psi_k \subset  B_{h^{\rho'}}(h^{\rho'}n^{-\frac{1}{2}}k ),$ we know that $$\supp \psi_k \circ Q \subset B_{Ch^{\rho'}}(Q^{-1}h^{\rho'} n^{-\frac{1}{2}}k),$$
where $C$ depends only on $A$. 
We use a splitting of coordinates $x=(x', x'')$ where $x'=(x_1, \ldots, x_l)$ and $x''=(x_{l+1}, \ldots, x_n)$. Set $\partial_{x'}=(\partial_{x_1}, \ldots, \partial_{x_l})$. 
The support condition~\eqref{eq:supp_b_psi} is now 
$$\supp \left((b_- \psi_k) \circ Q\right) \subset \left\{(x, \xi) \in  B_{Ch^{\rho'}}(Q^{-1}h^{\rho'}n^{-\frac{1}{2}}k) \mid \xi' \in \tilde{\Omega}_+\right\},$$
$$\supp \left((b_+ \psi_k) \circ Q\right) \subset \left\{(x, \xi) \in  B_{Ch^{\rho'}}(Q^{-1}h^{\rho'}n^{-\frac{1}{2}}k) \mid x' \in \tilde{\Omega}_-\right\}.$$

4. For $\delta>0$, using the notation of~\eqref{eq:neighborhood_def}, denote the $\delta$-neighborhood of $\tilde{\Omega}_\pm$ by $$\tilde{\Omega}_\pm(\delta) \coloneqq \tilde{\Omega}_\pm + B_\delta(0).$$

We will replace~\eqref{eq:P_k_decay2} with an estimate on the norm of two cutoff functions, akin to the argument in Lemma~\ref{lem:aandbbounds}.
Define $\chi_\mp$ to be the convolutions of the indicator functions of $\tilde{\Omega}_\pm(\frac{1}{2}h^\rho)$ with the function $h^{-\rho}\chi(h^{-\rho}t):\R^l \rightarrow \R$, where $\chi \in C_c^\infty(B_{1/2}(0))$ is a nonnegative function integrating to 1. Then, 
$$\chi_\mp \in C^\infty(\R^l, [0,1]), \quad \supp \chi_\mp \subset \tilde{\Omega}_\pm(h^\rho), \quad \chi_\mp =1 \text{ on } \tilde{\Omega}_\pm,$$
and for each $\alpha \in \N^l$, there exists a constant $C_\alpha$ (depending only on $\alpha$ and the choice of $\chi$) such that 
$$\sup_{t \in \R^l} |\partial_t^\alpha \chi_\pm (t) | \leq C_\alpha h^{-\rho |\alpha|}.$$

We define the symbols $\tilde{\chi}_\pm \in C^\infty(\R^{2n})$ by 
$$\tilde{\chi}_-(x, \xi) = \chi_-(\xi'), \quad \tilde{\chi}_+(x, \xi)=\chi_+(x').$$
Then $\tilde{\chi}_\pm$ lie in the symbol class $S_{Q^{-1}L_\pm, \rho, 0}(\R^{2n})$ uniformly in $h$. On the other hand,  $(b_\pm \psi_k) \circ Q \in S_{Q^{-1}L_\pm, \rho + \varepsilon, \rho' + \varepsilon}(\R^{2n})$, a larger symbol class.  Using the fact that $(b_\pm \psi_k) \circ Q = ((b_\pm \psi_k)\circ Q)\tilde{\chi}_\pm $, we have
\begin{align}\label{eq:chi_mult}
\begin{split}
\op_h\left((b_-\psi_k) \circ Q \right)=\op_h\left((b_-\psi_k) \circ Q \right)\op_h\left(\tilde{\chi}_- \right) + \cO(h^{1-\rho-\rho'-2\varepsilon})_{L^2 \to L^2},\\
\op_h\left((b_+\psi_k) \circ Q \right)=\op_h\left(\tilde{\chi}_+ \right)\op_h\left((b_+\psi_k) \circ Q \right) + \cO(h^{1-\rho-\rho'-2\varepsilon})_{L^2 \to L^2}.
\end{split}
\end{align}

We see that $\op_h(\tilde{\chi}_+) = \chi_+(x')$ is a multiplication operator and $\op_h(\tilde{\chi}_-)= \chi_-(-ih\partial_{x'})$ is a Fourier multiplier. Using the estimates in~\eqref{eq:chi_mult} and the fact that $\op_h((b_\pm \psi_k) \circ Q)$ are bounded uniformly as operators on $L^2(\R^n)$,~\eqref{eq:P_k_decay2} becomes 
\begin{equation}\label{eq:P_k_decay3}
\|\chi_-(-ih\partial_{x'}) \chi_+(x')\|_{L^2(\R^n) \to L^2(\R^n)} \leq C h^\beta.
\end{equation}

Now we consider $L^2(\R^n)$ as the Hilbert tensor product $L^2(\R^l) \otimes L^2(\R^{n-l})$ corresponding to the decomposition $x=(x',x'')$. Then the operators $\chi_+(x')$ and $\chi_-(-ih\partial_{x'})$ are the tensor products of the same operators in $l$-variables with the identity operator on $L^2(\R^{n-l})$. Thus,~\eqref{eq:P_k_decay3} is equivalent to 
\begin{equation}\label{eq:P_k_decay4}
\|\chi_-(-ih\partial_{x'}) \chi_+(x')\|_{L^2(\R^l) \to L^2(\R^l)} \leq Ch^\beta,
\end{equation}
where we now treat the factors in the product as operators on $L^2(\R^l)$. Using Definition~\ref{def:semiclassical_FT} of the unitary semiclassical Fourier transform $\cF_h: L^2(\R^l) \to L^2(\R^l)$, we see that $\chi_-(-ih\partial_{x'}) = \cF_h^{-1} \chi_-(x') \cF_h$. Therefore, 
$$\|\chi_-(-ih\partial_{x'}) \chi_+(x')\|_{L^2(\R^l) \to L^2(\R^l)}=\|\chi_-(x')\cF_h \chi_+(x')\|_{L^2(\R^l) \to L^2(\R^l)}.$$
Finally, as $\chi_\pm = \chi_\pm \1_{\tilde{\Omega}_\mp(h^\rho)}$ and $|\chi_\pm| \leq 1$,~\eqref{eq:P_k_decay4} reduces to
\begin{equation}\label{eq:P_k_decay5} 
\left\|\1_{\tilde{\Omega}_+(h^\rho)} \cF_h \1_{\tilde{\Omega}_-(h^\rho)} \right\|_{L^2(\R^l) \rightarrow L^2(\R^l)} \leq Ch^\beta. 
\end{equation}

5. The last step of the proof is to apply the fractal uncertainty principle of Proposition~\ref{prop:fractal_uncertainty}. 

Pick $\varrho \in (\frac12, \rho)$.
 By Lemma~\ref{lem:porous_lines}, $\tilde{\Omega}_+$ and $\tilde{\Omega}_-$ are $\nu$-porous on lines from scales $h^\varrho$ to $1$. Thus, $\tilde{\Omega}_+$ is $\nu$-porous on balls from scales $h^\varrho$ to $1$. Applying Lemma~\ref{lem:gen_porous_balls} and Lemma~\ref{lem:gen_porous_lines} for sufficiently small $h$ with $\alpha_0 = h^\varrho$, $\alpha_1 =1$, and $\alpha_2 = h^\rho$, we know $\tilde{\Omega}_+(h^\rho)$ is $\nu/2$-porous on balls from scales $h^\varrho$ to 1 and $\tilde{\Omega}_-(h^\rho)$ is $\nu/2$-porous on lines from scales $h^\varrho$ to 1. 
Then~\eqref{eq:P_k_decay5} follows from Proposition~\ref{prop:fractal_uncertainty}, completing the proof of Lemma~\ref{lem:Bw_decay}. 
\end{proof}

\section{Proof of Theorem~\ref{thm:counterexample}}\label{section:counterexample}
Recall that in Theorem~\ref{thm:coisotropic}, we showed that if hyperbolic $A \in \Sp(2n, \Z)$ is diagonalizable over $\C$ and has semiclassical measure $\mu$ supported on the union of finitely many subtori with the same tangent space, then the tangent space must be coisotropic. However, in this section, we prove Theorem~\ref{thm:counterexample}, which shows that this does not characterize all cases. In fact, in Appendix~\ref{appendix1}, we use Theorem~\ref{thm:counterexample} to show that Theorem~\ref{thm:full_support}~\eqref{item:K} is tight in one of the cases without full support.

Using the short periods of~\cite{FNdB}, we find a semiclassical measure that is supported on the union of two symplectic transversal subtori. In this section, we use the ordering of coordinates $(x_1, \xi_1, x_2, \xi_2)$.


\subsection{Short Periods}\label{subsection:short_periods}
Let $A \in \Sp(2, \Z)$. 
For $N \in \N$, let $P(N) = \min \{k\geq 1 : A^k \equiv I \mod 2N\}$. 
From~\cite{FNdB}*{Lemma 4}, we know 
$$M_{N, \theta}^{P(N)} = e^{i \phi(N)} I \quad \text{for some} \quad \phi(N) \in [-\pi, \pi).$$ Thus, we call $P(N)$ the \emph{period} of $M_{N, \theta}$.

In~\cite{BdB}*{Proposition 3.4}, the authors show that for the sequence
\begin{equation}\label{eq:tildeNk}
N_k = \frac{\lambda^k -\lambda^{-k}}{\lambda - \lambda^{-1}} \quad \text{and} \quad \theta = \begin{cases} (0,0) & \text{ for } N_k \text{ even,}\\ (\pi, \pi) & \text{ for } N_k \text{ odd,}
\end{cases}
\end{equation}
$P(N_k) = 2k$ and therefore
$$P(N_k) =\frac{2 \log N_k}{\log \lambda} + \cO(1).$$

\subsection{Assumptions and setup for Theorem~\ref{thm:counterexample}}\label{subsection:assumptions}
Using $N_k$ from~\eqref{eq:tildeNk} and the relationship $N_{k+1}= \Tr(A)N_k + N_{k-1}$, it can be proven inductively that for odd $\Tr(A)$, $N_{3k}$ is even, while $N_{3k+1}$ and $N_{3k+2}$ are odd. For even $\Tr(A)$, $N_{2k}$ is even, while $N_{2k+1}$ is odd. We restrict to 
$$
\begin{aligned}
k=0\bmod 6,&\quad \Tr(A)\text{ odd};\\
k=0\bmod 2,&\quad \Tr(A)\text{ even}.
\end{aligned}
$$
Then $N_k$ is even and $P(N_k)=2k$ is always divisible by $4$. Furthermore, $\theta=(0,0)$ satisfies the quantization condition~\eqref{eq:domainrange}. Therefore, we set $\theta=(0,0)$.

Let $A =\begin{bmatrix} a & b \\ c & d \end{bmatrix} \in \Sp(2, \Z)$ be a hyperbolic matrix with all positive entries. We further suppose that $A$ is symmetric, i.e., $b=c$. An infinite number of matrices still satisfy these conditions, for example, pick $A= \begin{bmatrix} 2 & 1 \\ 1 & 1 \end{bmatrix}$. The conditions that all entries of $A$ are positive and $b=c$ are used to simplify calculations; we suspect that they are not necessary for Theorem~\ref{thm:counterexample}. 

We label the eigenvalues of $A$ by $\lambda, \lambda^{-1}$, where $\lambda >1$.

We have the following explicit formula for a metaplectic quantization of $A$ (this does not yet use that $b=c$):
\begin{equation}
\label{e:M-A-def}
M_Af(x) = \frac{e^{-\frac{i \pi}{4}}}{\sqrt{2 \pi h b}} \int_\R e^{\frac{i}{h}\left(\frac{d}{2b}x^2 - \frac{xy}{b} + \frac{a}{2b}y^2 \right)} f(y) dy.
\end{equation}
The above formula comes from~\cite{HB1980}*{(9)}, with a change in the phase factor which simplifies~Lemma~\ref{lem:support} below
and also gives the following multiplication formula:
for any two matrices $A,B\in\Sp(2,\Z)$ with positive entries, we have
\begin{equation}
\label{e:multiplicativity}
M_AM_B=M_{AB}.
\end{equation}

Define the $L^2$-normalized Gaussian
$$G_h(x) \coloneqq  (\pi h)^{-\frac{1}{4}} e^{-\frac{x^2}{2h}}.$$ 

Now recall the definition of the projector $\Pi_N(0): \mathscr S(\mathbb R^n)\to\mathcal H_N(0)$ from~\eqref{eq:Pi_N_def}. For $N$ even, from~\eqref{eq:S_N}, we have $\Pi_N( 0) f= \Pi_N(0)^* \Pi_N(0) f= \sum_{l \in \Z^2} U_l f$. We then set
\begin{equation}\label{eq:G_N}
G_N \coloneqq \Pi_N(0) G_h =\sum_{l \in \Z^2} U_l G_h \quad\in \mathcal H_N(0).
\end{equation}

Using the ordering of coordinates $(x_1, \xi_1, x_2, \xi_2)$, we know if $A \in \Sp(2, \Z)$, then $A \oplus A \in \Sp(4, \Z)$. 
Since 
$$(M_A \otimes M_A)^{-1} \op_{h} (b \otimes c) (M_A \otimes M_A)= \op_{h} ((b \otimes c) \circ (A \oplus A)),$$
we see $M_{A \oplus A} = M_A \otimes M_A :L^2(\mathbb R^2)\to L^2(\mathbb R^2)$ is a quantization of $A\oplus A$. Note that $A \oplus A$ is diagonalizable.

We denote the restriction of $M_A$ to $\cH_{N}(0)$ by $M_{N}$. Since we have selected $\theta=(0,0)$, we drop $\theta$ from our notation for the rest of this section. Recall that $M_N^{P(N)} = e^{i \phi(N)}I$ for some $\phi(N) \in [-\pi, \pi)$.  Then define
\begin{equation}\label{eq:u_def}
u^{(k)}(x_1, x_2) \coloneqq \frac{1}{\sqrt{P}}\sum_{t=-\frac{P}{2}}^{\frac{P}{2}-1} \left(e^{-i\frac{\phi}{ P}t}M_{N}^tG_{N}\right) \otimes \left( e^{-i\frac{\phi}{ P}\left(t+\frac{P}{2}\right)}M_{N}^{t+\frac{P}{2}}G_{N} \right),
\end{equation}
where $P=P(N_k)$, $\phi=\phi(N_k)$, $M_N =M_{N_k}$, and $G_{N}=G_{N_k}$.
We hereon refer to the above function as $u$. Later on, when we take the semiclassical limit $h \rightarrow 0$, we mean $\lim_{k \rightarrow \infty} h_k$ for $h_k = (2 \pi N_k)^{-1}$. 

A simple calculation using that $\left(e^{-i\frac{\phi}{ P}t}M_{N}^tG_{N}\right)$ is periodic in $t$ with period $P$ verifies that each $u$ is an eigenfunction of $M_A \otimes M_A$ with eigenvalue $e^{2i \frac{\phi}{P}}$. Therefore, Theorem~\ref{thm:counterexample} follows from the following proposition. 

\begin{proposition}\label{prop:counterexample}
After normalization, $u$ (defined in~\eqref{eq:u_def}) is a  sequence of eigenfunctions  for $M_A \otimes M_A$ that weakly converges to the
semiclassical measure 
\begin{equation}\label{eq:semiclassical_measure}
\mu \coloneqq \frac{1}{2}\left(\delta(x_1, \xi_1) \otimes dx_2 d\xi_2 + dx_1 d\xi_1  \otimes \delta(x_2, \xi_2)\right).
\end{equation}
\end{proposition}
Clearly, $\mu$  is supported on two symplectic transversal subtori.

Our construction of $u$ is inspired by the work of Faure, Nonnenmacher, and De Bi\`{e}vre  in~\cite{FNdB}.
In~\cite{FNdB}*{Theorem 1}, the authors show that
\begin{equation}\label{eq:original_eigenfunction}
\frac{1}{\sqrt{P}}\sum_{t=-\frac{P}{2}}^{\frac{P}{2}-1} e^{-i \frac{\phi}{P} t} M_{N}^t G_{N}
\end{equation}
is a sequence of eigenfunctions for $M_A$ 
that (after normalization) weakly converges to \linebreak  $\tfrac{1}{2}(\delta(x_1, \xi_1) + dx_1 d\xi_1)$.

We give a brief summary of their argument. Each summand in~\eqref{eq:original_eigenfunction} with $|t| \leq P/4 - o(P)$ behaves ``locally,"  concentrating at the origin in the semiclassical limit. On the other hand, the summands with $P/4 +o(P) \leq |t| \leq P/2$ behave ``ergodically," equidistributing on the torus in the semiclassical limit. Finally, the terms close to $\pm P/4$ are between local and ergodic, but contribute negligibly in the semiclassical limit. Thus, the sum over $|t| \leq P/4$ weakly converges to $\tfrac{1}{2}\delta(x_1, \xi_1)$, while the sum over $P/4 < |t| \leq P/2$ weakly converges to $\tfrac{1}{2}dx_1 d\xi_1$. 

For our eigenfunctions~\eqref{eq:u_def}, all summands except those with $|t|$ close to $P/4$ are the tensor product of one local and one ergodic summand from~\eqref{eq:original_eigenfunction}. Therefore, we expect the sum over $|t| \leq P/4$ to weakly converge to $\frac{1}{2}(\delta(x_1, \xi_1) \otimes dx_2 d\xi_2)$ and the sum over $P/4 < |t| \leq P/2$ to weakly converge to $\frac{1}{2}(dx_1 d\xi_1  \otimes \delta(x_2, \xi_2))$.

\subsection{Preliminary calculations}\label{subsection:preliminary_calculations}
We first apply the quantization of $A$ to the normalized Gaussian:
\begin{align}\label{eq:gaussian}
\begin{split}
M_A G_h(x) &=(\pi h)^{-\frac{1}{4}} \frac{e^{-\frac{i \pi}{4}}}{\sqrt{2 \pi h b}}  e^{\frac{i}{2h}\left( \frac{d}{b}x^2\right)}\int_\R e^{\frac{i}{2h} \left(y^2 \left(\frac{a}{b} + i \right) - \frac{2 xy}{b}\right)} dy\\
&=(\pi h)^{-\frac{1}{4}}\frac{e^{-\frac{i \pi}{4}}}{\sqrt{2 \pi h b}}  e^{\frac{i}{2h}x^2 \left(\frac{d}{b} - \frac{1}{b(a+ib)}\right)} \int_\R e^{\frac{i}{2h}y^2 \left( \frac{a}{b} +i\right)} dy\\
&=(\pi h)^{-\frac{1}{4}} \frac{e^{-\frac{i \pi}{4}}}{\sqrt{b-ai}} e^{\frac{i}{2h}x^2 \left(\frac{c+id}{a+ib}\right)}.
\end{split}
\end{align}

Let $\omega = (y, \eta) \in \R^2$. Recall the definition of the quantum translation $U_\omega \coloneqq \op_h(a_\omega)$, where $a_w(z) \coloneqq \exp(\frac{i}{h} \sigma(\omega, z))$. We know from~\eqref{e:U-omega-def} that $U_\omega f(x) = e^{\frac{i}{h} \eta x - \frac{i}{2h} y \eta} f(x-y)$. Therefore, 
\begin{equation}\label{eq:quantum_translation}
U_\omega e^{-\frac{x^2}{2h}} = e^{\frac{i}{h} \eta x - \frac{i}{2h} y \eta -\frac{1}{2h}(x-y)^2}.
\end{equation}

Next, we show the following technical lemma.

\begin{lemma}\label{lem:support}
Let $A\in\Sp(2,\mathbb Z)$ be a symmetric matrix with all positive entries and let $M_A$ be defined in~\eqref{e:M-A-def}.
We have
$$\lrang{U_\omega G_{h}, M_A G_h}_{L^2}=\frac{\sqrt{2}}{\sqrt{\Tr(A)}} e^{-\frac{1}{2h\Tr(A)}\lrang{A^{-1}\omega, \omega}}e^{\frac{i}{2h\Tr(A)}\lrang{AJ\omega, \omega}},$$
where $J = \begin{bmatrix} 0 & 1 \\ -1 & 0 \end{bmatrix}$.
\end{lemma}
\begin{proof}
Using~\eqref{eq:gaussian},~\eqref{eq:quantum_translation}, and the fact that $b=c$, for $\omega = (y, \eta)$ we calculate 
\begin{equation}\label{eq:U_M_innerproduct}
\begin{split}
\lrang{U_\omega G_{h}, M_A G_h}_{L^2}&= e^{\frac{i \pi}{4}} \frac{(\pi h)^{-\frac{1}{2}}}{\sqrt{ai +b}} e^{-\frac{i}{2h}y \eta}\int_\R e^{\frac{i}{h} \eta x  -\frac{(x-y)^2}{2h}} \overline{e^{\frac{i}{2h}x^2 \left(\frac{c+id}{a+ib}\right)}}dx\\
&=e^{\frac{i \pi}{4}} \frac{(\pi h)^{-\frac{1}{2}}}{\sqrt{ai +b}} e^{-\frac{i}{2h} y \eta} e^{-\frac{y^2}{2h}} \int_\R e^{-\frac{x^2}{2h} \left(\frac{a+d}{a-ib}\right) +\frac{x}{h}(i \eta +y)}dx\\
&=e^{\frac{i \pi}{4}} \frac{(\pi h)^{-\frac{1}{2}}}{\sqrt{ai +b}} e^{-\frac{i}{2h} y \eta} e^{-\frac{y^2}{2h}} e^{\frac{(i\eta + y)^2}{2h} \frac{a-ib}{a+d}}\int_\R e^{-\frac{x^2}{2h} \frac{a+d}{a-ib}}dx\\
&=\frac{\sqrt{2}}{\sqrt{a+d}}e^{-\frac{i}{2h} y \eta} e^{-\frac{y^2}{2h}} e^{\frac{(i\eta + y)^2}{2h} \frac{a-ib}{a+d}}\\
&=\frac{\sqrt{2}}{\sqrt{a+d}} e^{-\frac{1}{2h(a+d)}\left(a\eta^2 -2b\eta y +  dy^2\right)}e^{\frac{i}{2h(a+d)}(b\eta^2 -by^2 +(a-d)\eta y)}\\
&=\frac{\sqrt{2}}{\sqrt{\Tr(A)}} e^{-\frac{1}{2h\Tr(A)}\lrang{A^{-1}\omega, \omega}}e^{\frac{i}{2h\Tr(A)}\lrang{AJ\omega, \omega}},
\end{split}
\end{equation}
for $J = \begin{bmatrix} 0 & 1 \\ -1 & 0 \end{bmatrix}$.
\end{proof}

We will often apply Lemma~\ref{lem:support} to a power of~$A$. Note for $t>0$, the power $A^t$ is still self-adjoint
with all positive entries and $M_A^t=M_{A^t}$ by~\eqref{e:multiplicativity}. For negative powers, we can use an identity following from~\cite[(2.18)]{dyatlov2021semiclassical}, the unitarity of $M_A$, and the fact that $U_\omega^*=U_\omega^{-1}=U_{-\omega}$:
$$
\lrang{U_\omega G_h,M_A^{-1}G_h}_{L^2}=\lrang{U_{A\omega}M_AG_h,G_h}_{L^2}=\overline{\lrang{U_{-A\omega}G_h,M_AG_h}_{L^2}},
$$
which shows that Lemma~\ref{lem:support} applies with $A,M_A$ replaced by $A^{-1},M_A^{-1}$.

Thus, from Lemma~\ref{lem:support}, we have for all $t\in\mathbb Z$,
\begin{equation}\label{eq:M^t_otimes}
\lrang{M^t_A G_h, G_h}_{L^2}=\frac{\sqrt{2}}{\sqrt{\lambda^t +\lambda^{-t}}}.
\end{equation}

For each $z \in \R^2$, we have the decomposition 
\begin{equation}\label{eq:stable_unstable_decomp}
z = a_+(z)v_+ + a_-(z) v_-,
\end{equation}
where $v_+$ is a unit length eigenvector of $A$ with eigenvalue $\lambda$ and $v_-$ is a unit length eigenvector of $A$ with eigenvalue $\lambda^{-1}$.
The slopes of $v_\pm$ are quadratic irrationals, which are well-known to be poorly approximated by rationals. 
This fact is exploited in~\cite{FNdB} to prove the following.
\begin{lemma}[\cite{FNdB}*{Lemma 2}]\label{lem:diophantine}
There exists $0<c_0 <1$, depending only on $A$, such that for all $l \in \Z^2 \setminus 0$,
\begin{equation*}
|a_+(l)| \geq \frac{c_0}{|a_-(l)|}.
\end{equation*}
\end{lemma}

We set the following notational convention:
\begin{notation}
We say that $f(h) \ll g(h)$ if $\frac{f(h)}{g(h)} \rightarrow 0$ as $h \rightarrow 0$.
\end{notation}

Using the above lemma, we can prove the following. The result  generalizes~\cite{FNdB}*{Proposition 1} to allow for a perturbation $r$. 
\begin{lemma}\label{lem:sum_over_integers}
Suppose $r \in \R^2$. Then for all $q \in \Z$,
\begin{equation}\label{eq:result_1}
\sum_{l \in \Z^2} \left|\lrang{M^q_A G_h, U_{l+r}  G_h}_{L^2}\right| \leq C(h\lambda^{|q|/2} + \lambda^{-|q|/2}).
\end{equation}
For  $|r| \ll \min (1, \lambda^{-|q|} (h\log h^{-1})^{-\frac{1}{2}})$ and $q \in \Z$ with $h \log h^{-1} \ll \lambda^{-|q|}$,
\begin{equation}\label{eq:result_2}
\sum_{l \in \Z^2 \setminus 0} \left|\lrang{M^q_A G_h, U_{l+r}  G_h}_{L^2}\right| \leq Ch\lambda^{|q|/2} .
\end{equation}
Finally, for $|r| \ll 1$ and all $q \in \Z$,
\begin{equation}\label{eq:result_3}
\sum_{l \in \Z^2 \setminus 0} \left|\lrang{M^q_A G_h, U_{l+r}  G_h}_{L^2}\right| \leq C(h\lambda^{|q|/2} +(h \log h^{-1})^{\frac{1}{4}}).
\end{equation}
For each inequality, $C$ does not depend on $r$ or $q$.
\end{lemma}

\begin{proof} 
1. From Lemma~\ref{lem:support} and~\eqref{eq:stable_unstable_decomp}, we know 
\begin{equation*}
\left|\lrang{M_A^q G_h, U_{l+r} G_{h}}_{L^2} \right| =\frac{\sqrt{2}}{\sqrt{\lambda^q + \lambda^{-q}}} e^{-\frac{a_+(l+r)^2 \lambda^{-q} + a_-(l+r)^2 \lambda^q}{2h (\lambda^q + \lambda^{-q})}}.
\end{equation*}
We assume that $q\geq 0$. If instead $q< 0$, the following argument is the same, up to exchanging $a_+(\cdot)$ and $a_-(\cdot)$  and replacing $q$ with $-q$. 

Note that
\begin{align*}
&\left\{z \in \R^2 : k-1 \leq \frac{a_+(z+r)^2 \lambda^{-q} + a_-(z+r)^2 \lambda^q}{2h (\lambda^q + \lambda^{-q})} <  k \right\}\\
&\quad \subset \left\{z \in \R^2 : \lambda^{-2q} a_+(z+r)^2 + a_-(z+r)^2 < 4kh\right\}\\
&\quad \subset \left\{z \in \R^2  : | a_+(z+r)| < 2\lambda^q \sqrt{kh} \text{ and } |a_-(z+r)| < 2 \sqrt{kh}\right\} \eqqcolon S_k.
\end{align*}
Then,
\begin{equation}\label{eq:bound1}
\sum_{l \in \Z^2} \left|\lrang{M^q_A G_h, U_{l+r}  G_h}_{L^2}\right|  \leq C\lambda^{-q/2} \sum_{k=1}^\infty \# (S_k \cap \Z^2) \cdot e^{-k}.
\end{equation}
The same inequality holds if we replace $\Z^2$ by $\Z^2\setminus 0$ on both sides.

To prove~\eqref{eq:result_1},~\eqref{eq:result_2}, and~\eqref{eq:result_3}, we will bound either $\# (S_k \cap \Z^2)$ or $ \# (S_k \cap \Z^2 \setminus 0)$.

2. We first show~\eqref{eq:result_1}. Suppose $l, l' \in S_k \cap \Z^2$. Then $|a_-(l-l')| \leq |a_-(l+r)| + |a_-(l'+r)| \leq 4 \sqrt{kh}$. Thus, if $l \neq l'$, from Lemma~\ref{lem:diophantine}, 
$|a_+(l+r) - a_+(l'+r)| = |a_+(l-l')| \geq \frac{c_0}{4 \sqrt{kh}}.$
Combining this with $|a_+ (l+r)|, |a_+ (l'+r)| \leq 2 \lambda^q \sqrt{kh}$ and $|a_- (l+r)| \leq 2  \sqrt{kh}$, we know 
\begin{equation}\label{eq:S_k_bound}
\# (S_k \cap \Z^2) \leq \left(2(2\lambda^q \sqrt{kh}) \left(\frac{4\sqrt{kh}}{c_0} \right) +1\right) \cdot \left(2(2  \sqrt{kh}) +1\right) \leq C \left(k^{3/2}h \lambda^q + \sqrt{k}\right).
\end{equation}

Therefore, 
\begin{equation}\label{eq:bound3}
\sum_{k=1}^\infty \# (S_k \cap \Z^2) \cdot e^{-k} \leq C \lambda^q h \sum_{k=1}^\infty k^{3/2} e^{-k} + \sum_{k=1}^\infty \sqrt{k} e^{-k}  \leq C (\lambda^q h +1).
\end{equation}
Combining~\eqref{eq:bound1} and~\eqref{eq:bound3}, we conclude~\eqref{eq:result_1}.

3. Now we prove~\eqref{eq:result_2}. Assume that $|r| \ll \min\{1, \lambda^{-q} (h \log h^{-1})^{-\frac{1}{2}} \}$ and $h \log h^{-1} \ll \lambda^{-q}$. Note that
$$S_k \subset \left\{z \in \R^2: |a_+(z)| \leq 2 \lambda^q \sqrt{kh} +|r|,\ |a_-(z)| \leq 2 \sqrt{kh} + |r| \right\} \eqqcolon \tilde{S}_k.$$

We claim that $\#(\tilde{S}_k \cap \Z^2) <2$ for $k \leq 2 \log h^{-1}$. 
Indeed, for $l, l' \in \tilde{S}_k \cap \Z^2$, $|a_-(l-l')| \leq 4 \sqrt{kh} + 2|r|$. Then if $l \neq l'$, by Lemma~\ref{lem:diophantine}, $|a_+(l) - a_+(l')| \geq \frac{c_0}{4 \sqrt{kh} + 2|r|}$. Since $|a_+(l)|, |a_+(l')| \leq 2 \lambda^q \sqrt{kh} +|r|$ and $|a_+(l)| \leq 2\sqrt{kh} +r \ll 1$, we see
$$\# (\tilde{S}_k \cap \Z^2) \leq
2\frac{4 \sqrt{kh} +2|r|}{c_0}(2 \lambda^q \sqrt{kh} + |r|) + 1 \leq C(\lambda^q kh +|r| \lambda^q \sqrt{kh} + |r|^2) + 1.$$
As $|r| \ll 1$, to show $\# (\tilde{S}_k \cap \Z^2)<2$, it suffices to prove $kh +|r|  \sqrt{kh} \ll \lambda^{-q}$. 
Since $|r| \ll \lambda^{-q}(h \log h^{-1})^{-\frac{1}{2}}$ and $h \log h^{-1} \ll \lambda^{-q}$, for $k \leq 2 \log h^{-1}$,
$$kh + |r| \sqrt{kh} \leq 2 h \log h^{-1} + |r| \sqrt{2 h \log h^{-1}} \ll \lambda^{-q}.$$ 
Thus, $\# (\tilde{S}_k \cap \Z^2) < 2$ for $k \leq 2 \log h^{-1}$.

Noting that $0 \in \tilde{S}_k$, see $\#(S_k \cap \Z^2 \setminus 0) \leq \#(\tilde{S}_k \cap \Z^2 \setminus 0) = 0$ for $k \leq 2 \log h^{-1}$. 
Therefore, using~\eqref{eq:S_k_bound},

\begin{equation}\label{eq:bound4}
\sum_{k=1}^\infty\# (S_k \cap \Z^2 \setminus 0) \cdot e^{-k} \leq \sum_{2 \log \frac{1}{h}< k}\# (S_k \cap \Z^2) \cdot e^{-k}    \leq C \lambda^q h \sum_{2 \log  \frac{1}{h} < k} e^{-k/2} + \sum_{2 \log  \frac{1}{h} < k} e^{-k/2} \leq  C\lambda^{q} h.
\end{equation}
Combining~\eqref{eq:bound1} and~\eqref{eq:bound4}, we conclude~\eqref{eq:result_2}.

4. Finally, we show~\eqref{eq:result_3}. Assume $|r| \ll 1$. If $1 \leq \lambda^{-q} (h \log h^{-1})^{-1/2}$, then $|r| \ll \min\{1, \lambda^{-q} (h \log h^{-1})^{-\frac{1}{2}} \}$ and $h \log h^{-1} \ll \lambda^{-q}$. Thus we have~\eqref{eq:result_2}, a stronger statement than $\eqref{eq:result_3}$.

Otherwise,
$\lambda^{-q} (h \log h^{-1})^{-\frac{1}{2}} < 1$, which gives $\lambda^{-q/2} < (h \log h^{-1})^{\frac{1}{4}}$. Then,~\eqref{eq:result_3} follows from~\eqref{eq:result_1}. 
\end{proof}

 Next we prove the following lemma which we later will use to show the contribution of terms with $|t|$ close to $P/4$ is negligible in the semiclassical limit.  The argument is adapted from~\cite{FNdB}*{Proposition 2}. 

\begin{lemma}\label{lem:size_u}
For $\alpha =\alpha(h), \beta=\beta(h)$, set $I =[\beta, \beta +\alpha] \subset [-\frac{P}{2}, \frac{P}{2}]$, where $\alpha \rightarrow \infty$ as $h \rightarrow 0$. Set
\begin{equation}\label{eq:uI_definition}
u_{I}  \coloneqq \frac{1}{\sqrt{P}} \sum_{t =\beta}^{\beta+\alpha-1} \left(e^{-i \frac{\phi}{P}t} M^t_N G_N \right) \otimes \left(e^{-i \frac{\phi}{P}\left(t + \frac{P}{2}\right)} M^{t+ \frac{P}{2} }_N G_N \right).
\end{equation}
Then as $h \rightarrow 0$, $$\lrang{u_I, u_I}_\cH = \frac{\alpha}{P} S_1(\lambda) + o(1),$$
where $S_1(\lambda)$ is a strictly positive smooth function.

\end{lemma}
\begin{proof} 
Using~\eqref{eq:Pi_adjoint},~\eqref{eq:adjoint_rule}, and~\eqref{eq:G_N},  we have
\begin{align}
&\lrang{u_I, u_I}_\cH \nonumber\\
&=\frac{1}{P} \sum_{t, s =\beta}^{\beta+\alpha-1} \lrang{e^{-i \frac{\phi}{P} t} M_N^t G_N,  e^{-i \frac{\phi}{P} s} M_N^s G_N}_\cH\lrang{e^{-i \frac{\phi}{P} \left(t+\frac{P}{2}\right)} M_N^{t + \frac{P}{2}} G_N , e^{-i \frac{\phi}{P} \left(s+\frac{P}{2}\right)} M_N^{s + \frac{P}{2}} G_N}_\cH \nonumber\\
&=\frac{1}{P} \sum_{t, s=0}^{\alpha-1}e^{-\frac{i \phi}{P}(2t -2s)}  \lrang{  M_N^{t-s} G_N,   G_N}_\cH^2 \nonumber\\
&=\frac{1}{P} \sum_{t, s=0}^{\alpha-1}e^{-\frac{i \phi}{P}(2t -2s)}  \lrang{\Pi_N(0)  M_A^{t-s} G_h,  \Pi_N(0)  G_h}_\cH^2 \nonumber\\
&=\frac{1}{P}  \sum_{t, s=0}^{\alpha-1} e^{-\frac{i \phi}{P}(2t -2s)} \left(\sum_{l \in \Z^2}  \lrang{ U_l M^{t-s}_A G_h ,   G_h}_{L^2}\right)^2  \nonumber\\
&=\frac{1}{P} \sum_{t=-\alpha +1}^{\alpha-1} \left(\alpha-|t| \right)  e^{-\frac{i \phi}{P}2t} \sum_{l, m \in \Z^2}  \lrang{ U_l M^{t }_A G_h,   G_h}_{L^2}  \lrang{U_m M^{t }_A G_h,   G_h }_{L^2}. \label{eq:uI_reduction}
\end{align}
We split~\eqref{eq:uI_reduction} into the term when $l=m =0$ and a sum over  the remaining terms, i.e., $(l, m) \in \Z^4 \setminus 0$. We first show that the sum over $\Z^4 \setminus 0$ decays to zero.

We apply~\eqref{eq:result_1} and~\eqref{eq:result_3} from Lemma~\ref{lem:sum_over_integers} with $r=0$ to know
\begin{align*}
&\left|\frac{1}{P} \sum_{t=-\alpha +1}^{\alpha-1} \left(\alpha-|t| \right)  e^{-\frac{i \phi}{P}2t} \sum_{(l, m) \in \Z^4 \setminus 0} \lrang{U_l M^{t}_A G_h ,   G_h  }_{L^2} \lrang{ U_m M^{t}_A G_h ,    G_h }_{L^2}\right|\\
&\leq \frac{1}{P} \sum_{t=-\alpha +1}^{\alpha-1} \left(\alpha -|t| \right)  \left| \sum_{l \in \Z^2 \setminus 0}  \lrang{U_l M^{t}_A G_h  ,   G_h }_{L^2} \right| \left| \sum_{m \in \Z^2} \lrang{U_m M^{t}_A G_h ,  G_h }_{L^2}\right|\\
&\quad + \frac{1}{P}\sum_{t=-\alpha +1}^{\alpha-1} \left(\alpha - |t| \right)  \left| \sum_{l \in \Z^2 \setminus 0}\lrang{ U_l M^{t}_A G_h,   G_h  }_{L^2}\right| \left| \lrang{ M^{t}_A G_h ,   G_h }_{L^2}\right|\\
&\leq \frac{C}{P} \sum_{t=-\alpha+1}^{\alpha-1} \left(\alpha-|t| \right)  \sum_{l \in \Z^2 \setminus 0} \left| \lrang{U_l M^{t}_A G_h,   G_h}_{L^2} \right|\\
&\leq \frac{C}{P} \sum_{t=-\alpha+1}^{\alpha-1} \left(\alpha-|t| \right) \left(h\lambda^{\frac{|t|}{2}} +(h \log h^{-1})^{\frac{1}{4}} \right) \\
& \leq \frac{C}{P}\left(h\frac{\lambda^{\frac{\alpha+1}{2}} -\alpha(\lambda^{\frac{1}{2}}-1) -\lambda^{\frac{1}{2}}}{(\lambda^{\frac{1}{2}} -1)^2}+ (h P)^{\frac{1}{4}}\alpha^2\right)\\
& \leq \frac{C}{P}\left(h\lambda^{\frac{\alpha}{2}} + (h P)^{\frac{1}{4}}\alpha^2\right) \leq \frac{C}{P}\left(1 + h^{\frac{1}{4}}P^{\frac{9}{4}}\right) \rightarrow 0,
\end{align*} 
where we used the fact that $\alpha \leq P = \frac{2 \log h^{-1}}{\log \lambda} + \cO(1)$.

Therefore, using~\eqref{eq:M^t_otimes}, 
\begin{align*}
\eqref{eq:uI_reduction} &= \frac{1}{P} \sum_{t=-\alpha+1}^{\alpha -1} \left(\alpha-|t| \right)  e^{-\frac{i \phi}{P}2t}  \lrang{M^{t}_A G_h,   G_h }_{L^2}^{2} +o(1) \\
&= \frac{1}{P} \sum_{t=-\alpha+1}^{\alpha -1} \left(\alpha-|t| \right)  e^{-\frac{2i \phi}{P}t}  \frac{2}{\lambda^{t} + \lambda^{-t}} +o(1)\\
&= \frac{\alpha}{P} S_1(\lambda, \alpha, \phi, P) - \frac{1}{P} S_2(\lambda, \alpha, \phi, P) +o(1), 
\end{align*}
where
$$
S_1(\lambda, \alpha, \phi, P) \coloneqq \sum_{t=-\alpha+1}^{\alpha -1} e^{-\frac{2 i  \phi}{P}t}  \frac{2}{\lambda^{t} + \lambda^{-t}}  \quad \text{and} \quad S_2(\lambda, \alpha, \phi, P) \coloneqq  \sum_{t=-\alpha+1}^{\alpha -1} |t| e^{-\frac{2i \phi}{P}t}  \frac{2}{\lambda^{t} + \lambda^{-t}}.
$$
Recall that $\phi \in  [-\pi, \pi)$; in particular, $\phi/P$ converges to 0.
From these formulas, it is clear that $S_1(\lambda) \coloneqq \lim_{h \rightarrow 0} S_1(\lambda, \alpha, \phi, P) = \sum_{t= -\infty}^\infty \frac{2}{\lambda^t + \lambda^{-t}}$ and $S_2(\lambda) \coloneqq
\lim_{h \rightarrow 0} S_2(\lambda, \alpha, \phi, P)$ exist and  $S_1$ is strictly positive. 
\end{proof}

\subsection{Proof of Proposition~\ref{prop:counterexample}}

In this section, we show Proposition~\ref{prop:counterexample}: after normalization, the eigenfunctions $u$ (defined in~\eqref{eq:u_def}) weakly converge to the semiclassical measure $\mu$ (defined in~\eqref{eq:semiclassical_measure}). This, in turn, proves Theorem~\ref{thm:counterexample}. The position space concentration of the eigenfunctions is demonstrated numerically in Figure~\ref{fig:eigenfunction}. 

\begin{figure}
    \centering
    \includegraphics[width=6cm, height=6cm]{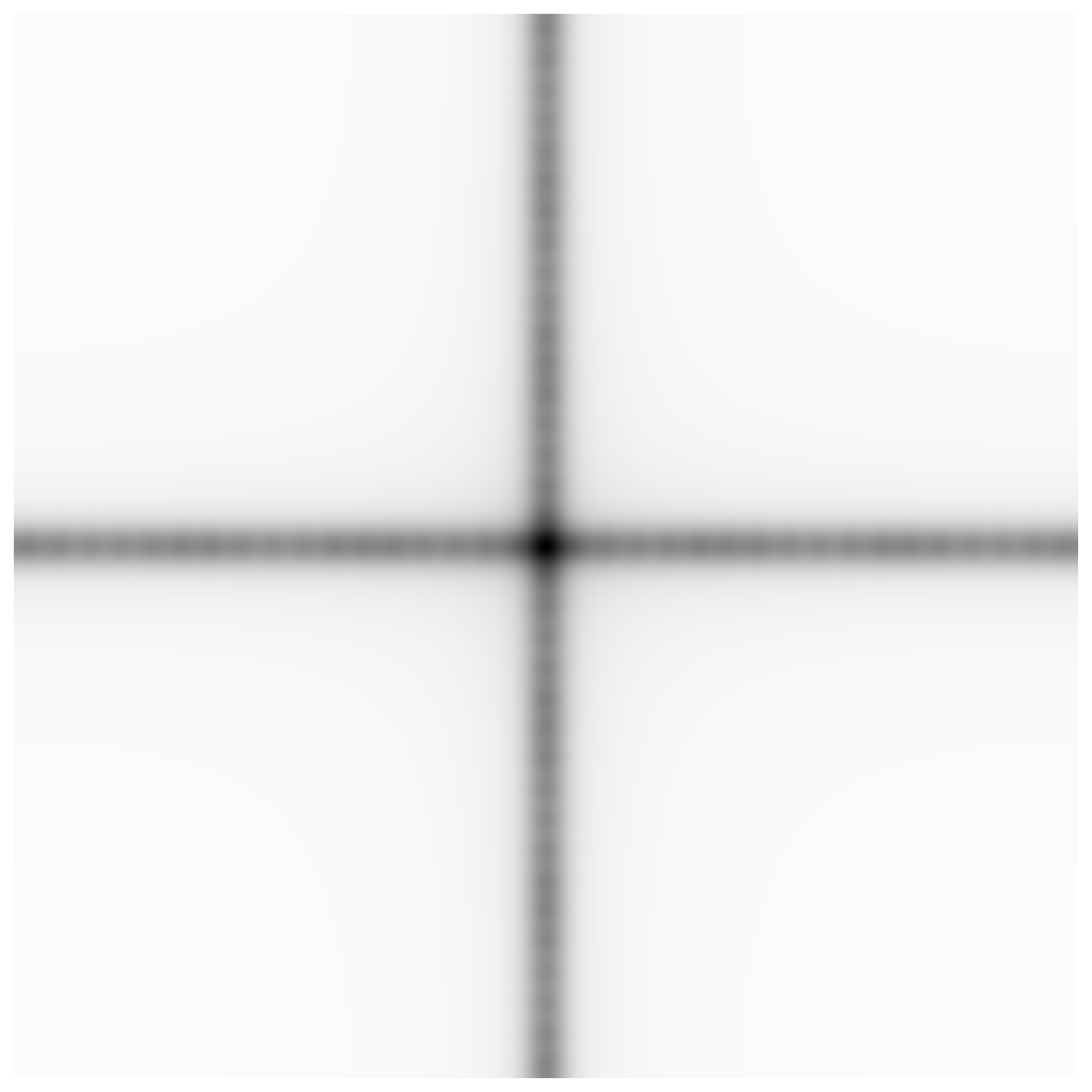}
    \caption{Plot of the concentration in position space of the eigenfunction defined in~\eqref{eq:u_def}
corresponding to $A = \begin{bmatrix} 5 & 2 \\ 2 & 1 \end{bmatrix}$, $k=6$, and $N_k= 6930$. Specifically, we  blurred the squared absolute value of the eigenfunction. This blurring corresponds to taking the macroscopic limit, as we can only integrate against slowly varying functions. To create a more intelligible picture, we have shifted the eigenfunction to be centered at $(1/2, 1/2)$, instead of the origin. 
}
    \label{fig:eigenfunction}
\end{figure}

Choose an integer-valued function $\Theta(h)$ such that 
\begin{equation}\label{eq:Theta_definition}
1 \ll \Theta(h) \ll P \ll \lambda^{\frac{\Theta(h)}{4}}.
\end{equation}
One can for instance put $\Theta(h)=\lfloor\sqrt{\log(1/h)}\rfloor$.

We first state the following two key lemmas, reserving their proofs for after we use them to prove Theorem~\ref{thm:counterexample}. In both lemmas, we use $\Theta(h)$ to bound away from $|t| = P/4$ and $|t| = P/2$, an idea taken from~\cite{FNdB}*{Propositions 5--6}.  Specifically, we use the disjoint union: 
$$\left[-\frac{P}{2}, \frac{P}{2}-1\right] = I_L \sqcup I_E \sqcup I_0,$$
$$I_0 \coloneqq  \left\{\frac{P}{4} - \Theta(h) < |t| < \frac{P}{4}+\Theta(h)\right\} \cup  \left\{-\frac{P}{2} \leq t < -\frac{P}{2} + \Theta(h)\right\} \cup  \left\{\frac{P}{2} -\Theta(h)  < t < \frac{P}{2} \right\},$$
$$I_L\coloneqq \left\{|t| \leq  \frac{P}{4}-\Theta(h)\right\}, \quad I_E\coloneqq \left\{ \frac{P}{4} + \Theta(h) \leq |t| \leq \frac{P}{2} -\Theta(h) \right\}.$$
Note that for $j \in \Z^2$, the quantum translation $U_{j/N}$ descends to an operator on $\mathcal H_N(0)$. It is given by 
\begin{equation}\label{eq:trig_polynomial}
U_{j/N} = \op_N(e^{2 \pi i \sigma(j, z)}).
\end{equation}

\begin{lemma}\label{lem:locbehavior}
Suppose $j \in \Z^2$. 
Then, 
\begin{align*}
&\frac{1}{P}\sum_{s= - \frac{P}{2}}^{\frac{P}{2}-1} \sum_{t \in I_L} e^{-2i\frac{\phi}{P}(t-s)}\lrang{U_{j/N}M^t_N G_N, M_N^s G_N}_\cH \lrang{M_N^tG_N, M^s_N G_N}_\cH\\
&=\frac{1}{P}\sum_{s= - \frac{P}{2}}^{\frac{P}{2}-1} \sum_{t \in I_L} e^{-2i\frac{\phi}{P}(t-s)} \lrang{G_N,  M_N^{s-t} G_N}_\cH^2  +o(1).
\end{align*}
\end{lemma}

\begin{lemma}\label{lem:ergbehavior}
Suppose $j \in \Z^2\setminus 0$, $k \in \Z^2$. Then,
\begin{align*}
\frac{1}{P}\sum_{s =-\frac{P}{2}}^{\frac{P}{2}-1} \sum_{t \in I_E} e^{-2i\frac{\phi}{P}(t-s)} \Big\langle U_{j/N} M_N^t G_N, M_N^s G_N\Big \rangle_\cH \Big\langle U_{k/N}M_N^{t+\frac{P}{2}} G_N, M_N^{s+\frac{P}{2}}G_N\Big \rangle_\cH
\end{align*}
converges to 0 as $h \rightarrow 0$.
\end{lemma}

Assuming the above two lemmas, we can prove Proposition~\ref{prop:counterexample} which implies Theorem~\ref{thm:counterexample}.

\begin{proof}[Proof of Proposition~\ref{prop:counterexample}]
1. Let $A \in \Sp(2, \Z)$ be a hyperbolic, symmetric matrix with positive entries and eigenvalues $\lambda >1$, $\lambda^{-1}$. From Lemma~\ref{lem:size_u}, we recall that $\langle u,u\rangle_{\mathcal H}\to S_1(\lambda)>0$, where $u$ (defined in~\eqref{eq:u_def}) is a sequence of eigenfunctions for the quantization of $A \oplus A$. 

By~\eqref{eq:trig_polynomial} and the density of trigonometric polynomials and tensor products, it suffices to examine $U_{j/N} \otimes U_{k/N}$ for $j,k \in \Z^2$ in lieu of a general $a \in C^\infty(\bT^4)$. Thus, it is enough to show
\begin{equation*}
\lim_{h \rightarrow 0} \frac{\lrang{\left(U_{j/N} \otimes U_{k/N}\right) u, u}_\cH}{\lrang{u, u}_\cH} = \begin{cases} 
1 & \text{ if } j=k =0\\
\frac{1}{2} & \text{ if } j =0, k \neq 0\\
\frac{1}{2} & \text{ if } j \neq 0, k = 0\\
0 & \text{ if } j, k \neq 0.
\end{cases}
\end{equation*}

First note that when $j=0$, $U_{j/N}=I$. Therefore, when $j=k=0$, $$\lrang{\left(U_{j/N} \otimes U_{k/N}\right) u, u}_\cH= \lrang{u, u}_\cH,$$ which proves the first case. 

2. We now focus on the other cases. We examine
\begin{align}
&\lrang{\left(U_{j/N} \otimes U_{k/N}\right) u, u}_\cH \nonumber \\
&= \frac{1}{P}\sum_{s = -\frac{P}{2}}^{\frac{P}{2}-1} \sum_{t = -\frac{P}{2}}^{\frac{P}{2}-1} e^{-2i\frac{\phi}{P}(t-s)}\Big\langle U_{j/N} M^t_N G_N, M_N^s G_N\Big\rangle_\cH\Big\langle U_{k/N} M^{t+\frac{P}{2}}_N G_N, M_N^{s+\frac{P}{2}} G_N\Big\rangle_\cH. \label{eq:summands}
\end{align}

Suppose $j \neq 0$, $k=0$. 
As discussed at the end of \S\ref{subsection:assumptions}, $e^{-i\frac{\phi}{P}t} M^t_N G_N$ localizes at the origin for $t \in I_L$ and  equidistributes  for $t \in I_E$. In between, for $P/4 -\Theta(h) \leq |t| \leq P/4 + \Theta(h)$, $e^{-i\frac{\phi}{P}t} M^t_N G_N$ transitions between these behaviors. We also cut away from $|t| = P/2$ in $I_E$, which simplifies the proof of Lemma~\ref{lem:ergbehavior}.
We will use Lemma~\ref{lem:locbehavior} to estimate the terms in~\eqref{eq:summands} for  $t \in I_L$ and  Lemma~\ref{lem:ergbehavior} to estimate the terms in~\eqref{eq:summands} for $t \in I_E$. However, we require a different strategy  for $t \in I_0$;  we claim these terms are negligible.

First, recall the definition of $u_I$ from~\eqref{eq:uI_definition}. If $I$ is an interval with  $I \subset I_0$, then $|I| \leq 2\Theta(h)$. Then by  Lemma~\ref{lem:size_u} and \eqref{eq:Theta_definition}, we know
$\|u_I\|_{\mathcal H}=o(1)$. Therefore for all $j, k \in \Z^2$ (including $j =k =0$) and $I \subset I_0$,
\begin{equation}\label{eq:cuttheedges}
\left| \lrang{(U_{j/N} \otimes U_{k/N})u_I, u}_\cH\right| \leq \|u_I\|_\cH \|u\|_\cH =o(1).   
\end{equation}

Thus by Lemma~\ref{lem:locbehavior}, Lemma~\ref{lem:ergbehavior}, and~\eqref{eq:cuttheedges},
\begin{align*}
\eqref{eq:summands} = \frac{1}{P}\sum_{s =-\frac{P}{2}}^{\frac{P}{2}-1} \sum_{t = -\frac{P}{4}}^{\frac{P}{4}-1}e^{-i\frac{\phi}{P}(2t-2s)}\Big\langle  G_N, M_N^{s-t} G_N\Big\rangle_\cH^2  +o(1).
\end{align*}
Note that the terms in the sum only depend on the value of $(t-s) \bmod P$.

As 
\begin{align*}
\lrang{u, u}_\cH &= \frac{1}{P} \sum_{s =-\frac{P}{2}}^{\frac{P}{2}-1} \sum_{t = -\frac{P}{2}}^{\frac{P}{2}-1} e^{-i \frac{\phi}{P}(2t -2s)} \Big\langle  G_N, M^{s-t}_N G_N\Big\rangle_\cH^2\\
&= \frac{2}{P} \sum_{s =-\frac{P}{2}}^{\frac{P}{2}-1} \sum_{t = -\frac{P}{4}}^{\frac{P}{4}-1} e^{-i \frac{\phi}{P}(2t -2s)} \Big\langle  G_N, M^{s-t}_N G_N\Big\rangle_\cH^2,
\end{align*}
we see $\lrang{(U_{j/N} \otimes U_{k/ N}) u, u}_\cH/\lrang{u, u}_\cH \rightarrow 1/2$. Again, the terms in the sum only depend on the value of $(t-s) \bmod P$.

Using the fact that $M_{N}^{P} = e^{i \phi} I$, a similar argument shows when $j\neq 0$ and $k =0$,  $\lrang{(U_{j/N} \otimes U_{k/ N}) u, u}_\cH/\lrang{u, u}_\cH \rightarrow 1/2$. 

3. Finally, we examine the case where $j, k \neq 0$. Let $T_* \coloneqq [-\frac{P}{2}, -\frac{P}{4}-1] \cup [\frac{P}{4}, \frac{P}{2}-1]$. We note that $[-P/2, P/2 - 1] \bmod P = T_* \cup (T_* + P/2) \bmod P$. Using this decomposition to split up the sum over $t$ in~\eqref{eq:summands} and exploiting the $P$-periodicity of $M_N$, we have 
\begin{align*}
\eqref{eq:summands} &=\frac{1}{P}\sum_{s = -\frac{P}{2}}^{\frac{P}{2}-1} \sum_{t \in T_*} e^{-i\frac{\phi}{P}(2t-2s)}\Big\langle U_{j/N} M^t_N G_N, M_N^s G_N\Big\rangle_\cH\Big\langle U_{k/N} M^{t+\frac{P}{2}}_N G_N, M_N^{s+\frac{P}{2}} G_N\Big\rangle_\cH\\
&\quad + \frac{1}{P}\sum_{s = -\frac{P}{2}}^{\frac{P}{2}-1}  \sum_{t \in T_*} e^{-i\frac{\phi}{P}(2t-2s)}\Big\langle U_{k/N} M^t_N G_N, M_N^s G_N\Big\rangle_\cH\Big\langle U_{j/N} M^{t+\frac{P}{2}}_N G_N, M_N^{s+\frac{P}{2}}G_N \Big\rangle_\cH.
\end{align*}
Thus by~\eqref{eq:cuttheedges} and Lemma~\ref{lem:ergbehavior}, we conclude  $\lrang{(U_{j/N} \otimes U_{k/ N}) u, u}_\cH \rightarrow 0.$
\end{proof}

\medskip

\begin{proof}[Proof of Lemma~\ref{lem:locbehavior}]
1. First note that it suffices to show that
\begin{equation}\label{eq:locconvergence}
\frac{1}{P}\sum_{|s| \leq \frac{P}{2}} \sum_{t \in I_L} \left| \left(\lrang{U_{j/N}M^t_N G_N, M_N^s G_N}_\cH -\lrang{ G_N, M^{s-t}_N G_N}_\cH \right)\lrang{G_N, M^{s-t}_N G_N}_\cH\right|
\end{equation}
converges to 0 for $j \in \Z^2$. 

For each $s, t$, choose $q_{s,t} \in s-t + P\Z$ such that $|q_{s,t}|\leq P/2$. As $M^P_N =e^{i \phi}I$, we know that $M_N^{s-t} = e^{i C_{s,t}} M_N^{q_{s,t}}$, where $C_{s,t} \in \{0, \pm \phi\}$. 
Using~\eqref{eq:intertwining},
\begin{align*}
\left|\lrang{U_{j/N}M^t_N G_N, M_N^s G_N} -\lrang{G_N, M^{s-t}_N G_N}  \right|=\left|  \big\langle U_{A^{-t}j/N} G_N, M_N^{q_{s,t}} G_N\big\rangle - \big\langle G_N, M_N^{q_{s,t}} G_N\big\rangle  \right|.
\end{align*}

By~\eqref{eq:G_N}, ~\eqref{eq:Pi_adjoint}, ~\eqref{eq:S_N}, and the fact that $U_\omega U_{\omega'} = e^{\frac{i}{2h} \sigma(\omega, \omega')} U_{\omega + \omega'}$, we know
\begin{align}
\eqref{eq:locconvergence} &\leq \frac{C}{P}\sum_{|s| \leq \frac{P}{2}} \sum_{t \in I_L} \left|\lrang{U_{A^{-t}j/N}G_N, M_N^{q_{s,t}}G_N}_\cH -\lrang{G_N, M^{q_{s,t}}_N G_N}_\cH \right| \nonumber\\
&=\frac{C}{P}\sum_{|s| \leq \frac{P}{2}} \sum_{t \in I_L} \left|\lrang{\Pi_N(0) U_{A^{-t}j/N}G_h, \Pi_N(0) M_A^{q_{s,t}}G_h}_\cH -\lrang{\Pi_N(0) G_h, \Pi_N(0) M^{q_{s,t}}_A G_h}_\cH \right| \nonumber\\
&\leq \frac{C}{P}\sum_{|s| \leq \frac{P}{2}} \sum_{t \in I_L} \left|\lrang{U_{A^{-t}j/N}G_h, M_A^{q_{s,t}}G_h}_{L^2} -\lrang{G_h, M_A^{q_{s,t}} G_h}_{L^2} \right| \nonumber \\
&\quad+\frac{C}{P}\sum_{|s| \leq \frac{P}{2}} \sum_{t \in I_L}  \left(\sum_{l \in \Z^2\setminus 0} \left|\lrang{U_{l+A^{-t}j/N}G_h, M_A^{q_{s,t}}G_h}_{L^2}\right|+ \sum_{l \in \Z^2\setminus 0}\left|\lrang{U_l G_h, M_A^{q_{s,t}}G_h}_{L^2}\right| \right). \nonumber 
\end{align}

Thus using the triangle inequality, to prove Lemma~\ref{lem:locbehavior}, it suffices to show that 
\begin{equation}\label{eq:l=0_estimate}
\max_{|s| \leq \frac{P}{2},t \in I_L}  \left|\lrang{U_{A^{-t}j/N}G_h, M_A^{q_{s,t}}G_h}_{L^2} -\lrang{G_h, M_A^{q_{s,t}} G_h}_{L^2} \right| = o(P^{-1})
\end{equation}
and
\begin{equation}\label{eq:lneq0_estimate}
\max_{|s| \leq \frac{P}{2},t \in I_L}   \left(\sum_{l \in \Z^2\setminus 0} \left|\lrang{U_{l+A^{-t}j/N}G_h, M_A^{q_{s,t}}G_h}_{L^2}\right|+ \sum_{l \in \Z^2\setminus 0}\left|\lrang{U_l G_h, M_A^{q_{s,t}}G_h}_{L^2}\right| \right) = o(P^{-1}).
\end{equation} 

2. We begin by proving~\eqref{eq:l=0_estimate}. From Lemma~\ref{lem:support}, recalling that $J = \begin{bmatrix} 0 & 1 \\ -1 & 0 \end{bmatrix}$,
\begin{equation*}
\lrang{U_\omega G_h, M_A^q G_h}_{L^2} = \frac{\sqrt{2}}{\sqrt{\lambda^q + \lambda^{-q}}} e^{-\frac{\lrang{A^{-q} \omega, \omega}}{2h(\lambda^q + \lambda^{-q})}}e^{i \frac{\lrang{A^qJ \omega, \omega}}{2h(\lambda^q + \lambda^{-q})}}.
\end{equation*}

Therefore, setting $\omega = A^{-t}j/N$,
\begin{equation}\label{eq:inner_product_difference}
\left|\lrang{U_{A^{-t}j/N}G_h, M_A^{q}G_h}_{L^2} -\lrang{G_h, M_A^{q} G_h}_{L^2} \right| \leq  \left|1 - e^{-\frac{\lrang{A^{-q} \omega, \omega}}{2h(\lambda^q + \lambda^{-q})}+i \frac{\lrang{A^qJ \omega, \omega}}{2h(\lambda^q + \lambda^{-q})}}\right|.
\end{equation}
Since $|1-e^z| \leq C|z|$ for $|z| \leq 1$, to further bound~\eqref{eq:inner_product_difference}, we want to show $|\frac{\lrang{A^{-q} \omega, \omega}}{2h(\lambda^q + \lambda^{-q})}-i \frac{\lrang{A^qJ \omega, \omega}}{2h(\lambda^q + \lambda^{-q})}| = o(P^{-1}) \leq 1$. 
In fact, we prove the stronger claim: for any fixed, $t$-independent matrix $B$  and for all $t \in I_L$, $\frac{|\lrang{A^{\pm q} B \omega, \omega}|}{2h(\lambda^q +\lambda^{-q})} =o(P^{-1})$.  We see this in the following calculation: 
\begin{equation*}\label{eq:generalized_bound}
\frac{|\lrang{A^{\pm q} B \omega, \omega}|}{2h(\lambda^q +\lambda^{-q})} \leq \frac{ |A^{\pm q} B \omega| |\omega|}{2h\lambda^{|q|}} \leq \frac{C|A^{-t}j/N|^2}{h} \leq \frac{C\lambda^{2|t|}}{N} \leq \frac{C\lambda^{\frac{P}{2}-2 \Theta(h)}}{N} \leq C \lambda^{-2 \Theta(h)}=o(P^{-1}).
\end{equation*}

Thus setting $B=I$ and $B=J$, from~\eqref{eq:inner_product_difference} we have
\begin{align*}
\left|\lrang{U_{A^{-t}j/N}G_h, M_A^{q}G_h}_{L^2} -\lrang{G_h, M_A^{q} G_h}_{L^2} \right|  &\leq C \left|\frac{\lrang{A^{-q} \omega, \omega}}{2h(\lambda^q + \lambda^{-q})}-i \frac{\lrang{A^qJ \omega, \omega}}{2h(\lambda^q + \lambda^{-q})}\right|\\
&=o(P^{-1}),    
\end{align*}
which proves~\eqref{eq:l=0_estimate}.

3. We now show~\eqref{eq:lneq0_estimate}. 
As $t \in I_L$, we know
$$|A^{-t}j/N| \leq \frac{C}{N} \lambda^{|t|} \leq \frac{C}{N} \lambda^{\frac{P}{4}-\Theta(h)}\leq C \sqrt{h} \lambda^{-\Theta(h)} \rightarrow 0.$$

Therefore, we apply~\eqref{eq:result_3} from Lemma~\ref{lem:sum_over_integers} to know for all  $|s| \leq \frac{P}{2}$ and $t \in I_L$,
\begin{align*}
\sum_{l \in \Z^2\setminus 0} \left( \left|\lrang{U_{l+A^{-t}j/N}G_h, M_A^{q}G_h}_{L^2}\right|+  \left|\lrang{U_l G_h, M_A^{q}G_h}_{L^2}\right| \right)& \leq C\left(h\lambda^{\frac{|q|}{2}} + (h \log h^{-1})^{\frac{1}{4}}\right)\\
&\leq C \left(\sqrt{h} + (h \log h^{-1})^{\frac{1}{4}}\right)\\
&= o(P^{-1}),
\end{align*}
which concludes the proof of~\eqref{eq:lneq0_estimate}. 
\end{proof}

\medskip

\begin{proof}[Proof of Lemma~\ref{lem:ergbehavior}]
Our proof uses many ideas from~\cite{FNdB}*{Proposition 6}. 

1. First note that it suffices to show 
\begin{equation}\label{eq:ergsum}
\frac{1}{P}\sum_{|s| \leq \frac{P}{2}} \sum_{t \in I_E}\left|\Big\langle U_{j/N} M_N^t G_N, M_N^s G_N\Big \rangle_\cH \Big\langle U_{k/N}M_N^{t+\frac{P}{2}} G_N, M_N^{s+\frac{P}{2}}G_N\Big \rangle_\cH\right|
\end{equation}
converges to zero for $j \in \Z^2 \setminus 0$, $k \in \Z^2$.

For each $s$ and $t$, choose $q_{s,t} \in s-t + P\Z$ such that $|q_{s,t}|\leq P/2$. As $M^P_N =e^{i \phi}I$, we know that $M_N^{s-t} = e^{i C_{s,t}} M_N^{q_{s,t}}$, where $C_{s,t} \in \{0, \pm \phi\}$.

From~\eqref{eq:MA},~\eqref{eq:G_N}, and the facts that $U_\omega U_{\omega'} = e^{\frac{i}{2h} \sigma(\omega, \omega')} U_{\omega + \omega'}$ and $\|G_N\|_{\cH} \leq C$ uniformly in $N$, we have 
\begin{align}
\eqref{eq:ergsum} &=\frac{1}{P}\sum_{|s| \leq \frac{P}{2}} \sum_{t \in I_E}\left|\Big \langle U_{j/N} M_N^t G_N, M_N^s G_N\Big \rangle_\cH \Big \langle U_{k/N}M_N^{t+\frac{P}{2}} G_N, M_N^{s+\frac{P}{2}}G_N\Big \rangle_\cH\right| \nonumber\\
&\leq \frac{C}{P}\sum_{|s| \leq \frac{P}{2}} \sum_{t \in I_E} \left|\lrang{U_{A^{-t}j/N} G_N, M_N^{q_{s,t}} G_N}_\cH \right| \nonumber\\
&= \frac{C}{P}\sum_{|s| \leq \frac{P}{2}} \sum_{t \in I_E} \left|\lrang{\Pi_N(0) U_{A^{-t}j/N} G_h, \Pi_N(0) M_A^{q_{s,t}} G_h}_\cH \right| \nonumber\\
& \leq \frac{C}{P}\sum_{|s| \leq \frac{P}{2}} \sum_{t \in I_E} \left| \lrang{U_{A^{-t}j/N} G_h, M_A^{q_{s,t}} G_h}_{L^2} \right|\label{eq:ergodicequation_l=0}\\
&\quad + \frac{C}{P}\sum_{|s| \leq \frac{P}{2}} \sum_{t \in I_E} \sum_{l \in \Z^2 \setminus 0} \left| \lrang{U_{l+A^{-t}j/N} G_h, M_A^{q_{s,t}} G_h}_{L^2} \right|. \label{eq:ergodicequation_sum}
\end{align}

2. We will show that~\eqref{eq:ergodicequation_l=0} and~\eqref{eq:ergodicequation_sum} decay to zero, starting with~\eqref{eq:ergodicequation_l=0}. For fixed $t$, $\{q_{s,t} : |s| \leq P/2\} = \{q : |q| \leq P/2\}$. Thus, we can sum over $q$ and $t$ instead of over $s$ and $t$. 
We split into two sums: 
$$
\eqref{eq:ergodicequation_l=0} = \frac{C}{P} \sum_{|q| \leq \frac{\Theta(h)}{2 }} \sum_{t \in I_E} \left|\lrang{U_{A^{-t}j/N} G_h, M_A^q G_h}_{L^2}\right| +  \frac{C}{P} \sum_{|q|= \frac{\Theta(h)}{2 }+1}^{\frac{P}{2}} \sum_{t \in I_E} \left|\lrang{U_{A^{-t}j/N} G_h, M_A^q G_h}_{L^2}\right|.
$$
We first analyze the sum over  $|q| \leq \frac{\Theta(h)}{2 }$. By the triangle inequality, it suffices to show $\max_{t \in I_E} \sum_{|q| \leq \frac{\Theta(h)}{2 }} \left|\lrang{U_{A^{-t}j/N} G_h, M_A^q G_h}_{L^2}\right| =o(1)$. Recall that $v_+$ is a unit length eigenvector of $A$ with eigenvalue $\lambda$ and  $v_-$ is a unit length eigenvector of $A$ with eigenvalue $\lambda^{-1}$. For each $z \in \R^2$, we have the decomposition $z= a_+(z) v_+ + a_-(z) v_-$. Thus from Lemma~\ref{lem:support}, we know  
$$\left|\lrang{U_{A^{-t} j/N}G_h, M_A^q G_h}_{L^2} \right| = \sqrt{\frac{2}{\lambda^q+\lambda^{-q}}} e^{-\frac{1}{2}\left(\frac{a_+(j)^2 \lambda^{-2t-q}}{h(\lambda^q + \lambda^{-q})N^2} +\frac{a_-(j)^2 \lambda^{2t+q}}{h(\lambda^q + \lambda^{-q})N^2} \right)}.$$

Importantly, as $j \in \Z^2 \setminus 0$ and the slopes of $v_+, v_-$ are irrational, we know $a_+(j) , a_-(j)\neq 0$.
Thus, letting $C$ and  $C'$ be constants that may change from line to line, for $t \in I_E$:

\begin{align}
\sum_{|q| \leq \frac{\Theta(h)}{2}}  \left|\lrang{U_{A^{-t}j/N} G_h, M_A^q G_h}_{L^2}\right| \nonumber 
&\leq C \sum_{|q| \leq \frac{\Theta(h)}{2}}   \lambda^{-\frac{|q|}{2}} e^{-\frac{C' h \lambda^{|2t+q|}}{\lambda^{|q|}}} \nonumber \\
&\leq Ce^{-C'h \lambda^{\frac{P}{2} + \Theta(h)}} \sum_{|q| \leq \frac{\Theta(h)}{2}}  \lambda^{-\frac{|q|}{2}}  \nonumber \\
&\leq C e^{-C' \lambda^{\Theta(h)} } =o(1). \label{eq:leq0}
\end{align}

We now focus on the sum over  $\Theta(h)/2 \leq |q| \leq P/2$. To finish showing the decay of~\eqref{eq:ergodicequation_l=0}, by the triangle inequality, we need to show 
$\max_{\frac{\Theta(h)}{2 } < |q| \leq  \frac{P}{2}, t \in I_E} \left|\lrang{U_{A^{-t}j/N} G_h, M_A^q G_h}_{L^2}\right| =o(P^{-1})$.
We see for $\frac{\Theta(h)}{2 } < |q| \leq  \frac{P}{2},  t \in I_E$,
\begin{equation}\label{eq:q_large_estimate}
\left|\lrang{U_{A^{-t}j/N} G_h, M_A^q G_h}_{L^2}\right| \leq \frac{\sqrt{2}}{\sqrt{\lambda^q + \lambda^{-q}}} \leq \sqrt{2} \lambda^{-\frac{|q|}{2}} \leq \sqrt{2} \lambda^{-\frac{\Theta(h)}{4}} =o(P^{-1}). 
\end{equation}

 Using~\eqref{eq:leq0} and~\eqref{eq:q_large_estimate}, we conclude
\eqref{eq:ergodicequation_l=0} converges to 0.

3. Finally, we show~\eqref{eq:ergodicequation_sum} converges to 0.

For $t \in I_E$, note that $A^{-t}j/N \leq C\lambda^{|t|}/N \leq C \lambda^{\frac{P}{2} -\Theta(h)}/N \leq C \lambda^{-\Theta(h)} \rightarrow 0$. Therefore, we can apply~\eqref{eq:result_3} from Lemma~\ref{lem:sum_over_integers} to know for all $|s| \leq P/2$ and $t \in I_E$,
\begin{align*}
\sum_{l \in \Z^2 \setminus 0} \left| \lrang{U_{l+A^{-t}j/N} G_h, M_A^{q} G_h}_{L^2} \right|  &\leq C \left(h \lambda^{\frac{|q|}{2}} + (h \log h^{-1})^{\frac{1}{4}}\right)\\
&\leq C\left(\sqrt{h} + (h \log h^{-1})^{\frac{1}{4}}\right) = o(P^{-1}).
\end{align*}
Thus, by the triangle inequality,~\eqref{eq:ergodicequation_sum} decays to 0, which finishes the proof. 
\end{proof}

\appendix 
\section{Examining Orbit Closures}\label{appendix1}

In Theorem~\ref{thm:full_support}~\eqref{item:K}, we showed if $A \in \Sp(2n, \Z)$ has a non-unit length eigenvalue and $A|_{E_\pm}$ is diagonalizable, then for some $v \in E_+ \cup E_- \setminus \{0\}$ and $z \in \bT^{2n}$, $$K \coloneqq \overline{\{A^k (z+\bT_v) : k \in \Z\}} \subset \supp \mu.$$

In this appendix, we characterize $K$ using the characteristic polynomial of $A$. 
Recall that $V_v \subset \Q^{2n}$ is the smallest rational subspace such that $\R v \subset V_v \otimes \R$ and that $\bT_v \subset \bT^{2n}$ is the subtorus given by the projection of $V_v \otimes \R$ onto $\bT^{2n}$. 
We first characterize $\bT_v$. 
\begin{lemma}\label{lem:T_v}
$\bT_v = \overline{\R v \mod \Z^{2n}}$. 
\end{lemma}

\begin{proof}
The argument is exactly the same as ~\cite{dyatlov2021semiclassical}*{Lemma 4.3} and needs no adaptation to higher dimensions. 
\end{proof}


By Lemma~\ref{lem:T_v}, we have an alternative characterization of $K$: $$K =\overline{\{A^k(z+ \R v) \bmod \Z^{2n} : k \in \Z\}}.$$

Define $$k_0 \coloneqq \min \{k \in \N : \text{ the characteristic polynomial of } A^k \text{ is reducible over } \Q  \}.$$
If the characteristic polynomial of $A^k$ is irreducible over $\Q$ for all $k \in \N$, set $k_0 = \infty$.
If $k_0 < \infty$, it is the smallest value of $k$ such that $A^k$ has a nontrivial proper rational invariant subspace. 

\begin{proposition}\label{prop:K}
Suppose $A|_{E_\pm}$ is diagonalizable over $\C$. If $k_0< \infty$, then $K$ contains the union of $k_0$ non-parallel tori of dimension at least 1. If $k_0 = \infty$, then $K =\bT^{2n}$.
\end{proposition}

\begin{proof}
We assume that $v \in E_+ \setminus \{0\}$; a similar argument holds if instead $v \in E_- \setminus \{0\}$.

1. Set $l \coloneqq \dim E_+$. Let $\Lambda$ be the largest absolute value of an eigenvalue of $A$. 

As $A|_{E_+}$ is diagonalizable over $\C$, we choose an eigenbasis for $E_+ \otimes \C$ and select an inner product that makes this eigenbasis orthonormal. Under the real part of this inner product, $B \coloneqq \Lambda^{-1}A\mid_{E_+} \in \ortho(l)$, where $\ortho(l)$ is the compact Lie group of orthogonal transformations on $E_+$.
Now set $$H \coloneqq \overline{\{B^k : k \in \Z\}},$$
where the closure is taken in $\ortho(l)$.
Clearly, $H$ is an abelian subgroup. As $H$ is a closed subgroup of a compact Lie group, it is a compact Lie group. Thus, $H$ has a finite number of connected components. 

Now set $H_0$ to be the connected component of the identity. As multiplication by $g \in H$ is a homomorphism, the cosets of $H_0$ are the connected components of $H$. 
Via the exponential map, all connected compact abelian Lie groups are isomorphic to tori. Therefore, for some $m \leq l$, $H_0 \simeq \bT^m = \underbrace{S^1 \times \cdots \times S^1}_m$.

Let
$$R \coloneqq \begin{bmatrix} 0 & 1 \\ -1 & 0 \end{bmatrix},$$ then define the block diagonal matrices:
$$X_1 \coloneqq \begin{bmatrix} 
R & & & \\
& 0 & & \\
& & \ddots & \\
& &  &  0\\
\end{bmatrix}, \quad X_2 \coloneqq \begin{bmatrix} 
0 & & & \\
& R & & \\
& & \ddots & \\
& &  &  0\\
\end{bmatrix}, \ldots, \quad X_m \coloneqq \begin{bmatrix} 
0 & & & \\
& 0 & & \\
& & \ddots & \\
& &  &  R\\
\end{bmatrix}.$$
Set $Z_i \coloneqq X_i \oplus [0]_{l-m}$.
As $[Z_i, Z_j] =0$ and $\exp(sR) = \begin{bmatrix} \cos s & \sin s \\ -\sin s & \cos s \end{bmatrix}$, we see that 
$H_0 \simeq \{e^{s_1 Z_1 + \cdots + s_m Z_m} : s_1, \ldots, s_m \in \R\}$, where the isomorphism is given by a linear map.
Thus, we see that $H_0 v$ is isomorphic to a torus of dimension at most $m$.  

3. Let $W$ be a rational hyperplane in $E_+$. Specifically, we mean that $W$ is given by a basis of  $l-1$ elements of $\Q^{2n} \cap E_+$. We examine the intersection $W \cap H_0 v$. As $H_0 v$ is isomorphic via a linear transformation to a  torus, either $H_0 v \subset W$ or $W \cap H_0 v$ has measure 0 in $H_0 v$.

4. First suppose $H_0 v \subset W$. By the definition of $H$, there exists some $L \geq 1$ for which $B^L \in H_0$. We assume that $L$ is the smallest such natural number for which this holds. Then, $A^{kL} v \in W$ for all $k \in \Z$. Let $W'$ be the minimal rational subspace such that $A^{kL}v \subset W'$ for all $k \in \Z$. 
Clearly, $W'$ is $A^L$-invariant. As $W'$ is a nontrivial proper rational subspace of $\R^{2n}$, $A^L$ must have reducible characteristic polynomial over $\Q$. Thus, $k_0 \leq L$. 

If $V_v =\Q^{2n}$, then $K= \bT^{2n}$ and trivially  $K$ contains the union of $k_0$ non-parallel tori of dimension at least 1.

Now assume that $V_v  \neq \Q^{2n}$. In this case, $1 \leq \dim V_v \otimes \R \leq 2n-1$.  By the definition of $k_0$, for $0 < k < k_0$, $A^k$ has no proper, nontrivial rational invariant subspaces. Therefore, $V_v \otimes \R$ cannot be $A^k$-invariant. In other words, $A^k (V_v \otimes \R) \neq A^j(V_v \otimes \R)$ for $0 \leq k < j < k_0$. Then,  for $0 \leq k < k_0$, each torus $A^k \bT_v$ has a different tangent space. We conclude
$$\bigcup_{k=0}^{k_0 -1} \left\{A^k(z +\bT_v) \right\} \subset K,$$
i.e., $K$ contains the union of $k_0$ non-parallel tori of dimension at least 1. 

5. Now suppose that $W \cap H_0 v$  has measure $0$ in $H_0v$. From the above argument, we see $k_0 = \infty$.
As the set of all rational hyperplanes in $E_+$ is countable, there exists $w \in H_0 v$ that is not contained in any rational hyperplane.
For some subsequence $k_l$, we have $B^{k_l} v \rightarrow w$. Passing to a further subsequence, there exists $z_\infty \in \bT^{2n}$ such that $A^{k_l} z \bmod \Z^{2n} \rightarrow z_\infty$. Thus,
$$\bT^{2n} = \overline{\{z_\infty + \R w \bmod \Z^{2n}\}} \subset \overline{\{A^{k_l}(z + \R v) \bmod \Z^{2n} : l \in \Z\}} \subset K,$$
which finishes the proof.
\end{proof}
We can then conclude Theorem~\ref{thm:full_support}~\eqref{item:full_support}.

From Proposition~\ref{prop:K}, for $A \in \Sp(2n, \Z)$ with $k_0=2$, the support of any semiclassical measure corresponding to $A$ must contain the union of two non-parallel tori. 
The following example proves this characterization of $\supp \mu$ is tight for $n=2$. Specifically, we exploit Theorem~\ref{thm:counterexample} to show there exists hyperbolic $A \in \Sp(4, \Z)$ with $k_0=2$ and a semiclassical measure supported exactly on the union of two tori. 
As in \S\ref{section:counterexample}, we use the ordering of coordinates $(x_1, \xi_1, x_2, \xi_2)$.
In order to distinguish between the quantizations of different matrices, we use the notation $M_A$ to denote the quantization of $A$.

\begin{proposition}\label{prop:tight_example}
Set $B \in \Sp(2, \Z)$ to be hyperbolic and symmetric, with all positive entries. 
For 
$$A= \begin{bmatrix}
0 & B \\ -B & 0
\end{bmatrix},$$ 
$A$ is a hyperbolic element of $\Sp(4, \Z)$ with $k_0=2$ and has an associated semiclassical measure supported exactly on the union of two tori.  
\end{proposition}

\begin{proof}
First off, a calculation shows that $A$ is in $\Sp(4, \Z)$. 

Let $\gamma^{\pm 1}$ be the eigenvalues of $B$, where $\gamma >1$. Therefore the eigenvalues of $A^2 = -(B^2 \oplus B^2)$ are $-\gamma^{\pm 2}$. Thus, $i \gamma$, $-i \gamma$, $i \gamma^{-1}$, and $-i \gamma^{-1}$ are the eigenvalues of $A$, which implies that $A$ is hyperbolic.

The characteristic polynomial of $A$ is  $(\lambda - i \gamma) (\lambda + i \gamma)(\lambda -i \gamma^{-1})(\lambda +i \gamma^{-1})$. Clearly, there is no linear factor in $\Q[x]$. Therefore, if the characteristic polynomial is reducible over $\Q$, it factors into $(\lambda^2 + \gamma^2)(\lambda^2 + \gamma^{-2})$. 
Via the Gauss lemma,  if a polynomial in $\Z[x]$ factors in $\Q[x]$, it also factors in $\Z[x]$. As $\gamma >1$, $\gamma^{-2} \notin \Z$. We conclude that the characteristic polynomial of $A$ is always irreducible over $\Q$.

However, as the characteristic polynomial of $A^2=-(B^2  \oplus B^2)$ is the square of the characteristic polynomial of $-B^2$, it is clearly reducible over $\Q$. Therefore, $k_0 = 2$.

For $N_k = \frac{\lambda^{k} -\lambda^{-k}}{\lambda - \lambda^{-1}}$, recall from \S\ref{subsection:short_periods} that $P({N}_k) =2k$ is the period of $M_B$, i.e., $M^{P(N_k)}_B |_{\cH_{N_k}(0)}  = e^{i \phi(N_k)} I$.
As in \S\ref{subsection:assumptions}, we restrict to $k=0 \mod 2$ when $\Tr(B)$ is odd and to 
$k=0 \bmod 6$ when $\Tr(B)$ is even. 

Set $P= P(N_k), \phi =\phi(N_k), M_B = M_B|_{\cH_{N_k}(0)}$, and $G_N = G_{N_k}$. From Proposition~\ref{prop:counterexample}, we know that after normalization, 
$$u^{(k)}(x_1, x_2) \coloneqq \frac{1}{\sqrt{P}}\sum_{t=-\frac{P}{2}}^{\frac{P}{2}-1} \left(e^{-i \frac{ \phi}{P}t} M_{B}^{t} G_{N}\right) \otimes \left(e^{-i \frac{ \phi}{P}\left(t+\frac{P}{2}\right)} M_{B}^{t+\frac{P}{2}} G_{N} \right)$$
is a  sequence of eigenfunctions for $M_{B \oplus B}$ that weakly converge to the semiclassical measure  $\mu = \frac{1}{2}[\delta(x_1, \xi_1)   \otimes dx_2 d\xi_2 + dx_1 d\xi_1 \otimes \delta(x_2, \xi_2)]$. 
The support of $\mu$ is exactly the union of two tori. 
Thus, to finish the proof, it suffices to show  that $u^{(k)}$ are also eigenfunctions for $M_A$. 

Note that  $A = (B \oplus B) R$, where $R=\begin{bmatrix} 0 & I \\ -I & 0 \end{bmatrix}$.  Therefore, $M_A = (M_B \otimes M_B) M_R$, with $M_R u(x_1, x_2) = u(-x_2, x_1)$. A calculation verifies that $M_B^{t} G_N(-x) = M_B^{t} G_N(x)$.
Therefore,
\begin{align*}
M_A\left(M_B^t G_N \otimes M_B^{\frac{P}{2} +t} G_N\right)(x_1, x_2) &= (M_B \otimes M_B)\left(M_B^{\frac{P}{2} +t} G_N \otimes M_B^t G_N\right)(x_1, -x_2)\\
&= \left(M_B^{\frac{P}{2} +t +1} G_N \otimes M_B^{t +1} G_N\right)(x_1, -x_2)\\
&= \left(M_B^{\frac{P}{2} +t +1} G_N \otimes M_B^{t +1} G_N\right)(x_1, x_2).
\end{align*}
Therefore, using the $P$-periodicity of $M_B$,
\begin{align*}
M_A u^{(k)} &=\frac{1}{\sqrt{P}}\sum_{t=-\frac{P}{2}}^{\frac{P}{2}-1} \left(e^{-i \frac{ \phi}{P}\left(t+\frac{P}{2}\right)} M_{B}^{t+\frac{P}{2}+1} G_{N} \right) \otimes \left(e^{-i \frac{ \phi}{P}t} M_{B}^{t+1} G_{N}\right)\\
&=e^{2i \frac{\phi}{P}}\frac{1}{\sqrt{P}}\sum_{t=-\frac{P}{2}}^{\frac{P}{2}-1} \left(e^{-i \frac{ \phi}{P}\left(t+\frac{P}{2} +1 \right)} M_{B}^{t+\frac{P}{2}+1} G_{N} \right) \otimes \left(e^{-i \frac{ \phi}{P}(t+1)} M_{B}^{t+1} G_{N}\right)\\
&=e^{2i \frac{\phi}{P}}\frac{1}{\sqrt{P}}\sum_{t=-\frac{P}{2}}^{\frac{P}{2}-1} \left(e^{-i \frac{ \phi}{P}t} M_{B}^{t} G_{N}\right) \otimes \left(e^{-i \frac{ \phi}{P}\left(t+\frac{P}{2} \right)} M_{B}^{t+\frac{P}{2}} G_{N} \right) =  e^{2i \frac{\phi}{P}}u^{(k)}.
\end{align*}

We conclude that $A$ has a semiclassical measure that is supported on the union of two transversal tori.
\end{proof}

In Lemma~\ref{lem:first-power-suffices}, it is proven that if the characteristic polynomial of $A$ is irreducible over the rationals with Galois group $S_2 \wr S_n$, then the characteristic polynomial of $A^k$ is also irreducible over the rationals with Galois group $S_2 \wr S_n$. For the definition of the Galois group of a polynomial, see the start of Appendix~\ref{appendix2}.

Clearly, the characteristic polynomial of the above example in Proposition~\ref{prop:tight_example} cannot have Galois group $S_2 \wr S_n$.
In fact, its Galois group is the Klein four group $V_4 \simeq Z_2 \times Z_2$.  As noted earlier, the eigenvalues of $A$ are $i \gamma$, $-i \gamma$, $i \gamma^{-1}$, and $-i\gamma^{-1}$. Each of the maps $i \gamma \mapsto -i \gamma$,  $i \gamma \mapsto i \gamma^{-1}$,  and $i \gamma \mapsto -i\gamma^{-1}$ is an involution. Moreover, each of these maps swaps the other two roots; respectively, $i\gamma^{-1}$ swaps with $-i\gamma^{-1}$, $-i\gamma$ swaps with $-i\gamma^{-1}$, and $-i\gamma$ swaps with $i \gamma^{-1}$.  Thus, every nontrivial element of the Galois group is of order two and a product of two disjoint transpositions. This is exactly the definition of the Klein four group as a subgroup of the symmetric group $S_4$.

\section{Galois groups of random symplectic matrices \\ by Theresa C. Anderson and Robert J. Lemke Oliver} \label{appendix2}
The purpose of this appendix is to provide a proof of the following:

\begin{theorem}\label{thm:galois-group-powers}
    For any fixed natural number $n \geq 1$, there exists a subset $\mathcal{S}$ of $\mathrm{Sp}(2n,\mathbb{Z})$ such that
        \[
            \lim_{H \to \infty} \frac{ \#\{ A \in \mathcal{S} : \| A\| \leq H\}}{\# \{A \in \mathrm{Sp}(2n,\mathbb{Z}) : \|A\| \leq H\}}
                = 1,
        \]
    every $A \in \mathcal{S}$ has a non-unit length eigenvalue, and so that for every $A \in \mathcal{S}$ and every integer $m \geq 1$, the characteristic polynomial of $A^m$ is irreducible with Galois group $S_2 \wr S_n$.  
    (Here, $\|\cdot \|$ denotes any norm on the space of $2n \times 2n$ matrices; for example, we may take $\| A \| = \max\{|a_{i,j}| : 1 \leq i,j \leq 2n\}$, the maximum absolute values of the entries of $A$.)
\end{theorem}

We briefly recall some basic necessary notions, particularly related to Galois groups and wreath products.  First, for any $n$, we let $S_n$ denote the symmetric group on $n$ symbols, that is, the full group of permutations of $\{1,\dots,n\}$  -- or, as is more instructive for our purposes, on $\{\alpha_1,\dots,\alpha_n\}$.  The splitting field of a polynomial $f$ of degree $n$ is the extension $\mathbb{Q}(\alpha_1,\dots,\alpha_n)$ obtained by adjoining to $\mathbb{Q}$ the $n$ roots $\alpha_1,\dots,\alpha_n$ of $f$ over $\mathbb{C}$.  The Galois group of the splitting field $\mathbb{Q}(\alpha_1,\dots,\alpha_n)$ is the set of its field automorphisms, each of which must fix $\mathbb{Q}$ and induce a permutation on the set of roots $\{\alpha_1,\dots,\alpha_n\}$.  In fact, any automorphism of $\mathbb{Q}(\alpha_1,\dots,\alpha_n)$ is determined by this permutation, so we may regard the Galois group as a subgroup of the permutation group on $\{\alpha_1,\dots,\alpha_n\}$.  We refer to this permutation group as the Galois group of $f$, which we denote by $\mathrm{Gal}(f)$.  It is worth noting that $100\%$ of monic, irreducible, degree $n$ polynomials $f\in \mathbb{Z}[x]$ have $\mathrm{Gal}(f)$ isomorphic to $S_n$ (for a more detailed introduction to the information mentioned above, see~\cite{DummitFoote}*{\S 14.6}).   

To define the wreath product $S_2 \wr S_n$, which is a subgroup of $S_{2n}$, we first view $S_{2n}$ as the permutation group of the $2n$ symbols $\alpha_1,\beta_1, \dots, \alpha_n, \beta_n$.  The wreath product is the subgroup that preserves the set of unordered pairs $\{\alpha_1,\beta_1\}, \dots, \{\alpha_n,\beta_n\}$.  In particular, in $S_2 \wr S_n$, any pair $\{\alpha_i, \beta_i\}$ may be sent to any other $\{\alpha_j,\beta_j\}$ (either by $\alpha_i \mapsto \alpha_j$ and $\beta_i \mapsto \beta_j$ or by $\alpha_i \mapsto \beta_j$ and $\beta_i \mapsto \alpha_j$), and the elements of any pair $\{\alpha_i, \beta_i\}$ may be swapped (that is, $\alpha_i \leftrightarrow \beta_i$), but it is not possible to send $\alpha_i,\beta_i$ to elements with unequal indices.  For a more formal introduction to the wreath product, see~\cite{DixonMortimer}.

The Galois group of the characteristic polynomial of any $A \in \mathrm{Sp}(2n,\mathbb{Z})$ is naturally a subgroup of $S_2 \wr S_n$, as follows from the reciprocal structure of its roots, so Theorem~\ref{thm:galois-group-powers} shows that a ``random'' element of $\mathrm{Sp}(2n,\mathbb{Z})$ has Galois group as large as possible.  To view this inclusion more concretely, if $\lambda$ is an eigenvalue of $A$, then $\frac{1}{\lambda}$ is as well, so the roots of the characteristic polynomial naturally come in pairs $\{\lambda,\frac{1}{\lambda}\}$.  Moreover, if under the action of some element of the Galois group $\lambda$ is sent to some $\lambda^\prime$, say, then $\frac{1}{\lambda}$ must be sent to $\frac{1}{\lambda^\prime}$.  In particular, the Galois group must preserve this pair structure, so must be a subgroup of $S_2 \wr S_n$.  Building on this, for any eigenvalue $\lambda$, we note that the minimal polynomial of $\lambda + \frac{1}{\lambda}$ will have degree at most $n$, corresponding to the $n$ choices of the pair $\{\lambda,\frac{1}{\lambda}\}$.  Writing $f_\lambda$ and $f_{\lambda+\frac{1}{\lambda}}$ for the associated minimal polynomials, we then find $\mathrm{Gal}(f_{\lambda+\frac{1}{\lambda}}) \leq S_n$ and $\mathrm{Gal}(f_\lambda) \leq S_2 \wr \mathrm{Gal}(f_{\lambda+\frac{1}{\lambda}})$, where the $S_2$ term accounts for the Galois group of the at most quadratic extension $\mathbb{Q}(\lambda)/\mathbb{Q}(\lambda + \frac{1}{\lambda})$. 

We now continue with the proof by first noting that it suffices to prove Theorem~\ref{thm:galois-group-powers} only for the case $m=1$, as the following lemma makes clear.

\begin{lemma} \label{lem:first-power-suffices}
    Let $A \in \mathrm{Sp}(2n,\mathbb{Z})$ be such that the characteristic polynomial of $A$ is irreducible with Galois group $S_2 \wr S_n$.  If $n \geq 2$, then for every integer $m \geq 1$, the characteristic polynomial of $A^m$ is irreducible with Galois group $S_2 \wr S_n$ and $A$ has a non-unit length eigenvalue.  If $n = 1$, this same conclusion holds as well, provided that the characteristic polynomial of $A$ is not $x^2+1$, $x^2+x+1$, or $x^2-x+1$.
\end{lemma}
\begin{proof}
    Let $f_A$ denote the characteristic polynomial of $A$, and let $\lambda_1,\dots,\lambda_{2n}$ be its roots inside $\mathbb{C}$.  Similarly, if we let $f_{A^m}$ denote the characteristic polynomial of $A^m$ for some $m \geq 1$, then roots of $f_{A^m}$ are $\lambda_1^m, \dots, \lambda_{2n}^m$.  Now, let $\lambda$ denote an arbitrary root of $f_A$, and consider the field $\mathbb{Q}(\lambda)$.  By the assumption that $f_A$ is irreducible, we see that $\mathbb{Q}(\lambda)$ has degree $2n$ over $\mathbb{Q}$.  Moreover, we have $\lambda^m \in \mathbb{Q}(\lambda)$, and the claim will follow if $\mathbb{Q}(\lambda^m) = \mathbb{Q}(\lambda)$ for each $m$, since this implies that $f_{A^m}$ is irreducible and that $\mathrm{Gal}(f_{A^m}) = \mathrm{Gal}(f_A)$.

    By the assumption that $\mathrm{Gal}(f_A) \simeq S_2 \wr S_n$, we see that $\mathbb{Q}(\lambda)$ admits only two proper subextensions, namely $\mathbb{Q}$ and $\mathbb{Q}(\lambda + \frac{1}{\lambda}) =: F$, the latter of which is a degree $n$ extension of $\mathbb{Q}$ whose normal closure over $\mathbb{Q}$ has Galois group $S_n$.  It therefore suffices to show that $\lambda^m \not\in F$.  Note that this rules out both $\mathbb{Q}$ and $F$ as candidates for $\mathbb{Q}(\lambda^m)$. 
    
    Suppose to the contrary that $\lambda^m = \alpha$ for some $\alpha \in F$.  The extension $\mathbb{Q}(\lambda)/F$ is Galois (it is degree $2$), with the nontrivial automorphism sending $\lambda$ to $\frac{1}{\lambda}$.  Applying this automorphism (which fixes $\alpha$) and simplifying, we also find that $\lambda^m = 1/\alpha$.  This implies that $\alpha = \pm 1$ and that $\lambda$ is a root of unity.  This cannot be the case if $n \geq 2$, since the extension $\mathbb{Q}(\lambda)$ is not abelian (the group $S_2 \wr S_n$ is not abelian for any $n \geq 2$), nor can it be the case if $n=1$ unless $f_A$ is one of the listed polynomials (the listed polynomials are the only irreducible quadratics whose roots are roots of unity). 
    
    Finally, we note that if an algebraic integer is not a root of unity, it must have a Galois conjugate off the unit circle. As $A$ cannot have eigenvalues that are roots of unity, it must have a non-unit length eigenvalue. 
    This completes the proof.
\end{proof}

To prove Theorem~\ref{thm:galois-group-powers} in the case $m=1$, we will employ a strategy of Davis, Duke, and Sun~\cite{DDS}, incorporating results of Chavdarov~\cite{Chavdarov} and Gorodnik and Nevo~\cite{GorodnikNevo}.  Namely, this involves extending statistics for reciprocal polynomials to the case of symplectic matrices via equidistribution results in ergodic theory.

We begin with the following lemma from~\cite{DDS}.

\begin{lemma}{\cite{DDS}*{Lemma 2}} \label{lem:sufficient-cycles}
    Let $n \geq 2$ and let $G$ be a subgroup of $S_2 \wr S_n$ containing a $2$-cycle, $4$-cycle, $(2n-2)$-cycle, and a $2n$-cycle.  Then $G = S_2 \wr S_n$.
\end{lemma}

Lemma~\ref{lem:sufficient-cycles} suggests that sieving is a natural approach to Theorem~\ref{thm:galois-group-powers}.  Specifically, the characteristic polynomial $f_A(x)$ of any element $A \in \mathrm{Sp}(2n,\mathbb{Z})$ must be reciprocal (i.e., palindromic), that is, it satisfies $f_A(x) = x^{2n} f_A(\frac{1}{x})$; let $\mathcal{P}_{2n}^\mathrm{rec}(\mathbb{Z})$ be the set of monic reciprocal polynomials in $\mathbb{Z}[x]$ of degree $2n$, so that $f_A \in \mathcal{P}_{2n}^\mathrm{rec}(\mathbb{Z})$.  Moreover, the factorization type of $f_A \pmod{\ell}$ for primes $\ell$ not dividing the discriminant of $f_A$ corresponds to the cycle type of an element in $\mathrm{Gal}(f_A)$ (namely, the Frobenius element at $\ell$).  Thus, for any prime $\ell$, let $\mathcal{P}_{2n}^\mathrm{rec}(\mathbb{F}_\ell)$ be the set of monic reciprocal polynomials of degree $2n$ over $\mathbb{F}_\ell$, and, motivated by Lemma~\ref{lem:sufficient-cycles}, for any $1 \leq k \leq n$ we let $\mathcal{P}_{2n}^\mathrm{rec}(F_\ell ; (2k))$ be the subset of $\mathcal{P}_{2n}^\mathrm{rec}(\mathbb{F}_\ell)$ consisting of polynomials factoring as an irreducible polynomial of degree $2k$ times a product of $2n-2k$ distinct linear polynomials.  Observe that $\# \mathcal{P}_{2n}^\mathrm{rec}(\mathbb{F}_\ell) = \ell^n$.

\begin{lemma} \label{lem:reciprocal-cycles}
    For any integer $1 \leq k \leq n$, we have 
        \[
            \#\mathcal{P}_{2n}^\mathrm{rec}(\mathbb{F}_\ell ; (2k) )
            = \frac{ 1 } {2^{n-k+1}  \cdot k \cdot (n-k)!} \ell^{n} + O(\ell^{n-1}).
        \]
\end{lemma}
\begin{proof}
    This follows on combining~\cite{DDS}*{Lemma 1} and~\cite{DDS}*{Lemma 3}.
\end{proof}

Lemmas~\ref{lem:sufficient-cycles} and \ref{lem:reciprocal-cycles} are the main ingredients leading to the main principle of~\cite{DDS}, which concerns the Galois groups of random integral reciprocal polynomials.  For our purposes, however, we need to relate this back to elements of $\mathrm{Sp}(2n,\mathbb{Z})$.  We begin with the following result of Chavdarov~\cite{Chavdarov}.

\begin{lemma} \label{lem:matrices-with-given-poly}
    For any prime $\ell \geq 5$ and any $f \in \mathcal{P}_{2n}^\mathrm{rec}(\mathbb{F}_\ell)$,  there holds
        \[
            \# \{ A \in \mathrm{Sp}(2n,\mathbb{F}_\ell) : f_A = f \}
                = \ell^{2n^2} + O(\ell^{2n^2-1}),
        \]
    where, for any $A \in \mathrm{Sp}(2n,\mathbb{F}_\ell)$, $f_A$ denotes its characteristic polynomial.
\end{lemma}
\begin{proof}
    This follows from~\cite{Chavdarov}*{Theorem 3.5}, by taking $\gamma = 1$.  Note that the exclusion of $\ell < 4$ is unimportant in the sieve application that we will employ.
\end{proof}

In particular, Lemma~\ref{lem:matrices-with-given-poly} shows that the characteristic polynomials of elements of $\mathrm{Sp}(2n, \mathbb{F}_\ell)$ are approximately equidistributed in $\mathcal{P}^\mathrm{rec}_{2n}(\mathbb{F}_\ell)$.  (Recall that $|\mathrm{Sp}(2n,\mathbb{F}_\ell)| = \ell^{2n^2+n} +  \linebreak   O(\ell^{2n^2+n-1})$ and $\#\mathcal{P}_{2n}^\mathrm{rec}(\mathbb{F}_\ell) = \ell^n$.)  However, to obtain Theorem~\ref{thm:galois-group-powers}, it is necessary to know that elements of $\mathrm{Sp}(2n,\mathbb{Z})$ are essentially equidistributed when reduced $\pmod{\ell}$, and that these reductions are essentially independent for different primes $\ell$.  In other words, we are asking for an effective form of approximation for the group $\mathrm{Sp}(2n, \mathbb{Z})$.  For this, we have the following consequence of work of Gorodnik and Nevo~\cite{GorodnikNevo}.  

\begin{lemma}\label{lem:effective-approximation}
    Let $n \geq 1$.  There is a constant $\delta >0$ such that for any squarefree integer $q$ and any $2n \times 2n$ integer matrix $A_0$ that lies in $\mathrm{Sp}(2n,\mathbb{F}_\ell)$ when reduced modulo $\ell$ for each prime $\ell \mid q$, there holds for any $H \geq 1$
        \begin{align*}
            & \#\{ A \in \mathrm{Sp}(2n,\mathbb{Z}) : \| A\| \leq H, A \equiv A_0 \pmod{q} \} \\
                 & \quad = \prod_{\ell \mid q} |\mathrm{Sp}(2n,\mathbb{F}_\ell)|^{-1} \cdot \#\{ A \in \mathrm{Sp}(2n,\mathbb{Z}) : \|A\| \leq H\} \cdot ( 1 + O( q^{2n^2+n} H^{-\delta})),
        \end{align*}
    where $\| A\|$ is any norm on the set of $2n \times 2n$ matrices (for example, $\| A\| := \max\{ |a_{i,j}| \}$, the maximum absolute value of the entries of $A$).
\end{lemma}
\begin{proof}
    As indicated, this follows from the work of Gorodnik and Nevo~\cite{GorodnikNevo}.  Specifically, in their terminology, the family of sets
        \[
            \{ A \in \mathrm{Sp}(2n,\mathbb{R}) : \| A \| \leq H\}
        \]
    is admissible~\cite{GorodnikNevo}*{top of p. 79}, hence H\"older well-rounded.  The claim then follows from~\cite{GorodnikNevo}*{Corollary 5.2} on taking $y$ to be the identity and $x$ to be an element of $\mathrm{Sp}(2n,\mathbb{Z})$ congruent to $A_0 \pmod{q}$ (which exists, by strong approximation for the group $\mathrm{Sp}(2n,\mathbb{Z})$; see~\cite{PR}*{Theorem 7.12}, or~\cite{NewmanSmart}*{Theorem 1} for an elementary proof of existence).  See also~\cite{GorodnikNevo}*{Equation (6.9)}.
\end{proof}

With this, we're now ready to prove Theorem~\ref{thm:galois-group-powers}.

\begin{proof}[Proof of Theorem~\ref{thm:galois-group-powers}]
    Let $A \in \mathrm{Sp}(2n,\mathbb{Z})$ and let $f_A$ denote its characteristic polynomial.  Since $f_A$ is reciprocal, its Galois group is a subgroup of $S_2 \wr S_n$.  It then follows from Lemma~\ref{lem:sufficient-cycles} that if $\mathrm{Gal}(f_A) \ne S_2 \wr S_n$, then $f_A$ must lie in one of the four ``exceptional'' sets $\mathcal{E}_2$, $\mathcal{E}_4$, $\mathcal{E}_{2n-2}$, and $\mathcal{E}_{2n}$, where for any $1 \leq k \leq n$, we define
        \[
            \mathcal{E}_{2k} := \{ f \in \mathcal{P}_{2n}^\mathrm{rec}(\mathbb{Z}) : f \pmod{\ell} \not\in \mathcal{P}_{2n}^\mathrm{rec}(\mathbb{F}_\ell;(2k)) \text{ for all primes $\ell$}\}.
        \]
    (For example, $\mathcal{E}_{2n}$ is the set of monic reciprocal polynomials that are reducible $\pmod{\ell}$ for every prime $\ell$.)  To obtain the conclusion of the theorem for $m=1$ (which suffices by Lemma~\ref{lem:first-power-suffices}, at least provided $n \geq 2$) it is therefore enough to show that
        \begin{equation} \label{eqn:goal}
            \#\{ A \in \mathrm{Sp}(2n,\mathbb{Z}) : \| A \| \leq H, f_A \in \mathcal{E}_{2k}\}
                = o( \#\{ A \in \mathrm{Sp}(2n,\mathbb{Z}) : \| A \| \leq H\} ).
        \end{equation}
    For any squarefree $q$, let $\mathcal{E}_{2k}(q)$ be defined by
        \[
             \mathcal{E}_{2k}(q) := \{ f \in \mathcal{P}_{2n}^\mathrm{rec}(\mathbb{Z}) : f \pmod{\ell} \not\in \mathcal{P}_{2n}^\mathrm{rec}(\mathbb{F}_\ell;(2k)) \text{ for all primes $\ell \mid q$}\},
        \]
    and observe that $\mathcal{E}_{2k}$ is a subset of $\mathcal{E}_{2k}(q)$.  If $q$ is coprime to $6$, then by combining Lemmas~\ref{lem:reciprocal-cycles}-\ref{lem:effective-approximation}, we find
        \begin{align*}
            &\#\{ A \in \mathrm{Sp}(2n,\mathbb{Z}) : \| A \| \leq H, f_A \in \mathcal{E}_{2k}(q)\}\\
                &\quad= \prod_{\ell \mid q} \left( 1 - \frac{1}{2^{n-k+1} \cdot k \cdot (n-k)!} + O(\ell^{-1})\right) \cdot \#\{ A \in \mathrm{Sp}(2n,\mathbb{Z}) : \| A \| \leq H\} \cdot (1 + O(q^{2n^2+n} H^{-\delta})).
        \end{align*}
    Each term in the product above may be uniformly bounded away from $1$ for $\ell$ sufficiently large, so by taking $q$ to be a product of $N$ primes, we obtain~\eqref{eqn:goal} as $N \to \infty$.  This completes the proof if $n \geq 2$.  If $n=1$, then we have shown that $\mathrm{Gal}(f_A) = S_2$ for $100\%$ of elements $A \in \mathrm{Sp}(2n,\mathbb{Z})$, and it remains only to show that for $100\%$ of $A$, $f_A$ is not one of the three exceptional characteristic polynomials in Lemma~\ref{lem:first-power-suffices}.  But this follows exactly as above, for example, by bounding the number of matrices whose characteristic polynomial is one of the exceptional polynomials $\pmod{\ell}$ for every prime $\ell$.  This completes the proof.
\end{proof}

\begin{remark}
    While we have presented the proof of Theorem~\ref{thm:galois-group-powers} in a soft form (i.e., with a little-oh), it is essentially no harder to use Lemmas~\ref{lem:reciprocal-cycles}--\ref{lem:effective-approximation} in concert with a more sophisticated sieve (e.g. the large sieve or the Selberg sieve; see, e.g.,~\cite{FI}) to obtain a power saving bound on the number of non-generic matrices.
\end{remark}

\begin{remark}
    We can arrive at the weaker result that $100\%$ of $A \in \mathrm{Sp}(2n,\mathbb{Z})$ are such that the characteristic polynomial of $A$ is irreducible by replacing Lemma~\ref{lem:reciprocal-cycles} with~\cite{Chavdarov}*{Lemma 3.2} and following the same argument, but the stronger result takes only slightly more effort.  It also more easily affords the conclusion about powers $A^m$ and illuminates the structure of the polynomials considered.  Additionally, Rivin~\cite{Rivin} proved the irreducibility result for characteristic polynomials (but not powers) for an ordering different from matrix height.
\end{remark}

\medskip\noindent\textbf{Acknowledgements.} EK is supported by MathWorks and by NSF GRFP under grant No. 1745302. She would like to thank Semyon Dyatlov, who is  partially supported by NSF CAREER grants DMS-1749858 and DMS-2400090, for suggesting and advising this project. She  would also like to thank Bjorn Poonen and Vijay Srinivasan for their help with Proposition \ref{prop:tight_example}. She is grateful to RJLO for his help with the Galois group explanation after Proposition \ref{prop:tight_example}. She thanks Zeev Rudnick for his helpful comments on the first version of this paper. TCA is supported by NSF DMS-2231990 and NSF CAREER
DMS-2237937. She thanks Peter Sarnak for introducing a variant of this problem to her.
RJLO is supported by NSF DMS-2200760 and a Simons Fellowship in Mathematics. We are grateful to the anonymous referee for their numerous helpful comments on a previous version of this paper.


\begin{bibdiv}
\begin{biblist}
\bib{anantharaman}{article}{
  author = {N. Anantharaman},
  journal = {Annals of Mathematics},
  pages = {435--475},
  title = {Entropy and the localization of eigenfunctions},
  volume = {168},
  number = {2},
  year = {2008}
}




\bib{AKN}{book}{
  author = {N. Anantharaman},
  author = {H. Koch},
  author = {S. Nonnenmacher},
  journal = {New Trends in Mathematical Physics},
  pages = {1--22},
  title = {Entropy of eigenfunctions},
  year = {2009},
  publisher = {Springer Netherlands},
  editor = {Vladas Sidoravi\v{c}ius}
}

\bib{AN}{article}{
  author = {N. Anantharaman},
  author = {S. Nonnenmacher},
  journal = {Ann. Inst. Fourier},
  pages = {2465--2523},
  title = {Half-delocalization of eigenfunctions for the Laplacian on an Anosov manifold},
  volume = {57},
  number = {7},
  year = {2007}
}



\bib{AS}{article}{
  author = {N. Anantharaman},
  author = {L. Silberman},
  journal = {Israel J. Math.},
  pages = {393--447},
  title = {A Haar component for quantum limits on
locally symmetric spaces.},
  volume = {195},
  number = {1},
  year = {2013},
}

\bib{ADM24}{article}{
  author = {J. Athreya}, 
  author = {S. Dyatlov},
  author = {N. Miller},
  title = {Semiclassical measures for complex hyperbolic quotients},
  year = {2025},
  journal = {Geom. Funct. Anal.},
  pages = {979--1050},
  volume = {35}
}



\bib{Bateman}{book}{
  place={New York}, 
  title={Tables of Integral Transforms}, 
  ISBN={9780070195493},
  publisher={McGraw-Hill Book Company}, 
  author={Bateman, H.},
  author={Bateman Manuscript Project},
  year={1954} 
}


\bib{HB1980}{article}{
  title={Quantization of linear maps-Fresnel diffraction by a periodic grating},
  author={ M.V. Berry},
  author={J.H. Hannay},
  volume={267},
  number={1},
  year={1980},
  publisher={Physica D},
  journal = {Physica D: Nonlinear Phenomena}
}

\bib{BdB}{article}{
  title={Exponential mixing and {$|\log \hbar|$} time scales  in quantized hyperbolic maps on the torus},
  author={F. Bonechi},
  author={S. {De Bi\'{e}vre}},
  journal={Comm. Math. Phys.},
  volume={211},
  number={3},
  pages={659--686},
  year={2000},
  publisher={Springer}
}


\bib{BD}{article}{
  title={Spectral gaps without the pressure condition},
  author={J. Bourgain},
  author={S. Dyatlov},
  volume={187},
  number={3},
  pages={825--867},
  year={2018},
  journal={Ann. of Math. }
}

\bib{BL03}{article}{
  title = {Entropy of Quantum Limits},
  author = {J. Bourgain},
  author = {E. Lindenstrauss},
  journal = {Commun. Math. Phys.},
  volume = {233},
  pages = {153--171},
  year = {2003}
}


\bib{Bouzouina-deBievre}{article}{
  title={Equipartition of the eigenfunctions of quantized ergodic maps on the torus},
  author={A. Bouzouina},
  author={S. {De Bi\`{e}vre}},
  volume={179},
  number={1},
  pages={83--105},
  year={1996},
  journal={Commun. Math. Phys.}
}


\bib{Bo10}{article}{
  author = {S. Brooks},
  title = {On the entropy of quantum limits for 2-dimensional cat maps},
  volume = {293},
  journal = {Commun. Math. Phys.},
  year = {2010}, 
  pages = {231--255}
}

\bib{Chavdarov}{article}{
  author = {Chavdarov, N.},
  journal = {Duke Math. J.},
  pages = {151--180},
  title = {The generic irreducibility of the numerator of the zeta function in a family of curves with large monodromy},
  volume = {87},
  number = {1}, 
  year = {1997}
}

\bib{Cohen}{article}{
  title={Fractal uncertainty in higher dimensions},
  author={A. Cohen},
  journal = {Ann. of Math.},
  year = {2025},
  pages = {267--307},
  volume = {202},
  number = {1}
}



\bib{colindeverdiere}{article}{
  author = {Y. Colin de Verdiere},
  journal = {Chaos and Quantum Physics (Les-Houches)},
  pages = {307--329},
  title = {Hyperbolic geometry in two dimensions and trace formulas},
  year = {1989}
}

\bib{DDS}{article}{
  author = {Davis, S.},
  author = {Duke, W.},
  author = {Sun, X.},
  journal = {Expo. Math.},
  pages = {263--270},
  title = {Probabilistic Galois theory of reciprocal polynomials},
  volume = {16},
  number = {3}, 
  year = {1998}
}


\bib{FNdB}{article}{
  author = {S. De Bi\`{e}vre},
  author = {F. Faure},
  author = {S. Nonnenmacher},
  journal = {Comm. Math. Phys.},
  pages = {449--492},
  title = {Scarred eigenstates for quantum cat maps of minimal periods},
  volume = {239},
  year = {2003}
}

\bib{DixonMortimer}{book}{
  title={Permutation groups},
  volume={163},
  author={Dixon, J. D.},
  author={Mortimer, B.},
  year={1996},
  address={New York},
  pages={xii+346},
  publisher={Springer-Verlag},
  series={Grad. Texts in Math.}
}


\bib{DummitFoote}{book}{
  title={Abstract algebra},
  edition={3},
  author={Dummit, D. S.},
  author={Foote, R. M.},
  year={2004},
  address={Hoboken, NJ},
  pages={xii+932},
  publisher={John Wiley \& Sons Inc.}
}


\bib{Dy}{article}{
  author = {S. Dyatlov},
  journal = {Ann. Math. du Qu\'{e}bec},
  pages = {11–-26},
  volume = {46},
  title = {Around quantum ergodicity},
  year = {2021}
}

\bib{dyatlov2021semiclassical}{article}{
  title={Semiclassical measures for higher dimensional quantum cat maps},
  author={S. Dyatlov},
  author={M. J{\'e}z{\'e}quel},
  journal={Ann. Henri Poincar{\'e}},
  year={2023}
}

\bib{DJi}{article}{
  author = {S. Dyatlov},
  author = {L. Jin},
  journal = {Acta Math.},
  pages = {297--339},
  title = {Semiclassical measures on hyperbolic surfaces have full
support},
  volume = {220},
  number = {2},
  year = {2018}
}

\bib{DJN}{article}{
  author = {S. Dyatlov},
  author = {L. Jin},
  author = {S. Nonnenmacher},
  journal = {J. Amer. Math},
  pages = {361--465},
  title = {Control of eigenfunctions on surfaces of variable curvature},
  volume = {35},
  number = {2},
  year = {2022}
}

\bib{DZ16}{article}{
  author = {S. Dyatlov},
  author = {J. Zahl},
  journal = {Geom. Funct. Anal.},
  pages = {1011--1094 },
  title = {Spectral gaps, additive energy, and a fractal uncertainty principle},
  volume = {26},
  year = {2016}
}

\bib{FN04}{article}{
  author = {F. Faure},
  author = {S. Nonnenmacher},
  journal = {Commun. Math. Phys.},
  year = {2004},
  title = {On the maximal scarring for quantum cat map eigenstates},
  volume = {245},
  pages = {201-214}
}

\bib{FI}{book}{
  title={Opera de Cribro.},
  volume={57},
  author={Friedlander, J.},
  author={Iwaniec, H.},
  year={2010},
  address={Providence, RI},
  pages={xx+527},
  publisher={Amer. Math. Soc.},
  series={Amer. Math. Soc. Colloq. Publ.}
}

\bib{GZ21}{article}{
  title = {Lower bounds for Cauchy data on curves in a negatively curved surface},
  author = {J. Galkowski},
  author = {S. Zelditch},
  journal = {Isr. J. Math.},
  volume = {244},
  pages = {971-1000},
  year = {2021}
}



\bib{GorodnikNevo}{article}{
  author = {Gorodnik, A.},
  author = {Nevo, A.},
  journal = {Bull. Amer. Math. Soc. (N.S.)},
  pages = {65--113},
  title = {Quantitative ergodic theorems and their number-theoretic applications},
  volume = {52},
  number = {1}, 
  year = {2015}
}

\bib{GH}{article}{
  author = {S. Gurevich},
  author = {R. Hadani},
  journal = {Ann. of Math.},
  pages = {1--54},
  title = {Proof of the Kurlberg-Rudnick rate conjecture},
  volume = {174},
  number = {1},
  edition = {2},
  year = {2011}
}


\bib{Ji18}{article}{
  author = {L. Jin},
  journal = {Math. Res. Lett.},
  pages = {1865 -1877},
  title = {Control for Schrodinger equation on hyperbolic surfaces},
  year = {2018},
  volume = {24},
  number = {6},
  label = {Ji18}
}

\bib{Jin}{article}{
  author = {L. Jin},
  journal = {Comm. in Math. Phys.},
  pages = {815--879},
  title = {Damped wave equations on compact hyperbolic surfaces},
  volume = {373},
  number = {3}, 
  year = {2020}
}

\bib{Ke}{article}{
  author = {J.P. Keating},
  journal = {Nonlinearity},
  pages = {277--307},
  title = {Asymptotic properties of the periodic orbits of the cat maps},
  volume = {4},
  year = {1990}
}

\bib{Kel}{article}{
  author = {D. Kelmer},
  journal = {Ann. of Math.},
  pages = {815--879},
  title = {Arithmetic quantum unique ergodicity for symplectic linear maps of the
multidimensional torus},
  volume = {171},
  number = {2},
  edition = {2},
  year = {2010}
}

\bib{Kow08}{book}{
  place={Cambridge},
  author = {E. Kowalski},
  title = {The large sieve and its applications: arithmetic Geometry, random walks and discrete Groups},
  volume = {175},
  DOI={10.1017/CBO9780511542947},
  publisher={Cambridge University Press},
  year={2008},
  series={Cambridge Tracts in Math.}
}

\bib{KORS24}{article}{
 author = {P. Kurlberg},
 author = {A. Ostafe},
 author = {Z. Rudnick},
 author = {I. Shparlinski},
 title = {On quantum ergodicity for higher dimensional cat maps},
 journal = {Comm. in Math. Phys.},
 year = {2025},
 volume = {406},
 number = {174}
}

  


\bib{KR}{article}{
  author = {P. Kurlberg},
  author = {Z. Rudnick},
  journal = {Duke Math. J.},
  pages = {47--77},
  title = {Hecke theory and equidistribution for the quantization of linear maps of the torus},
  volume = {103},
  number = {1},
  year = {2000}
}

\bib{KR2}{article}{
  author = {P. Kurlberg},
  author = {Z. Rudnick},
  journal = {Comm. in Math. Phys.},
  pages = {201--227},
  title = {On quantum ergodicity for linear maps of the torus},
  volume = {222},
  year = {2001}
}

 

\bib{NewmanSmart}{article}{
  author = {Newman, M.},
  author = {Smart, J. R.},
  journal = {Acta Arith.},
  pages = {83--89},
  title = {Symplectic modulary groups},
  volume = {9},
  year = {1964}
}


\bib{PR}{book}{
  title={Algebraic Groups and Number Theory},
  volume={139},
  author={Platonov, V.},
  author={Rapinchuk, A.},
  year={1994},
  series={Pure Appl. Math.}
}

\bib{R1}{article}{
  author = {G. Rivi\`{e}re},
  journal = {Ann. Henri Poincar\'{e}},
  pages = {1085--1116},
  title = {Entropy of semiclassical measures for nonpositively curved surfaces},
  volume = {11},
  number = {6}, 
  year = {2010}
}

\bib{R2}{article}{
  author = {G. Rivi\`{e}re},
  journal = {Duke Math. J},
  pages = {271--336},
  title = {Entropy of semiclassical measures in dimension 2},
  volume = {155},
  number = {2}, 
  year = {2010}
}

\bib{Ri11}{article}{
  author = {G. Rivi\`{e}re},
  year = {2011},
  volume = {2011},
  number = {11},
  pages = {2396--2443},
  title = {Entropy of semiclassical measures for symplectic linear maps of the multidimensional torus},
  journal = {Int. Math. Res. Not.}
}

\bib{RW25}{article}{
    author = {G. Rivi\`{e}re},
    author = {L. L. Wolf},
    title = {On the {L}ebesgue Component of Semiclassical Measures for Abelian Quantum Actions},
    journal = {arXiv:2505.16472},
    year = {2025}
}


\bib{Rivin}{article}{
  author = {I. Rivin},
  journal = {Duke Math. J},
  pages = {353--379},
  title = {Walks on groups, counting reducible matrices, polynomials, and surface and free group automorphisms},
  volume = {142},
  number = {2}, 
  year = {2008}
}

\bib{RS}{article}{
  author = {Z. Rudnick},
  author = {P. Sarnak},
  journal = {Comm. in Math. Phys.},
  pages = {195--213},
  title = {The behaviour of eigenstates of arithmetic hyperbolic
manifolds},
  volume = {161},
  number = {1}, 
  year = {1994}
}


\bib{Sch}{article}{
  author = {N. Schwartz},
  journal = {Pure Appl. Anal.},
  title = {The full delocalization of eigenstates for the quantized cat map},
  year = {2024},
  volume = {2024},
  number = {4},
  pages = {1017--1053}
}



\bib{Shn}{article}{
  author = {A. Shnirelman},
  journal = {Uspehi Mat. Nauk.},
  pages = {181--182},
  title = {Ergodic properties of eigenfunctions},
  volume = {29},
  edition = {6},
  number = {180}, 
  year = {1974}
}

\bib{Ze}{article}{
  author = {S. Zelditch},
  journal = {Duke Math. J.},
  pages = {919--941},
  title = {Uniform distribution of eigenfunctions on compact hyperbolic surfaces},
  volume = {55},
  number = {4}, 
  year = {1987}
}

\bib{z12semiclassical}{book}{
  title={Semiclassical analysis},
  author={M. Zworski},
  volume={138},
  year={2012},
  publisher={Amer. Math. Soc.}
}




\end{biblist}
\end{bibdiv}
\end{document}